\documentclass[11pt]{amsart}
\usepackage{graphicx}
\usepackage{amssymb,amsmath,amsthm}
\usepackage{epstopdf}
\usepackage{mdframed}
\usepackage{empheq}
\usepackage{bm}
\usepackage{hyperref}
\usepackage{mathabx} %
\usepackage{enumitem}
\usepackage[linesnumbered,ruled,vlined]{algorithm2e}
\usepackage{subcaption}

\newtheorem{theorem}{Theorem}

\newtheorem{problem}{Problem}
\newtheorem{lemma}[theorem]{Lemma}
\newtheorem{proposition}[theorem]{Proposition}
\newtheorem{remark}[theorem]{Remark}

\newcommand\bN{\mathbf N}

\newcommand\bW{\mathbf W}
\newcommand\bX{\mathbf X}

\newcommand\bv{\mathbf v}
\newcommand\bw{\mathbf w}

\def\balpha{\boldsymbol{\alpha}}

\def\bvarphi{\boldsymbol{\varphi}}

\newcommand\cC{\mathcal C}

\newcommand\cE{\mathcal E}
\newcommand\cF{\mathcal F}

\newcommand\cL{\mathcal L}

\newcommand\cN{\mathcal N}
\newcommand\cP{\mathcal P}
\newcommand\cU{\mathcal U}

\def\AA{\mathbb{A}}
\newcommand\EE{\mathbb E}

\newcommand\HH{\mathbb H}
\newcommand\JJ{\mathbb J}
\newcommand\PP{\mathbb P}

\newcommand\RR{\mathbb R}
\newcommand\TT{\mathbb T}
\newcommand\VV{\mathbb V}

\newcommand\ff{\mathfrak f}

\makeatletter
\renewcommand\l@paragraph[2]{}
\renewcommand\l@subparagraph[2]{}
\makeatother
\DeclareMathOperator*{\argmin}{arg\,min}
\newcommand\indic{\mathbf{1}}

\title[Machine Learning for Mean Field Optimal Control \& Games]{Convergence Analysis of Machine Learning Algorithms for the Numerical Solution of Mean Field Control and Games: \\ II -- The Finite Horizon Case.}
\author{Ren\'e Carmona \& Mathieu Lauri\`ere}

\newcounter{counterExtraAssumptions}

\begin{document}
\maketitle

\begin{abstract}

We propose two numerical methods for the optimal control of McKean-Vlasov dynamics in finite time horizon. Both methods are based on the introduction of a suitable loss function defined over the parameters of a neural network. This allows the use of machine learning tools, and efficient implementations of stochastic gradient descent in order to perform the optimization.
In the first method, the loss function stems directly from the optimal control problem. The second method tackles a generic forward-backward stochastic differential equation  system (FBSDE) of McKean-Vlasov type, and relies on suitable reformulation as a mean field control problem. 
To provide a guarantee on how our numerical schemes approximate the solution of the original mean field control problem, we introduce a new optimization problem, directly amenable to numerical computation, and for which we rigorously provide an error rate. 
Several numerical examples are provided. Both methods can easily be applied to certain problems with common noise, which is not the case with the existing technology. Furthermore, although the first approach is designed for mean field control problems, the second is more general and can also be applied to the FBSDE arising in the theory of mean field games.
\end{abstract}

\section{Introduction}

The purpose of this paper is to develop numerical schemes for the solution of Mean Field Games (MFGs) and Mean Field Control (MFC) problems. 
The mathematical theory of these problems has attracted a lot of attention in the last decade (see e.g.~\cite{MR2295621,Cardaliaguet-2013-notes,MR3134900,MR3752669,MR3753660}), and from the numerical standpoint several methods have been proposed, see e.g.~\cite{MR2679575,MR2888257,MR3148086,BricenoAriasetalCEMRACS2017,MR3914553} and~\cite{MR3501391,MR3575615,MR3619691,balata2019class} for finite time horizon MFG and MFC respectively, and~\cite{MR3698446,MR3882530,MR3772008} for stationary MFG. 
However, despite recent progress, the numerical analysis of these problems is still lagging behind because of their complexity, in particular when the dimension is high or when the dynamics are affected by a source of common noise.

In~\cite{CarmonaLauriere_DL_periodic}, we have studied the ergodic problems, whose mathematical analysis led to an infinite dimensional optimization problem for which we identified and implemented numerical schemes capable of providing stable numerical solutions. 
Here, we consider mean field control problems in finite time horizon. The thrust of our approach is the numerical solution of Forward-Backward Stochastic Differential Equations (FBSDEs) of the McKean-Vlasov type.  Indeed, the well established probabilistic approach to MFGs and MFC problems posits that the search for Nash equilibria for MFGs, as well as the search for optimal controls for MFC problems, can be reduced to the solutions of FBSDEs of this type. See for example the books \cite{MR3752669,MR3753660} for a comprehensive expos\'e of this approach.

Our mathematical analysis of the model leads to an optimization problem for a large number of agents, for which we can
identify and implement numerical schemes capable of providing stable numerical solutions. 
We prove the theoretical convergence of these approximation schemes and we demonstrate the efficiency of their implementations by comparing their outputs to solutions of benchmark models obtained either by analytical formulas or by deterministic schemes for Partial Differential Equations (PDEs). In particular, we solve numerically two examples with common noise. Although the idea of using machine learning and neural networks for control problems is not new (see e.g.~\cite{MR2137498}), the core of our numerical schemes is based on a generalization to the mean field setting of the deep learning method used very recently e.g. in~\cite{han2016deep-googlecitations,MR3736669} for control problems and partial differential equations. References~\cite{MR2137498} and \cite{han2016deep-googlecitations} tackle optimal control problems by replacing the control by a neural network. In our first method, we extend their idea to the case of mean field interactions and we provide bounds on the approximation error (see Theorem~\ref{thm:main-thm-discreteMKV}). Reference~\cite{MR3736669} approximates the solution of backward stochastic differential equations (BSDEs) by neural network. Our second method generalizes this approach to the case where the BSDE is coupled with a forward equation and furthermore, the system is of mean field type.  Similar ideas have been used simultaneously by Fouque and Zhang in~\cite{FouqueZhang}. Although their problem features an extra difficulty due to the presence of delay,  their work is restricted to a specific linear-quadratic model, and does not provide any theoretical convergence result.

The present paper is structured as follows. In Section~\ref{sec:pb-assumptions}, we introduce useful notations, provide the standing assumptions and state our main result (see Theorem~\ref{thm:main-thm-discreteMKV}), including rates of convergence. In Section~\ref{sec:proof-main}, we outline the key steps for the proof of our main result. In Section~\ref{sec:two-num-meth}, we present two numerical methods for mean field type problems, the first one tackling directly an optimal control of MKV dynamics, and second one dealing with MKV FBSDEs. Last, numerical results are presented in Section~\ref{sec:numres}.

\vskip 2pt\noindent
\emph{Acknowledgements:  Both authors were partially supported by ARO grant AWD1005491 and NSF award
AWD1005433. Also, we would like to thank two anonymous referees for a rigorous review of the original version of the paper. Their insightful comments helped us improve significantly the quality of the paper.}

\section{Formulation of the problem, standing assumptions and main results}
\label{sec:pb-assumptions}

Given a probability measure $\mu_0$ on $\RR^d$, a $d$-dimensional Wiener process $\bW = (W_t)_{t \ge 0}$ and a class $\AA$ of admissible controls taking values in a subset $A$ of $\RR^k$, the MKV control problem we consider can be stated in the following way.

\begin{problem}
\label{pb:MFC1-def}
	Minimize over $\balpha = (\alpha_t)_{t \ge 0} \in \AA$ the quantity
	\begin{equation}
	\label{eq:MFC1-cost}
		J(\balpha) 
		= 
		\EE\left[\int_0^T f(t, X_t, \cL(X_t), \alpha_t)dt + g(X_T,\cL(X_T))\right] \, ,
	\end{equation}
	under the constraint that the process $\bX = (X_t)_{t \ge 0}$ solves the Stochastic Differential Equation (SDE)
	\begin{equation}
	\label{eq:MFC1-dyn}
		dX_t = b(t, X_t, \cL(X_t), \alpha_t) dt + \sigma(t, X_t, \cL(X_t)) dW_t \, , t \ge 0,
		\quad
		X_0 \sim \mu_0,
	\end{equation}
\end{problem}
\noindent
where we use the notation $\cL(V)$ for the law of a random variable $V$.
For the sake of definiteness we choose $\AA = \HH^{2,k}$, the set of $\RR^k$-valued progressively-measurable square-integrable processes defined as:
$$
	\HH^{2,k} = \left\{ Z \in \HH^{0,k} ;\; \EE \int_0^T |Z_s|^2 ds < \infty \right\},
$$
where $\HH^{0,k}$ denotes the set of $\RR^k$-valued  progressively measurable processes on $[0,T]$.
These control processes are often called \emph{open loop}.

When it is helpful to stress which control is being used, we shall denote by $\bX^{\balpha} = (X^{\balpha}_t)_{t \ge 0}$ the solution to~\eqref{eq:MFC1-dyn} when the control $\balpha$ is used.
In equation \eqref{eq:MFC1-dyn}, the drift and volatility coefficients $b$ and $\sigma$ are functions on $[0,T]\times \RR^d\times\cP_2(\RR^d)\times A$ with values in $\RR^d$ and $\RR^{d\times d}$ respectively. We shall assume that the drift and volatility functions $b$ and $\sigma$ are regular enough so that for each admissible control $\balpha$, existence and uniqueness of a process $\bX^{\balpha}$ satisfying \eqref{eq:MFC1-dyn} hold.  Specific assumptions under which this is indeed the case are given in Appendix \ref{ap:assumptions}. We explain below how to understand the notion of regularity with respect to the measure argument $\mu$.

\subsection{Definitions, notations and background}

For any $p\geq 1$, $\cP_p(\RR^d)$ is the set of probability measures of order $p$ on $\RR^d$, namely those probability measures on $\RR^d$ for which the $p$-th moment $M_p(\mu)$ defined by
\begin{equation}
\label{fo:M_p_mu}
	M_p(\mu) = \left( \int_{\RR^d} |x|^p d \mu(x) \right)^{1/p}
\end{equation}
is finite. $W_p$ is the $p$-Wasserstein distance defined, for $\mu,\mu' \in \cP_p(\RR^d)$ by
$$
	W_p(\mu, \mu') = \inf_{\pi \in \Pi(\mu,\mu')} \left[\int_{\RR^d \times \RR^d} |x-y|^p \pi(dx,dy)\right]^{1/p}
$$
where $\Pi(\mu,\mu')$ denotes the set of probability measures on $\RR^d\times \RR^d$ with marginals $\mu$ and $\mu'$. 
For an integer $r \geq 0$ and a domain $D \subseteq \RR^d$, we denote by $\cC^r(D; \RR^{d'})$ the set of functions on $D$ taking values in $\RR^{d'}$ which are continuously differentiable up to order $r$ (included). It is endowed with the usual norm: for $r=0$, this is the sup norm and for $r > 0$, it is the sum of the sup norms of the derivatives up to order $r$. For $K>0$, we will denote by $\cC^0_K(D; \RR^{d'})$ the subset of continuous functions with (sup) norm bounded by $K$. 
For $L>0$, we denote by $Lip_L(D; \RR^{d'})$ the set of Lipschitz functions on $D$ with Lipschitz constant at most $L$. 
When we consider real valued functions, i.e. $d'=1$, we shall write simply $\cC^r(D), \cC^0_K(D)$ and $Lip_L(D)$.

\subsection{Standing assumptions}
\label{sub:assumptions}
In this subsection we introduce the assumptions under which existence of optimal controls holds, and under which we prove convergence of the numerical algorithms we propose to compute the solutions. Some of the assumptions are stated at a high level, and the reader may wonder for which classes of coefficients  these assumptions are satisfied. In Appendix \ref{ap:assumptions}, we give low level explicit conditions  under which all of our assumptions hold. Roughly speaking, 
for the state dynamics and the cost functions, we use slight variations on assumption ``\textnormal{\textbf{Control of MKV Dynamics}}'' from~\cite[p. 555]{MR3752669}. 
Throughout the paper, we assume that the volatility is not controlled and that the initial distribution $\mu_0$ is in $\cP_4(\RR^d)$. 
Still, in order to justify and quantify the approximation of the optimal control by a neural network, we shall need extra regularity assumptions. To formulate them, we introduce more notations. 

The Hamiltonian of the system is the function $H$ defined by
\begin{equation}
\label{eq:def-Hamiltonian}
	H(t, x, \mu, y, z, \alpha)
	= \,  b(t, x, \mu, \alpha) \cdot y
	+
	\sigma(t, x, \mu) \cdot z
	+
	f(t, x, \mu, \alpha)
\end{equation}
for $t \in [0,T]$, $x,y \in \RR^d$, $z \in \RR^{d \times d}$, $\mu \in \cP_2(\RR^d)$ and $\alpha \in A$, and since the volatility $\sigma$ does not depend on the control variable in our setting, we will also make use of the notation $\tilde H$ for the reduced Hamiltonian defined by
\begin{equation}
\label{eq:def-Hamiltonian-red}
	\tilde H(t, x, \mu, y, \alpha)
	= \, 
	 b(t, x, \mu, \alpha) \cdot y
	+
	f(t, x, \mu, \alpha).
\end{equation}

We shall assume that given any $(t, x, \mu, y) \in [0,T] \times \RR^d \times \cP_2(\RR^d) \times \RR^d$, the function $A \ni \alpha \mapsto \tilde H(t, x, \mu, y, \alpha)$ has a unique minimizer $\hat{\alpha}(t, x, \mu, y)$:
\begin{equation}
\label{eq:def-hat-alpha}
	\hat{\alpha}(t, x, \mu, y) = \argmin_{\alpha \in A} \tilde H(t, x, \mu, y, \alpha)
\end{equation}
being (jointly) Lipschitz in all its variables. We shall also assume that the coefficients $b$ and $\sigma$, as well as the cost functions $f$ and $g$, are differentiable with respect to the variables $x$ and $\mu$. The partial derivatives with respect to the argument $\mu$ have to be understood in the Wasserstein sense, or in the Lions sense (L-derivatives). See \cite[Chapter 5]{MR3752669} for details. 
The forward-backward system of SDEs (FBSDE for short) associated to the control problem is (see e.g.~\cite[Section 6.4.2]{MR3752669})
\begin{equation}
\label{eq:MKV-FBSDE}
\left\{
\begin{aligned}
	d X_t
	=
	& b\bigl(t,X_t,\cL(X_t),\hat{\alpha}(t, X_t, \cL(X_t), Y_t)\bigr) dt 
	\\
	& \,\,+\sigma\bigl(t,X_t,\cL(X_t),\hat{\alpha}(t, X_t, \cL(X_t), Y_t)\bigr) dW_t
	\\
	d Y_t
	=
	& - \partial_x H(t, X_t, \cL(X_t), Y_t,Z_t,\hat{\alpha}(t, X_t, \cL(X_t), Y_t))  dt
	\\
	& \,\, - \tilde{\EE}\left[\partial_\mu H(t, \tilde{X}_t, \cL(X_t), \tilde Y_t, \tilde Z_t, \hat{\alpha}(t, \tilde{X}_t, \cL(X_t), \tilde{Y}_t))(X_t)\right] dt + Z_t d W_t,
\end{aligned}
\right.
\end{equation}
with initial condition $X_0 = \xi \in L^4(\Omega, \cF_0, \PP; \RR^d)$ and terminal condition $Y_T = \partial_x g(X_T, \cL(X_T)) + \tilde{\EE}\left[\partial_\mu g(\tilde{X}_T, \cL(X_T))(X_T)\right]$. The tilde $\tilde{}$ over random variables means that these random variables are copies in the sense that they have the same distributions as the original variables, but are possibly defined on a different probability space over which the expectation is denoted by $\tilde \EE$.

We shall assume that the system \eqref{eq:MKV-FBSDE} is uniquely solvable and that there exists a function $\cU$, called master field of the FBSDE~\eqref{eq:MKV-FBSDE}, such that the process $(Y_t)_{t \in [0,T]}$ can be represented as 
\begin{equation}
\label{eq:link-Y-cU}
	Y_t = \cU(t, X_t, \mu_t), \qquad t \in [0,T],
\end{equation}
where $\mu_t = \cL(X_t)$. See for example \cite[Theorem 6.19 p.559]{MR3752669} for an existence and uniqueness result for the FBSDE \eqref{eq:MKV-FBSDE}, and \cite[Lemma 6.25 p.563]{MR3752669} for existence of the master field. 

\begin{remark}
Note that if one merely assumes the existence of a (possibly non-unique) optimal
control, one can inject this optimal control in the forward dynamics of the state and the adjoint equation associated to this control to obtain an FBSDE analog to \eqref{eq:MKV-FBSDE}.
\end{remark}

Next, we introduce the decoupling field
\begin{equation}
\label{eq:def-decouplingfield}
	V(t, x) = \cU(t, x, \mu_t),
\end{equation}
which we assume to be jointly Lipschitz in its variables, and to be twice differentiable with respect to $x$ with $\partial_{xx} V$ being Lipschitz in its variables. The optimal control $\hat{\balpha} = (\hat{\alpha}_t)_{t \in [0,T]}$ can be rewritten in the feedback form
\begin{equation}
\label{eq:hat-alpha-feedbackform}
	\hat{\alpha}_t = \hat{\alpha}(t, X_t, \mu_t, V(t, X_t)).
\end{equation}
Note that $\hat{\alpha}_t$ appearing in the left hand side is a random variable obtained by computing the deterministic function $\hat\alpha$ defined in \eqref{eq:def-hat-alpha} for random arguments depending upon $X_t$.
For this reason, the use of the same letter $\alpha$ for both objects should not be a source of confusion.
\vskip 2pt\noindent
For the sake of the analysis of the time discretization that we will use in our numerical scheme, we shall also assume that, at each $t \in [0,T]$, the feedback function $\hat v(t,x)=\hat{\alpha}(t, x, \mu_t, V(t, x))$ is twice differentiable with second order derivatives which are Lipschitz continuous. This function will play a crucial role in Proposition \ref{prop:approx-feedback-N} below.
We refer the reader to Appendix \ref{ap:assumptions} where we articulate explicitly a set of specific assumptions under which all the properties stated above hold.
Unless otherwise specified, the constants depend only on the data of the problem ($T$, $\mu_0$, $d$, $k$, and the constants appearing in the assumptions), and $C$ denotes a generic constant whose value may change from one line to the next.

\subsection{Approximation results}
\label{sec:approx-result}

Next, we propose a new optimization problem, amenable to numerical computations (see Section~\ref{sec:two-num-meth}), which serves as a proxy for the original MKV control problem, and for which we quantify the approximation error. The rationale behind this new problem is encapsulated in the following three steps:
\begin{itemize}
	\item the distribution $\cL(X_t)$ of the state $X_t$ is approximated by the empirical distribution of $N$ agents;
	\item the set  $\AA$ of controls is approximated by a set of controls in feedback form, the feedback function being given by a neural net with a fixed architecture; 
	\item the time variable is discretized.
\end{itemize}
Before defining the problem, we first introduce notations pertaining to neural networks.

\subsection{Neural networks}
We denote by:
\begin{align*}
	\mathbf{L}^\psi_{d_1, d_2} 
	= \Big\{ \phi: \RR^{d_1} \to \RR^{d_2} \,\Big|\,  &\exists \beta \in \RR^{d_2}, \exists w \in \RR^{d_2 \times d_1}, \forall  i \in \{1,\dots,d_2\}, \;
	\\
	&\qquad \phi(x)_i = \psi\Big(\beta_i + \sum_{j=1}^{d_1} w_{i,j} x_j\Big) \Big\} 
\end{align*}
the set of layer functions with input dimension $d_1$, output dimension $d_2$, and activation function $\psi: \RR \to \RR$. For the sake of definiteness, we shall assume that the activation function 
$\psi:\RR \to \RR$ is a $2\pi-$periodic function of class $\cC^3$, that is three times continuously differentiable, satisfying:
$
	\hat\psi_1:=\int_{-\pi}^\pi \psi(x) e^{-i  x} dx \neq 0.
$
More general activation functions could be accommodated at the expense of additional technicalities. The choice of this class of activation functions is motivated by the fact that we want to find a neural network which approximates, on a compact set, the optimal feedback control while being Lipschitz continuous and whose Lipschitz constant can be related to the one of the optimal control.  
Building on this notation we define:
\begin{equation}
\label{fo:N_psi}
\begin{split}
	\bN^\psi_{d_0, \dots, d_{\ell+1}} 
	= 
	\Big\{ \varphi: \RR^{d_0} \to \RR^{d_{\ell+1}} \,\Big|\, 
	&\forall i \in \{0, \dots, \ell-1\}, \exists \phi_i \in \mathbf{L}^\psi_{d_i, d_{i+1}}, 
	\\
	&\exists \phi_\ell \in \mathbf{L}_{d_{\ell}, d_{\ell+1}},  \varphi = \phi_\ell \circ \phi_{\ell-1} \circ \dots \circ \phi_0  \Big\} \, 
\end{split}
\end{equation}
 the set of regression neural networks with $\ell$ hidden layers and one output layer, the activation function of the output layer being the identity $\psi(x)=x$. We shall not use the superscript $\psi$ when the activation function is the identity. As explained earlier, we assume that $A=\RR^k$, but if we wanted to accommodate control constraints and assume for example that the set $A$ of actions is a compact subset of $\RR^k$, one would have to use a bounded activation function whose range matches $A$. The number $\ell$ of hidden layers, the numbers $d_0$, $d_1$, $\cdots$ , $d_{\ell+1}$ of units per layer, and the activation functions (one single function $\psi$ in the present situation), are what is usually called the architecture of the network. Once it is fixed, the actual network function $\varphi\in \bN^\psi_{d_0, \dots, d_{\ell+1}} $ is determined by the remaining parameters:
$ \theta=(\beta^{(0)}, w^{(0)},\beta^{(1)}, w^{(1)},\cdots,\beta^{(\ell)}, w^{(\ell)})$
defining the functions $\phi_0$, $\phi_1$, $\cdots$ , $\phi_{\ell-1}$ and $\phi_\ell$ respectively. Their set is denoted by $\Theta$. For each $\theta\in\Theta$, the function $\varphi$ computed by the network will sometimes be denoted by $\varphi_\theta$.
We will work mostly with the  case $d_0 = d+1$ and $d_{\ell+1} = k$, dimensions of $(t,x)$ and of the control variable respectively.

As implied by the above discussion, the search for optimal controls in the general class $\AA$ of open loop controls eventually leads to controls in feedback form given by \eqref{eq:hat-alpha-feedbackform}. This fact is the rationale for the second step announced earlier, namely the search for approximately optimal controls among the controls given in feedback form by neural network functions $\varphi$. Notice that if $\balpha$ is such a control given in the form  $\alpha_t = \varphi(t, X^\alpha_t)$ for some $ \varphi \in \bN^\psi_{d_0, \dots, d_{\ell+1}}$, then this $\balpha$ is admissible, i.e. $\balpha\in\AA$, because of standard properties of solutions of MKV stochastic differential equations with Lipschitz coefficients, and the fact that the function $\varphi$ is at most of linear growth. 
See for example \cite{MR1108185}. Last, we stress that the elements of $\bN^\psi_{d_0, \dots, d_{\ell+1}}$ have the same regularity as the activation function $\psi$, namely $\cC^3$.

\begin{remark}
Some models include a control constraint forcing the space $A$ in which controls can be chosen to be a closed convex strict subset of $\RR^k$. Our approach can easily accommodate such a constraint. We only need to adjust the last activation function, not take the identity as suggested above, but instead, to choose for the output layer an activation $\psi$ which is bounded and takes values in $A$.

\end{remark}

\subsection{New optimization problem and main result}
For the discretization of the time interval, we will use a grid $t_0=0 < t_1 < \dots < t_{N_T} = T$ where $N_T$ is a positive integer. For simplicity we consider a uniform grid, that is, $t_n = n \Delta t$, with $\Delta t = T/N_T$.

The three approximation steps described at the beginning of this section lead to the following minimization problem. The latter can be viewed as the problem of a central planner trying to minimize the social cost of $N$ agents using a decentralized control rule in feedback form given by a neural network.

\begin{problem}\label{pb:discrete-MKV}
	Minimize the quantity
	\begin{equation}
	\label{eq:def-checkJN}
		\check J^N(\bvarphi) 
		= 
		\EE\left[\frac{1}{N}\sum_{i=1}^N \Bigl(\Delta t\sum_{n=0}^{N_T-1} f\left(t_n, \check X^i_{t_n}, \check\mu_{t_n}, \varphi(t_n, \check X^i_{t_n})\right)+ g(\check X^i_{t_{N_T}},\check\mu_{t_{N_T}})\Bigr)\right] \, ,
	\end{equation}
over $\bvarphi \in \bN^\psi_{d+1, d_2, \dots, d_{\ell+1}, k}$,  under the dynamic constraint: 
	\begin{equation}
	\label{eq:discrete-dyn}
		\check X^i_{t_{n+1}} = \check X^i_{t_n} + b\left(t_n, \check X^i_{t_n}, \check\mu_{t_n}, \varphi(t_n, \check X^i_{t_n})\right) \Delta t + \sigma\left(t_n, \check X^i_{t_n}, \check\mu_{t_n}\right) \Delta \check W^i_n \, , 	\end{equation}
	for all $n \in \{0,\dots, N_T-1\},$ $i \in \{1,\dots,N\}$,
where the $(\check X^i_0)_{i \in \{1, \dots, N\}}$ are i.i.d. with common distribution $\mu_0$, $\check\mu_{t_n} = \frac{1}{N} \sum_{i=1}^N \delta_{\check X^i_{t_n}}$, and the $(\Delta \check W_n^i)_{i,n}$ are i.i.d. random variables with distribution $\mathcal N(0, \Delta t)$.
\end{problem}

We show that solving Problem~\ref{pb:discrete-MKV} provides an approximate solution to Problem~\ref{pb:MFC1-def}, and we quantify the accuracy of the approximation. The main theoretical result of the paper is the following.
\begin{theorem}
\label{thm:main-thm-discreteMKV}
We have
	$$
		\inf_{\balpha \in \AA} J(\balpha)
		\ge 
		\inf_{\bvarphi \in \bN^{\psi}_{d+1, n_{\mathrm{in}}, k}} \check J^N(\bvarphi) - \epsilon(N, n_{\mathrm{in}}, \Delta t) \, ,
	$$
	where
		$
		\epsilon(N,n_{\mathrm{in}},\Delta t) 
		= \epsilon_1(N) + \epsilon_2(n_{\mathrm{in}}) + \epsilon_3(\Delta t),
		$ 
	with for some $\epsilon_1,\epsilon_2,\epsilon_3$ satisfying:
	$$\epsilon_1(N) \in O\left( N^{-1/\max(d, 4)} \sqrt{1 + \ln (N) \mathbf{1}_{\{d=4\}}} \right), \quad 
		\epsilon_2(n_{\mathrm{in}}) \in O\left( n_{\mathrm{in}}^{-\frac{1}{3(d+1)}} \right), \hbox{ and } 
		\epsilon_3(\Delta t) \in O\left( \sqrt{\Delta t} \right),$$ 
	the constants in the big Bachmann - Landau terms $O(\cdot)$ depending only on the data of the problem and on the activation function $\psi$ through $\hat\psi_1$, $\|\psi'\|_{\cC^0(\TT)}$, $\|\psi''\|_{\cC^0(\TT)}$  and $\|\psi'''\|_{\cC^0(\TT)}$.
\end{theorem}
The proof is provided in Section~\ref{sec:proof-main}.

\begin{remark}
	The dependence of the error on the the number $n_{\mathrm{in}}$ of units in the layer could certainly be improved by looking at multilayer neural networks. Unfortunately, quantifying the performance of these architectures is much more difficult and we could not find in the function approximation literature suitable rate of convergence results which could be useful in our setting. This is due to the fact that we need to approximate simultaneously a function and its derivative. Also, the error from the Euler scheme could probably be improved to order $1$ instead of $1/2$ (i.e., the last term $\epsilon_3(\Delta t)$ should be of the order $\Delta t$ instead of $\sqrt{\Delta t}$) since the result of Theorem~\ref{thm:main-thm-discreteMKV} is expected to require only a weak order. However, we are only able to reach this order at the expense of extra regularity properties of the decoupling field. Since we did not want to add (and prove) extra smoothness assumptions on the decoupling field, we only claim order $1/2$ in the mesh of the time discretization.
\end{remark}

\begin{remark}
	Theorem~\ref{thm:main-thm-discreteMKV} provides a bound on the approximation error. In the numerical implementation, the expectation in $\check J^N$ over the $N-$agent population is replaced by an empirical average over a finite number of populations, which leads to an estimation (or ``generalization'') error. 
A bound on this type of error is provided in Theorem~8 of our companion paper~\cite{CarmonaLauriere_DL_periodic} using the regularity of the cost function. Note that the problem analyzed in~\cite{CarmonaLauriere_DL_periodic} is also an optimal control problem (obtained after a suitable reformulation of the original problem) in which we replace the optimal control by a neural network. Due to this kinship, we believe that a similar analysis could be carried out for the time dependent problem tackled in the present work, but due to space constraints we leave this for future work.  
However, strictly speaking, on the top of the two types of error discussed in the text, we should also consider a third source of error coming from the use of stochastic gradient descent to implement algorithmically the approximation of the population ensemble by an empirical average. This could be approached using recent results like those described of~\cite{MR3797719}.

\begin{remark}
	Using similar techniques to those we employ in the proof of Theorem~\ref{thm:main-thm-discreteMKV}, one can show how the optimal control of the approximating problem performs once injected in the original cost functional. In other words, assuming there exists $\hat \bvarphi$ such that  $\check J^N(\hat \bvarphi)  = \inf_{\bvarphi \in \bN^{\psi}_{d+1, n_{\mathrm{in}}, k}} \check J^N(\bvarphi) $, we can derive an upper bound for the difference $|\check J^N(\hat \bvarphi) - J(\hat\balpha)|$. First, Proposition~\ref{prop:approx-discreteT} provides a bound for $|J^N(\hat\bvarphi) -  \check J^N(\hat\bvarphi)|$ in terms of $\Delta t$. Then, $|J^N(\hat\bvarphi) - J(\hat\bvarphi)|$ can be bounded using propagation of chaos type results. This step is in fact already present in the proof of~\cite[Theorem 6.17]{MR3753660}, which is the cornerstone of our proof of Proposition~\ref{prop:approx-feedback-N}.
\end{remark}

\end{remark}

\section{Proof of the main result}
\label{sec:proof-main}

We split the proof into three steps  presented in separate subsections, each one consisting in the control of the approximation error associated to one of the steps described in the bullet points at the beginning of the previous section. As explained above, the first step consists in approximating the mean field problem with open loop controls by a problem with $N$-agents using \emph{distributed} closed-loop controls (see subsection~\ref{SEC:APPROX-NAGENTS-FEEDBACK}); in the second step, we replace general closed-loop controls by the subclass of closed-loop controls that can be represented by neural networks (see subsection~\ref{SEC:APPROX-NN}); in the third step, we discretize time (see subsection~\ref{SEC:APPROX-DISCRETETIME}).

\subsection{Problem with $N$ agents and closed-loop controls}
\label{SEC:APPROX-NAGENTS-FEEDBACK}
Because of
 \eqref{eq:hat-alpha-feedbackform}, we expect the optimal control to be in feedback form. So in
the sequel, we restrict our attention to closed-loop controls that are deterministic functions of $t$ and $X_t$, i.e. of the form $\alpha_t = v(t, X_t)$ for some deterministic function $v : [0,T] \times \RR^d \to A\subset\RR^k$. We denote by $\VV$ the class of admissible feedback functions, i.e., the set of measurable functions $v : [0,T] \times \RR^d \to A$ such that the control process $\balpha$ defined by $\alpha_t = v(t, X^\alpha_t)$ for all $t$ is admissible, i.e. $\balpha \in \AA$. Note that if the feedback function $v$ is Lipschitz in $x$ uniformly in $t$, then the state SDE \eqref{eq:MFC1-dyn} is well posed, and since $v$ is at most of linear growth, standard stability estimates for solutions of Lipschitz SDE of McKean-Vlasov type guarantee that $\balpha\in\AA$ or equivalently $\bv\in\VV$.

\vskip 2pt
According to the first step described above, we recast Problem~\ref{pb:MFC1-def} as the limiting problem for the optimal control of a large number of agents by a \emph{central planner} as the number of agents tends to infinity. In the model with $N<\infty$ agents, the goal is to minimize the average (``social'') cost, when all the agents are using the same control rule, namely the same feedback function of their individual states. The resulting optimization problem is:

\begin{problem}
\label{pb:MKV-Nagents}
	Minimize the quantity 
	\begin{equation}
	\label{eq:def-JN-closedloop}
		J^N(\bv) 
		= 
		\frac{1}{N}\sum_{i=1}^N \EE\left[\int_0^T f(t, X^i_t, \mu^N_t, v(t, X^i_t))dt + g(X^i_T,\mu^N_T)\right] \, ,
	\end{equation}
over $\bv = (v(t, \cdot))_{0\le t \le T} \in \VV$ under the constraint: 
	\begin{equation}
	\label{eq:evolXi-closedloop}
		dX^i_t = b(t, X^i_t, \mu^N_t, v(t,X^i_t)) dt + \sigma(t, X^i_t, \mu^N_t) dW^i_t \, , \quad t \ge 0, \, i \in \{1,\dots,N\} \, ,
	\end{equation}
where the $\bW^i=(\bW^i)_{i=1, \dots, N}$ are independent $d$-dimensional Wiener processes, the $(X^i_0)_{i \in \{1, \dots, N\}}$ are i.i.d. with common distribution $\mu_0$ and are independent of the Wiener processes, and where $\mu^N_t = \frac{1}{N} \sum_{i=1}^N \delta_{X^i_t}$ is the empirical distribution of the population of the $N$ agents at time $t$.
\end{problem}

The following result shows how well this new optimization problem approaches the original one.

\begin{proposition}
\label{prop:approx-feedback-N}
If we define the feedback function $\hat\bv$ by: 
\begin{equation}
\label{fo:vhat}
		\hat v(t,x) = \hat \alpha\left( t, x, \mu_t, V(t, x) \right),
\end{equation}
then
	$$
		\inf_{\balpha \in \AA} J(\balpha)
		\ge 
		J^N(\hat \bv) - \epsilon_1(N) \, ,
	$$
	with 
	$
		\epsilon_1(N) = c_1 \sqrt{N^{-2/\max(d, 4)} \left(1 + \ln (N) \mathbf{1}_{\{d=4\}}\right)}  \, ,
	$
	for some constant $c_1$ depending only on the data of the problem.
\end{proposition}
Recall that $\hat\alpha$ is the minimizer of the Hamiltonian defined by~\eqref{eq:def-hat-alpha},  $\mu_t$ and $V$ are respectively the marginal distribution of $X_t$ of the associated MKV FBSDE system~\eqref{eq:MKV-FBSDE} and the decoupling field defined by~\eqref{eq:def-decouplingfield}, and as explained above $\hat\bv$ is admissible under the standing assumptions. 
The result follows from a slight modification of the proof of~\cite[Theorem 6.17]{MR3753660}. 
The proof is deferred to Appendix~\ref{sec:proof-sec-approx-Nagents-feedback}.

\subsection{Problem with neural networks as controls}
\label{SEC:APPROX-NN}

Next, we show that the feedback function $\hat\bv$ used in Proposition~\ref{prop:approx-feedback-N}, can be approximated by a neural network in such a way that minimizing over neural networks ends up being not much worse than minimizing over feedback controls, and we quantify the loss due to this approximation.

The main result of this section is the following. 

\begin{proposition}
\label{prop:approx-NN-cmpJ-epsilon}
There exist two positive constants $K_1$ and  $K_2$ depending on the data of the problem and on $\psi$ through $\hat\psi(1)$, $\|\psi'\|_{\cC^0(\TT)}$, $\|\psi''\|_{\cC^0(\TT)}$  and $\|\psi'''\|_{\cC^0(\TT)}$, 
such that for each integer $n_{\mathrm{in}} \ge 1$, there exists $\hat\bvarphi \in \bN^{\psi}_{d+1,n_{\mathrm{in}},k}$  such that the Lipschitz constants of $\hat\bvarphi$, $\partial_x\hat\bvarphi$ and $\partial^2_{x,x}\hat\bvarphi$ are bounded by $K_1$, and which satisfies
	$$
		J^N(\hat \bv)
		\ge 
		J^N(\hat\bvarphi) -  K_2 n_{\mathrm{in}}^{- \frac{1}{3(d+1)}}.
	$$
\end{proposition}
The proof is based on two technical results proven in Proposition~\ref{prop:approx-fct-NN} and Proposition~\ref{lem:var-J-feedback-compact} below. Given the Lipschitz continuity of the optimal feedback control $\hat{\bv}$ w.r.t. $(t,x)$ under the standing assumptions,
and given the result of Proposition~\ref{prop:approx-fct-NN} below quantifying the rate at which one can approximate Lipschitz functions by neural networks over compact sets, the proof of the above result relies on the regularity property of the objective functional provided by Proposition~\ref{lem:var-J-feedback-compact}: if two controls are close enough, the objective value does not vary too much.

The proof of Proposition~\ref{prop:approx-fct-NN} relies on a special case of~\cite[Theorems 2.3 and 6.1]{MhaskarMicchelli} which we state below for the sake of completeness. It forces us to work with periodic functions, but in return, it provides a neural network approximating simultaneously a function and its derivative. This is especially important in an optimal control setting as it provides needed bounds on the Lipschitz constant of the control. We use this approximation estimate locally so we do need to assumes periodicity in the statement of our result.
First we recall a standard notation. For a positive integer $m$, $\TT^m$ denotes the $2\pi-$torus in dimension $m$.
For positive integers $n$ and $m$, and a function $g \in \cC^0(\TT^m) $, $E_n^m(g)$ denotes the trigonometric degree of approximation of $g$ defined by
$
	E_n^m(g) = \inf_{P} \|g - P\|_{\cC^0(\TT^m)},
$
where the infimum is over trigonometric polynomials of degree at most $n$ in each of its $m$ variables. While \cite[Theorems 2.3 and 6.1]{MhaskarMicchelli} do not differentiate between time and space variables, and the same regularity is assumed for all the components of the variables, an inspection of the proofs of these theorems actually shows that the result stated as Theorem~\ref{th:MM} in Appendix \ref{app:proof-approx-NN} holds.

\vskip 4pt
In the sequel, we denote by $\cC(R, K, L_1, L_2, L_3)$ the class of functions $\ff: [0,T] \times \bar{B}_d(0,R) \ni (t,x) \mapsto \ff(t,x) \in \RR^{k}$ such that $\ff$ is Lipschitz continuous in $(t,x)$ with $\cC^0([0,T] \times \bar{B}_d(0,R))-$norm bounded by $K$ and Lipschitz constant bounded by $L_1$, and $\ff$ is twice differentiable w.r.t. $x$ such that for every $i=1,\dots,d$, $\partial_{x_i} \ff$ is Lipschitz continuous w.r.t. $(t,x)$ with Lipschitz constant bounded by $L_2$  and for every $i,j=1,\dots,d$, $\partial_{x_i, x_j} \ff$ is Lipschitz continuous w.r.t. $(t,x)$ with Lipschitz constant bounded by $L_3$. Note that this class of functions implicitly depends on $d,k,T$, which are part of the \emph{data of the problem}. We will sometimes use the notations $\nabla \ff = (\partial_{x_1} \ff, \dots, \partial_{x_d} \ff)$ and $\nabla_{(t,x)} \ff = (\partial_t \ff, \partial_{x_1} \ff, \dots, \partial_{x_d} \ff)$.

\vskip 2pt

We first recall a useful result from approximation theory which we shall use repeatedly.

\begin {theorem}[Theorems 2.3 and 6.1 in~\cite{MhaskarMicchelli}]
\label{th:MM}
	There exists an absolute constant $C_{mm}$ such that for every positive integers $d, n$ and $N$, there exists a positive integer $n_{\mathrm{in}} \le C_{mm} N n^d$ with the following property.
	 For any $\ff:\TT\times \TT^{d} \to \RR$ of class $\cC^{1,3}(\TT\times\TT^d)$, that is continuously differentiable in the time variable $t\in\TT$ and three times continuously differentiable in the space variable $x\in\TT^d$, there exists $\varphi_\ff \in \bN^\psi_{d+1, n_{\mathrm{in}}, 1}$ such that:
	\begin{align}
		\label{eq:MhaskarMicchelli-bdd-f}
		\|\ff - \varphi_\ff\|_{\cC^0(\TT^{d+1})} &\leq c\left[E^{d+1}_n(\ff) + E^1_N( \psi) n^{(d+1)/2} \|\ff\|_{\cC^0(\TT^{d+1})} \right],
		\\
		\label{eq:MhaskarMicchelli-bdd-grad-f}
		\|\partial_i \ff - \partial_i \varphi_\ff\|_{\cC^0(\TT^{d+1})} &\leq c\left[E^{d+1}_n(\partial_i \ff) + E^1_N(\psi') n^{(d+1)/2} \|\partial_i \ff\|_{\cC^0(\TT^{d+1})} \right], 
		\\
		\label{eq:MhaskarMicchelli-bdd-grad2-f}
		\|\partial_{i,j} \ff - \partial_{i,j} \varphi_\ff\|_{\cC^0(\TT^{d+1})}
		 &
		\leq c\left[E^{d+1}_n(\partial_{i,j} \ff) + E^1_N(\psi'') n^{(d+1)/2} \|\partial_{i,j} \ff\|_{\cC^0(\TT^{d+1})} \right], 
	\end{align}
	for all $i,j = 2,\dots,d+1$, 
	where $\partial_i$ denotes the partial derivative with respect to the $i^{th}$ variable and the constant $c$ depends only on $d$ and $\hat\psi(1)$, and where the degree of trigonometric approximation $E^m_n(\cdot)$ was defined in the text.
\end{theorem}

\vskip 2pt
The workhorse of our control of the approximation error is the following. 

\begin{proposition}
\label{prop:approx-fct-NN} 
For every real numbers $K>0$, $L_1>0, L_2>0$  and $L_3>0$, there exists a constant $C$ depending only on the above constants, on $d,k,T,$ and on the activation function through $\hat\psi(1)$, $\|\psi'\|_{\cC^0}$, $\|\psi''\|_{\cC^0}$,  and $\|\psi'''\|_{\cC^0}$, and there exists a constant $n_0$ depending only on $C_{mm}$ and $d$ with the following property. For every $R>0$, for every $\ff \in \cC(R, K, L_1, L_2, L_3)$ and for every integer $n_{\mathrm{in}}>n_0$, there exists a one-hidden layer neural network $\varphi_{\ff} \in \bN^{\psi}_{{d + 1}, n_{\mathrm{in}}, {k}}$ such that
	$$
		\|\ff - \varphi_{\ff}\|_{\cC^0([0,T] \times \bar{B}_d(0,R); \RR^{k})} \le C (1+R)n_{\mathrm{in}}^{-1/(2 (d+1))}, 
	$$
and such that the Lipschitz constants of $\varphi_{\ff}$, $\partial_x \varphi_{\ff}$ and $\partial^2_{x,x} \varphi_{\ff}$ are at most $C (1+Rn_{\mathrm{in}}^{-1/(2 (d+1))})$.
\end{proposition}
We stress that the constant $C$ in the above statement does not depend on $R$.

\begin{remark}
The exponent $-1/(2(d+1))$ in the statement of the proposition is what is blamed for the so-called curse of dimensionality. This is due to the fact that we want a guarantee on the Lipschitz constant of the approximating neural network in terms of the Lipschitz constant of the optimal control, and hence to the best of our knowledge, the current state of the literature does not provide results allowing us to overcome at the theoretical level,  the curse of dimensionality in our setting. 
\end{remark}

\begin{remark}
\label{re:C2+C3}
	The Lipschitz continuity of $\partial_{x_i, x_j} \ff$ and the bound on its Lipschitz constant will only be used in the analysis of one of the contributions of the Euler scheme to the error due to the time discretization (see Section~\ref{SEC:APPROX-DISCRETETIME}). There, in order to apply the above result to the optimal control $\hat\bv$ we shall require the extra Assumptions~\ref{hyp:euler-extra-assumption-b-sigma-f} - \ref{hyp:euler-extra-assumption-gradAlpha} added to the standing assumptions in Apppendix \ref{ap:assumptions}. We stress that these assumptions can be omitted if one is not interested in the analysis of the error due to time discretization.
\end{remark}

\begin{proof}
Let $\ff \in \cC(R, K, L_1, L_2, L_3)$ be as in the statement. 
Proving the desired bound for each component of $\ff$ separately, we may assume that $k=1$ without any loss of generality.

Let $n_0 = \lceil C_{mm} 2^{2(d+1)} \rceil$, and fix the integers $n_{\mathrm{in}}$ and $n$ so that $n_{\mathrm{in}} \ge n_0$ and $n = \lfloor (n_{\mathrm{in}}/C_{mm})^{1/(2(d+1))} \rfloor$. As a result, 
\begin{equation}
\label{eq:nin_n}
	C_{mm} n^{2(d+1)} \le n_{\mathrm{in}},
	\qquad \hbox{ and }
	\qquad n^{-1} \le 2 \left( n_{\mathrm{in}} / C_{mm} \right)^{-1/(2(d+1))}. 
\end{equation}
For $\sigma>0$, we denote by $g_\sigma$ the density of the mean-zero Gaussian distribution on $\RR^{d+1}$ with variance/covariance matrix $\sigma^2 I_{d+1}$ where $I_{d+1}$ denotes the $(d+1)\times(d+1)$ identity matrix, and we define $\tilde {\ff}=\ff*g_\sigma: \RR^{d+1} \to \RR$ which is $\cC^\infty(\RR^{d+1})$. We choose $\sigma>0$ small enough to ensure that 
$$
	\|\ff - \tilde {\ff}\|_{\cC^0([0,T] \times \bar{B}_d(0,R))}\le n^{-1},
	\quad
	\| \tilde {\ff} \|_{\cC^0([0,T] \times \bar{B}_d(0,R))} \le K,
$$
$$
	\| \nabla_{(t,x)} \tilde {\ff} \|_{\cC^0([0,T] \times \bar{B}_d(0,R); \RR^{d})} \le L_1,
$$
and for $i,j=1,\ldots,d$,
$$
	\| \nabla_{(t,x)} \partial_{x_i}\tilde {\ff} \|_{\cC^0([0,T] \times \bar{B}_d(0,R); \RR^{d})} \le L_2, 
$$
and
$$
	\| \nabla_{(t,x)} \partial_{x_i, x_j}\tilde {\ff} \|_{\cC^0([0,T] \times \bar{B}_d(0,R); \RR^{d})} \le L_3.
$$
We then choose a constant $K_0$ depending only on $T$ and $L_2$, and a function $\zeta:[0,\infty)\to [0,1]$ with the following properties: $\zeta$ is $\cC^\infty$ and non-increasing, and there exists $R' \in (\sqrt{T^2+R^2}+1, K_0(1+R))$ such that $\zeta(r)=1$ if $r \le\sqrt{T^2+R^2}$ and $\zeta(r)=0$ if $r>R'-1$, and $|\zeta'(r)|\le L_2, |\zeta''(r)|\le L_3$ for all $r\ge 0$. 
We now consider the function $\xi$ defined on $[-R', R']^{d+1}$ by $\xi(t,x)=\tilde {\ff}(t,x)\zeta(|(t,x)|)$. Note that $\xi\in \cC^\infty$, that $\xi(t,x)$ coincides with $\tilde \ff(t,x)$ for $t \in (0,T) $ and $|x|\le R$, and that  
$\xi(t,x)=0$ if $|(t,x)|\ge R'-1$.
In particular, we have
\begin{align}
\label{eq:approx-f-tildef-xi}
	\|\ff - \xi\|_{\cC^0([0,T] \times \bar{B}_d(0,R))}
	= \|\ff - \tilde {\ff}\|_{\cC^0([0,T] \times \bar{B}_d(0,R))}
	\le n^{-1}, 
	\\
\label{eq:approx-nabla-f-tildef-xi}
	\| \nabla_{(t,x)} \xi \|_{\cC^0([0,T] \times \bar{B}_d(0,R); \RR^{d+1})}
	= \| \nabla_{(t,x)} \tilde {\ff} \|_{\cC^0([0,T] \times \bar{B}_d(0,R); \RR^{d+1})}
	\le L_1,
	\\
\label{eq:approx-nabla-partial-f-tildef-xi}
	\| \nabla_{(t,x)} \partial_{x_i} \xi \|_{\cC^0([0,T] \times \bar{B}_d(0,R); \RR^{d+1})}
	= \| \nabla_{(t,x)} \partial_{x_i} \tilde {\ff} \|_{\cC^0([0,T] \times \bar{B}_d(0,R); \RR^{d+1})}
	\le L_2, %
	\\
\label{eq:approx-nabla-partial2-f-tildef-xi}
	\| \nabla_{(t,x)} \partial_{x_i,x_j} \xi \|_{\cC^0([0,T] \times \bar{B}_d(0,R); \RR^{d+1})}
	= \| \nabla_{(t,x)} \partial_{x_i,x_j} \tilde {\ff} \|_{\cC^0([0,T] \times \bar{B}_d(0,R); \RR^{d+1})}
	\le L_3, %
\end{align}
for all $i,j=1,\dots,d.$
We now apply the result of Theorem \ref{th:MM} to the function $\tilde \xi$ defined by
$$
	[-\pi, \pi]^{d+1} \ni (t,x) \mapsto \tilde\xi(t,x) := \xi\left( R' t/\pi , R' x/\pi \right)
$$ 
extended into a $2\pi$-periodic function in each variable, and we bound from above the right hand sides of~\eqref{eq:MhaskarMicchelli-bdd-f},~\eqref{eq:MhaskarMicchelli-bdd-grad-f} and~\eqref{eq:MhaskarMicchelli-bdd-grad2-f}. We use a Jackson type result, namely the fact that the trigonometric degree of approximation of a function of class $\cC^r$,  that is $r$ times continuously differentiable, is of order $O(n^{-r})$ when using polynomials of degree at most $n$. More precisely, by~\cite[Theorem 4.3]{MR0251410}, if $F:\RR^k \to \RR$ is an $r$-times continuously differentiable function which is $2\pi$-periodic in each variable, then for every positive integer $m$, there exists a trigonometric polynomial $P_m$ of degree at most $m$ such that 
$$
	|F(x) - P_m(x)| \leq C m^{-r} \left(\sum_{i=1}^k (M_r)^{k-i} \omega_i\left(\frac{\partial^r F}{\partial x_i^r}; \frac{1}{m}\right)\right), \qquad x \in [-\pi, \pi]^k,
$$
where $C$ is an absolute constant, $M_r = \max_{1 \le s \le r+1} \begin{pmatrix} r+1 \\ s \end{pmatrix}$, and $\omega_i$ is the modulus of continuity defined, for a function $F \in \cC^0(\RR^k)$ and $h>0$, as 
$$
	\omega_i (F; h) = \max_{x \in \RR^k,\, x'\in\RR\,|x_i - x'| \le h} |F(x_1, \dots, x_i, \dots, x_k) - F(x_1, \dots, x', \dots, x_k)|.
$$
In particular, if $F$ is $L-$Lipschitz, then $\omega_i (F; h) \le L h $ and  $\omega_i (F; h) \le 2\|F\|_\infty $ when $F$ is merely bounded.

We apply this result to $F = \tilde \xi$, $F = \partial_{x_i} \tilde \xi(x)$ and $F = \partial_{x_i,x_j} \tilde \xi(x)$, with $m=n$, $k = d+1$ and $r = 1$, since $\tilde \xi$ is at least of class $\cC^{2}$. Note that, by the definition of $\tilde \xi$, by~\eqref{eq:approx-nabla-f-tildef-xi},~\eqref{eq:approx-nabla-partial-f-tildef-xi}  and~\eqref{eq:approx-nabla-partial2-f-tildef-xi}, $\tilde \xi$, $\partial_{x_i} \tilde \xi$  and $\partial_{x_i,x_j} \tilde \xi$ have derivatives with respect to all variables which are uniformly bounded from above by $\frac{R'}{\pi} L_1$, $\left(\frac{R'}{\pi}\right)^2 L_2$  and $\left(\frac{R'}{\pi}\right)^3 L_3$ respectively. 
Thus, we obtain
$$
	E^{d+1}_n( \tilde \xi )\le c_0 R' n^{-1},
	\quad
	E^{d+1}_n( \partial_i \tilde \xi )\le c_0 (R')^2 n^{-1},
	\quad
	E^{d+1}_n( \partial_{i,j} \tilde \xi )\le c_0 (R')^3 n^{-1},
$$
where $c_0$ depends only on $L_1$, $L_2$, $L_3$ and $d$. Note that we could have expected an upper bound of the order of $n^{-2}$ but this does not seem achievable with the above arguments because we do not have a control of the modulus of continuity of $\partial_t\xi$ since we do not know that $\xi$ is twice differentiable in time.
Moreover, applying~\cite[Theorem 4.3]{MR0251410} to $\psi$, $\psi'$ and $\psi''$, which are $\cC^1$, we obtain that for any integer $N$, 
$$
	E^1_N( \psi)\le c_1 N^{-1},
	\qquad
	E^1_N( \psi' )\le c_1 N^{-1},
	\qquad\text{and}\qquad
	E^1_N( \psi'' )\le c_1 N^{-1},
$$
where $c_1$ depends only on $\|\psi'\|_{\cC^0(\TT)}$, $\|\psi''\|_{\cC^0(\TT)}$ and $\|\psi'''\|_{\cC^0(\TT)}$. 
We use the above inequalities with $N = n^{1+(d+1)/2}$, so that $N^{-1} n^{(d+1) / 2} = n^{-1}$.
By Theorem~\ref{th:MM}, we obtain that
there exist $\tilde n_{\mathrm{in}} \le C_{mm} N n^{d+1}$ and $\tilde \varphi\in \bN^\psi_{d+1, \tilde n_{\mathrm{in}}, 1}$ such that, 
	\begin{align*}
		\|\tilde \xi - \tilde \varphi\|_{\cC^0(\TT^{d+1})} &\leq c_2 (1 + R') n^{-1},
		\\
		\|\partial_{x_i} \tilde \xi - \partial_{x_i} \tilde \varphi\|_{\cC^0(\TT^{d+1})} &\leq c_2 (1 + (R')^2) n^{-1}, \qquad i=1,\dots,d,
		\\
		\|\partial_{x_i,x_j} \tilde \xi - \partial_{x_i,x_j} \tilde \varphi\|_{\cC^0(\TT^{d+1})} 
		&
		\leq c_2 (1 + (R')^3) n^{-1}, \qquad i,j=1,\dots,d,
	\end{align*}
where $c_2$ depends only on $\hat \psi(1)$, $d$, $K$, $c_0$ and $c_1$. Notice that, by our choice of $N$ and~\eqref{eq:nin_n}, it holds $\tilde n_{\mathrm{in}} \le C_{mm} N n^{d+1} \le C_{mm} n^{2(d+1)} \le n_{\mathrm{in}}$ so that, up to adding dummy neurons to $\tilde \varphi$, we obtain the existence of $\varphi\in \bN^\psi_{d+1, n_{\mathrm{in}}, 1}$ such that, 
	\begin{align}
		\label{eq:bdd-tildexi-varphi}
		\|\tilde \xi - \varphi\|_{\cC^0(\TT^{d+1})} &\leq c_2(1 + R') n^{-1},
		\\
		\label{eq:bdd-grad-tildexi-varphi}
		\|\nabla_{(t,x)} \tilde \xi - \nabla_{(t,x)} \varphi\|_{\cC^0(\TT^{d+1}; \RR^{d})} &\leq c_2 (1 + (R')^2) n^{-1},
		\\
		\label{eq:bdd-grad2-tildexi-varphi}
		\|\nabla_{(t,x)} \partial_{x_i} \tilde \xi - \nabla_{(t,x)} \partial_{x_i} \varphi\|_{\cC^0(\TT^{d+1}; \RR^{d})} 
		&
		\leq c_2 (1 + (R')^3) n^{-1}, \qquad i=1,\dots,d.
	\end{align}
	
	Consider the function $\varphi_{\ff}$ defined by $\varphi_{\ff}: [0,T] \times \bar{B}_d(0,R) \ni (t,x) \mapsto \varphi (\pi t / R', \pi x / R' )$. Combining~\eqref{eq:approx-f-tildef-xi},~\eqref{eq:bdd-tildexi-varphi} and~\eqref{eq:nin_n}, yields
	\begin{align*}
		\|\ff - \varphi_{\ff} \|_{\cC^0([0,T] \times \bar{B}_{d}(0,R))}
		&\le
		\| \ff - \xi \|_{\cC^0([0,T] \times \bar{B}_{d}(0,R))}
		+ \| \xi - \varphi_{\ff} \|_{\cC^0([0,T] \times \bar{B}_{d}(0,R))}
		\\
		&\le
		\| \ff - \xi \|_{\cC^0([0,T] \times \bar{B}_{d}(0,R))}
		+ \| \tilde \xi - \varphi \|_{\cC^0(\TT^{d+1})}
		\\
		&\le n^{-1} + c_2(1 + R') n^{-1}
		\\
		 &\le c_3 (1 + R) n_{\mathrm{in}}^{-1/(2(d+1))},
	\end{align*}
where $c_3$ depends only on $c_2, C_{mm}, d$ and $K_0$. Moreover, from the definition of $\varphi_{\ff}$, ~\eqref{eq:bdd-grad-tildexi-varphi}, and~\eqref{eq:approx-nabla-f-tildef-xi}, we obtain
	\begin{align*}
		&\| \nabla_{(t,x)} \varphi_{\ff} \|_{\cC^0( [0,T] \times \bar{B}_{d}(0,R); \RR^{d})}  
		\\
		&= \frac{\pi}{R'} \| \nabla_{(t,x)} \varphi \|_{\cC^0([0,\pi T / R'] \times \bar{B}_{d}(0, \pi R/ R'); \RR^{d})} 
		\\
		&\le \frac{\pi}{R'} \left(  \|\nabla_{(t,x)} \tilde \xi  \|_{\cC^0([0,\pi T / R'] \times \bar{B}_{d}(0, \pi R/ R'); \RR^{d})} + c_2 (1 + (R')^2) n^{-1}  \right)  
		\\
		&= \frac{\pi}{R'} \left( \frac{R'}{\pi} \|\nabla_{(t,x)} \xi  \|_{\cC^0([0, T] \times \bar{B}_{d}(0, R); \RR^{d})} + c_2 (1 + (R')^2) n^{-1}  \right)  
		\\
		&\le c_4 \left(1 + R n_{\mathrm{in}}^{-1/(2(d+1))} \right),
	\end{align*}
where $c_4$ depends only on $c_2$, $L_1$, $C_{mm}, d$ and $K_0$. 
	Similarly, for each $i=1\dots,d$, from the definition of $\varphi_{\ff}$, ~\eqref{eq:bdd-grad2-tildexi-varphi}, and~\eqref{eq:approx-nabla-partial-f-tildef-xi}, we obtain
	\begin{align*}
		&\| \nabla_{(t,x)} \partial_{x_i} \varphi_{\ff} \|_{\cC^0( [0,T] \times \bar{B}_{d}(0,R); \RR^{d})}  
		\\
		&= \left(\frac{\pi}{R'}\right)^2 \| \nabla_{(t,x)} \partial_{x_i} \varphi \|_{\cC^0([0,\pi T / R'] \times \bar{B}_{d}(0, \pi R/ R'); \RR^{d})} 
		\\
		&\le \left(\frac{\pi}{R'}\right)^2 \left(  \|\nabla_{(t,x)} \partial_{x_i} \tilde \xi  \|_{\cC^0([0,\pi T / R'] \times \bar{B}_{d}(0, \pi R/ R'); \RR^{d})} + c_2 (1 + (R')^3) n^{-1}  \right)  
		\\
		&= \left(\frac{\pi}{R'}\right)^2 \left( \left(\frac{R'}{\pi}\right)^2 \|\nabla_{(t,x)} \partial_{x_i} \xi  \|_{\cC^0([0, T] \times \bar{B}_{d}(0, R); \RR^{d})} + c_2 (1 + (R')^3) n^{-1}  \right)  
		\\
		&\le c_5 \left(1 + R n_{\mathrm{in}}^{-1/(2(d+1))} \right),
	\end{align*}
where $c_5$  depends only on $c_2$, $L_3$, $C_{mm}, d$ and $K_0$.
\end{proof}

We then turn our attention to the variations of the objective functional when evaluated at two controls that are similar on a compact set. Since, under our assumptions, both the (approximately optimal) feedback control $\hat\bv$ and the control $\hat\bvarphi$ produced by the neural network are Lipschitz continuous, we restrict our attention to feedback controls with this regularity.
The following proposition provides a control of the variation of the objective function when approximating a feedback control on a compact set.

\begin{proposition}
\label{lem:var-J-feedback-compact}
	Let $\bv$ and $\bw \in \VV$ be two Lipschitz continuous feedback controls (i.e., Lipschitz continuous functions of $(t,x)$). There exists a constant $C$ depending only on the data of the problem and the Lipschitz constant of the controls $\bv$ and $\bw$ such that for all $\Gamma > 0$ and $R > 0$, if
	\begin{equation}
	\label{eq:varJN-bnd-bv-bw}
		\| \bv_{|[0,T] \times \bar{B}_{d}(0,R)} - \bw_{|[0,T] \times \bar{B}_{d}(0,R)} \|_{\cC^0([0,T] \times \bar{B}_{d}(0,R))} \leq \Gamma \, ,
	\end{equation}
	then
	$
		\left | J^N(\bv) - J^N(\bw) \right | \leq C \left(\Gamma^2 + \frac{1}{R}\right)^{1/2} \, .
	$
\end{proposition}
	The arguments required for the proof are rather standard. However, we could not find a close enough version of this result in the literature. So for the sake of completeness we provide the proof in Appendix~\ref{app:proof-approx-NN}.

\begin{proof}[Proof of Proposition \ref{prop:approx-NN-cmpJ-epsilon}]
Let $n_{\mathrm{in}}\ge 1$ be an integer, and let $R = n_{\mathrm{in}}^{1/(3(d+1))}$.  
As explained in Subsection \ref{sub:assumptions} on our standing assumptions, the function $(t,x) \mapsto \hat{\bv}(t,x)$ defined by~\eqref{fo:vhat} is in the class $\cC(R, K, L_1, L_2, L_3)$ for some constants $K, L_1, L_2, L_3$ depending only on the data of the problem. So by Proposition~\ref{prop:approx-fct-NN} applied to $\ff = \hat{\bv}$, there exists a neural network $\hat{\bvarphi} \in \bN^{\psi}_{d+1, n_{\mathrm{in}}, k}$, such that
	$
		\| \hat{\bv} - \hat{\bvarphi} \|_{\cC^0([0,T] \times \bar{B}_{d}(0,R))}  \le \Gamma,
	$
with $\Gamma = \Gamma(R, n_{\mathrm{in}}) = C_0  (1 + R) n_{\mathrm{in}}^{-1/(2(d+1))}$ for some $C_0$ depending only on the data of the problem and on $\hat\psi(1), \|\psi'\|_{\cC^0}, \|\psi''\|_{\cC^0}, \|\psi'''\|_{\cC^0}$.
Moreover, still by Proposition~\ref{prop:approx-fct-NN}, the Lipschitz constants of $\hat\bvarphi$  and $\partial_{x_i} \hat\bvarphi$ can be bounded from above by $C_0  (1 + R n_{\mathrm{in}}^{-1/(2(d+1))})$, which is itself bounded by $2 C_0$ thanks to our choice of $R$.
  	
Next, applying Proposition~\ref{lem:var-J-feedback-compact} yields
 	$$
	J^N(\hat{\bv}) \geq J^N(\hat{\bvarphi}) - C_2 \left(\Gamma^2 + \frac{1}{R}\right)^{1/2} \,,
	$$
	for some constant $C_2$ depending only on the data of the problem and $C_0$.
 Recalling that $R = n_{\mathrm{in}}^{1/(3(d+1))}$,
 \begin{align*}
 	\Gamma^2 + \frac{1}{R} 
	&= C_0^2  \left(1 + n_{\mathrm{in}}^{1/(3(d+1))}  \right)^2 n_{\mathrm{in}}^{-1/(d+1)} + n_{\mathrm{in}}^{-1/3((d+1))} 
	\\
	&\le C_1 n_{\mathrm{in}}^{-1/(3(d+1))},
 \end{align*}
where $C_1$ depends only on $C_0$. This completes the proof.
\end{proof}

\subsection{Problem in discrete time}
\label{SEC:APPROX-DISCRETETIME}

We now prove that the objective values of the discrete and continuous time problems (respectively Problem~\ref{pb:discrete-MKV} and Problem~\ref{pb:MKV-Nagents} when minimizing over neural networks) are close. 

One possibility is to view the interacting particle dynamics as a system of SDEs in dimension $N \times d$ and apply known results about the Euler scheme for systems of SDEs. However such results usually involve constants which depend upon the dimension. This is unsatisfactory as we want the constants to be independent of $N$.

Viewing the problem as the optimal control of a finite number of particles is reminiscent of the particle method for MKV equations studied by Bossy and Talay e.g.~\cite{MR1370849}. However their bound for the error induced by the time discretization (see~\cite[Lemma 2.7]{MR1370849}) is not applicable directly to our setting. 

For these reasons and for the sake of completeness, we provide in Appendix~\ref{app:proof-discreteT} a detailed proof, which relies on the following estimate for the strong error rate also proved in the appendix under the assumptions and notations we now specify. $\varphi$ is a Lipschitz continuous feedback control function of $(t, x)$, and $(X^i_t)_{i=1,\ldots,N,\;0\le t\le T}$ is the $N$-agent continuous-time state process satisfying the MKV dynamics given by equation
\eqref{eq:evolXi-closedloop} in the statement of Problem \ref{pb:MKV-Nagents} with feedback function $\varphi$ instead of $v$. $N_T$ is an integer giving the number of time steps in the subdivision $t_n=nT/N_T$ for $n=0,\ldots,N_T$ and we denote by $\Delta t=T/N_T$ the mesh of this subdivision.
$(\check X^i_{t_n})_{i=1,\ldots,N,\;n=0,\ldots,N_T}$ is the $N$-agent discrete-time state process satisfying the MKV dynamics given by equation
\eqref{eq:discrete-dyn} in the statement of Problem \ref{pb:discrete-MKV} with the current Lipschitz feedback function $\varphi$ (not necessarily given by a neural network). Under these conditions we have

\begin{lemma}[Strong error for the Euler scheme] 
\label{lem:discreteT-strongerr}
For all $n=0,\dots,N_T$,
$$
	 \sum_{i=1}^N \EE \left[ |\check X^i_{t_n} - X^i_{t_n}|^2 \right]
	\leq
	C N\Delta t,
$$
	where $C$ depends only on the data of the problem,  the Lipschitz constant of $\varphi$ and $\varphi(0,0)$.
\end{lemma}
Using this estimate, we can prove the following bound.

\begin{proposition}
\label{prop:approx-discreteT}
Assume that  for every $(t,x) \in [0,T] \times \RR^d$,
\begin{align}
\label{eq:hyp-vaphi-Euler}
	|\partial_t\varphi(t,x)| \le C_1 |(t, x)|,
	\qquad
	|\partial_x\varphi(t,x)| \le C_1,
	\qquad 
	|\partial^2_{xx} \varphi(t,x)| \le C_1.
\end{align}
Then there exists a constant $C$ depending only on the data of the problem and on the control $\varphi$ via its Lipschitz constant, its value at $(0,0)$ and  $C_1$ introduced above, such that for all $N$ and $t\in[0,T]$, we have
	$$
	|J^N(\varphi) -  \check J^N(\varphi)|
	\leq C \sqrt{\Delta t}.
	$$	
\end{proposition}
In particular, for the same $\hat\bvarphi$ as the one constructed in Proposition~\ref{prop:approx-NN-cmpJ-epsilon}, the constant $C$ depends only on the data of the problem and on the activation function through $\hat\psi(1)$, $\|\psi'\|_{\cC^0(\TT)}$, $\|\psi''\|_{\cC^0(\TT)}$  and $\|\psi'''\|_{\cC^0(\TT)}$.

\section{Two numerical methods}
\label{sec:two-num-meth}

In this section, we describe two numerical methods for mean field type problems. The first method is a \emph{straightfoward} implementation of the optimization problem underpinning the optimal control of MKV dynamics. The second method is geared to the solution of forward-backward systems of stochastic differential equations of the McKean-Vlasov type (MKV FBSDEs). As a result, it can be applied to the solutions of both mean field control (MFC) problems and mean field games (MFGs).

\subsection{Method 1: minimization of the MFC cost function}
\label{sec:method-1}

The goal is to minimize $\check J^N$ defined by~\eqref{eq:def-checkJN} by searching for the right function $\bvarphi$ in a family of functions $x\mapsto \bvarphi_\theta(x)$ parameterized by $\theta\in\Theta$, the desired parameter $\theta^*$ minimizing the functional:
$
	\theta \mapsto \JJ^N(\theta) = \check J^N(\bvarphi_\theta).
$
This cost function can be interpreted as a loss function and viewed as an expectation. Its minimization is \emph{screaming} for the use of the Robbins-Monro procedure. Moreover, if we use the family $(\varphi_\theta)_{\theta\in\Theta}$ given by a feed-forward neural network, this minimization can be implemented efficiently with the powerful tools based on the so-called Stochastic Gradient Descent (SGD) developed for the purpose of machine learning. Recall the notation $\bN^\psi_{d_0, \dots, d_{\ell+1}}$ introduced in \eqref{fo:N_psi} for a set of neural networks with a specific architecture. In the implementation of the minimization, we replace the expectation in~\eqref{eq:def-checkJN} by an empirical average over a finite number of sample populations each of size $N$. These populations can be interpreted as mini-batches (in which the samples are not independent). At each iteration $m = 1, 2, \dots,$ of the SGD, we denote by $\theta_m$ the parameters of the neural network at the current iteration, and we pick $N$ samples $x^{i}_0,$ $i=1,\dots,N,$ i.i.d. with distribution $\mu_0$, and $N \times N_T$ samples $\epsilon^{i}_{n},$ $i=1,\dots,N, 0 = 1, \dots, N_T-1$, i.i.d. with distribution $\cN(0, \sqrt{\Delta t})$. We then consider the following loss function which depends on the sample $S = ((x^{i}_0)_{i=1,\dots,N}, (\epsilon^{i}_n)_{i=1,\dots,N, n=0,\dots,N_T-1})$:
\begin{equation}
\label{eq:bbJ-S-thetam}
		\JJ^N_{S}(\theta_{m})
		= 
		\frac{1}{N}\sum_{i=1}^N \sum_{n=0}^{N_T-1} f\left(t_n, \check X^i_{t_n}, \check\mu_{t_n}, \varphi_{\theta_m}(t_n, \check X^i_{t_n})\right) \Delta t + g\left(\check X^i_{t_{N_T}},\check\mu_{t_{N_T}}\right) \, ,
	\end{equation}
	where $\check\mu_{t_n} = \frac{1}{N} \sum_{i=1}^N \delta_{\check X^i_{t_n}}$, under the constraints:
	\begin{equation*}
		\check X^i_{t_{n+1}} = \check X^i_{t_n} + b\left(t_n, \check X^i_{t_n}, \check\mu_{t_n}, \varphi_{\theta_m}(t_n, \check X^i_{t_n})\right) \Delta t + \sigma\left(t_n, \check X^i_{t_n}, \check\mu_{t_n}\right) \epsilon^i_n  \, ,
	\end{equation*}
	for $n \in \{0,\dots, N_T-1\}, i \in \{1,\dots,N\}$, with $(\check X^i_0)_{i \in \{1, \dots, N\}} = (x^i_0)_{i \in \{1, \dots, N\}}$  where $x^i_0,\; i=1,\dots,N,$ are i.i.d. with distribution $\mu_0$. Recall that the initial distribution $\mu_0$ as well as the distributions of the noise terms are parts of the data defining the problem.

\begin{algorithm}
\DontPrintSemicolon
\KwData{An initial parameter $\theta_0\in\Theta$.  A number of iterations $M$. A sequence $(\beta_m)_{m = 0, \dots, M-1}$ of learning rates.}
\KwResult{Approximation of $\theta^*$}
\Begin{
  \For{$m = 0, 1, 2, \dots, M-1$}{
    Pick $S = ((x^{i}_0)_{i=1,\dots,N}, (\epsilon^{i}_n)_{i=1,\dots,N, n=1,\dots,N_T})$ where $x^i_0$ are i.i.d. with distribution $\mu_0$, and $\epsilon^{i}_n$ are i.i.d. with distribution $\cN(0, \sqrt{\Delta t})$\;
    Compute the gradient $\nabla \JJ^N_{S}(\theta_{m})$ of $\JJ^N_{S}(\theta_{m})$ defined by~\eqref{eq:bbJ-S-thetam} \;
    Set $\theta_{m+1} =  \theta_{m} -\beta_m \nabla \JJ^N_{S}(\theta_{m})$ \;
    }
  \KwRet{ $\theta_{M}$}
  }
\caption{SGD for Mean Field Control\label{algo:SGD-MFC}}
\end{algorithm}

Details on the implementation and numerical examples are given in Section \ref{sec:numres}

\subsection{Method 2: MKV FBSDE}
\label{eq:method2-FBSDE}

We now propose a method for a general system of forward-backward SDEs of McKean-Vlasov type (or MKV FBSDE for short). For simplicity, we restrict our attention to the case of a non-controlled volatility. Such a system then takes the following generic form
\begin{equation}
\label{eq:MKV-FBSDE-general}
\left\{
\begin{aligned}
	d X_t
	=
	& B\left(t, X_t, \cL(X_t), Y_t \right) dt
		+  \sigma\left(t, X_t, \cL(X_t) \right) d W_t,
	\\
	d Y_t
	=
	& - F\left(t, X_t, \cL(X_t), Y_t, \sigma\left(t, X_t, \cL(X_t) \right)^\dagger Z_t \right) dt
		+ Z_t d W_t,
\end{aligned}
\right.
\end{equation}
with initial condition $X_0 = \xi \in L^2(\Omega, \cF_0, \PP; \RR^d)$ and terminal condition $Y_T = G(X_T, \cL(X_T))$. The system~\eqref{eq:MKV-FBSDE} derived from the application of the Pontryagin stochastic maximum principle applied to our MKV control problem is an instance of the above general MKV FBSDE. As argued in \cite{MR3752669}, it turns out that the solution of a MFG can also be captured by a FBSDE system of the form given by~\eqref{eq:MKV-FBSDE-general}. Moreover, such a system can also be obtained using dynamic programming, in which case $Y$ represents the value function instead of its gradient.

The strategy of the method is to replace the backward equation forced on us by the optimization, by a forward equation and treat its initial condition, which is what we are looking for, as a control for a new optimization problem. This strategy has been successfully applied to problems in economic contract theory where it is known as Sannikov's trick. See for example \cite{MR1766421,MR2963805, MR3738664}.
Translated into the present context, this strategy allows us to turn the search for a  solution of~\eqref{eq:MKV-FBSDE-general} into the following optimization procedure. We let the controller choose the initial point and the volatility of the $Y$ process, and penalize it proportionally to how far it is from matching the terminal condition. More precisely, the problem is defined as: Minimize over $y_0 : \RR^d \to \RR^d, z: \RR_+ \times \RR^d \to \RR^{d \times d}$ the cost functional
$$
	J_{FBSDE}(y_0, z) = \EE \left[ \, \left| Y^{y_0,z}_T - G(X^{y_0,z}_T, \cL(X^{y_0,z}_T)) \right|^2 \, \right]
$$
where $(X^{y_0,z}, Y^{y_0,z})$ solve
\begin{equation}
\label{eq:MKV-FBSDE-2forward}
\left\{
\begin{aligned}
d X^{y_0,z}_t
	=
	& B\left(X^{y_0,z}_t, \cL(X^{y_0,z}_t), Y^{y_0,z}_t \right) dt
		+  \sigma\left(X^{y_0,z}_t, \cL(X^{y_0,z}_t) \right) d W_t,
	\\
	d Y^{y_0,z}_t
	=
	& - F\left(X^{y_0,z}_t, \cL(X^{y_0,z}_t), Y^{y_0,z}_t, \sigma\left(X^{y_0,z}_t, \cL(X^{y_0,z}_t) \right)^\dagger z(t, X^{y_0,z}_t)  \right) dt
	\\
	&\qquad
		+ z(t, X^{y_0,z}_t) d W_t,
\end{aligned}
\right.
\end{equation}
with \emph{initial} condition 
$X^{y_0,z}_0 = \xi \in L^2(\Omega, \cF_0, \PP; \RR^d),$ 
$Y^{y_0,z}_0 = y_0 (X_0).$ 
We purposely omitted the dependence of $B,F,$ and $\sigma$ on $t$ to shorten the notations.
Note that the above problem is an optimal control problem of MKV dynamics if we view $(X^{y_0,z}_t,Y^{y_0,z}_t)$ as the state and $(y_0,z)$ as the control. It is of a peculiar form, since the control is the initial value and the volatility of the second component of the state as functions of the first component of the state. Under suitable conditions, the optimally controlled process $(X_t,Y_t)_t$ solves the FBSDE system~\eqref{eq:MKV-FBSDE-general} and vice-versa. 

Due to its special structure, the above MFC problem does not fit in the framework we have analyzed in Section~\ref{sec:approx-result}. However, for numerical purposes, we can still apply the first method described above, see \S~\ref{sec:method-1}: we consider a finite-size population, replace the the controls (namely, $y_0$ and $z$) by neural networks, discretize time, and then use SGD to perform the optimization. In Section \ref{sec:numres} below, we illustrate the performance of this method on MKV FBSDEs coming from MFGs, i.e. game models for which we cannot use the first method.

\begin{remark}
The qualitative analysis of the error due to a plain Euler scheme applied to \eqref{eq:MKV-FBSDE-2forward} could be approached by the stability results of \cite{han2020convergence}. Indeed Theorem 1 therein provides an upper bound in terms of the shooting error of the Euler scheme in the forward direction, while their Theorem 2 addresses the control of this error in terms of approximations of higher order derivatives of the solution.
Combined with the approximation of the interactions using finite-size population sample means followed by the use of SGD, these upper bounds should hold in the present set-up.
\end{remark}

\section{Numerical results}
\label{sec:numres}

In this section, we provide numerical results obtained using implementations of the two numerical methods proposed above. Algorithm 1 refers to the first method, based on direct minimization of the cost function for a mean field control problem; Algorithm 2 refers to the second method, which solves a control problem encoding the solution to an FBSDE system of MKV type.  For all the numerical results presented below, the neural networks we used take the time $t$ as an input. As a consequence, the same architecture (and hence the same number of parameters) can be used for various time discretizations.  Some of the test cases (see in particular test cases 5 and 6 below) do not fall in the scope of the theory presented in the previous sections. However the numerical results show that the algorithms are robust enough and work well even in these cases. As a benchmark, we shall use for each test case either the solution provided by an analytical formula, or the solution computed by a deterministic method based on finite differences for the corresponding PDE system, see e.g.~\cite{MR2679575}.

We recall here the PDE system for the sake of completeness and for future reference.
For simplicity, in the numerical examples considered below, we take $\sigma$ to be constant and we assume that the initial distribution $\mu_0$ has a density denoted by $m_0$. 
In \cite{MR3134900}, A. Bensoussan, J. Frehse and P. Yam have proved that a necessary condition
for the existence of a smooth feedback function $v^* $ achieving
$J( v^*)= \min J(v)$ is that
\begin{displaymath}
  v^*(t,x)={\rm{argmin}}_{v}  \Bigl( f(x, m(t,\cdot), v)+ \nabla u(t,x)\cdot b(x,m(t,\cdot),v)\Bigr),
\end{displaymath}
where $J$ is defined by~\eqref{eq:MFC1-cost} (with a slight abuse of notation we view $J$ as a function of a feedback control) and
where $(m,u)$ solve the
following system of partial differential equations
\begin{equation}
\label{eq:MFC-PDEsystem}
\left\{
\begin{split}
	0&=
	\frac{\partial u} {\partial t} (t,x) + \frac{\sigma^2}{2} \Delta u(t,x) + \tilde H( x, m(t,\cdot), \nabla u(t,x)) 
	 \\
	 &\qquad + \int \frac{\partial \tilde H} {\partial m}(\xi, m(t,\cdot), \nabla u(t, \xi))(x) m(t,\xi) d\xi  ,
\\
	0&= \frac{\partial m} {\partial t} (t,x)  - \frac{\sigma^2}{2} \Delta  m(t,x) +
	 \mathrm{div}\Bigl( m(t,\cdot) \frac{\partial \tilde H} {\partial p} (\cdot,m(t,\cdot),\nabla u(t,\cdot))\Bigr)(x) ,
\end{split}
\right.
\end{equation}
where $\tilde H$ is defined by~\eqref{eq:def-Hamiltonian-red}, 
with the initial and terminal conditions
\begin{equation*}
  m(0,x)=m_0(x),
  \quad %
   u(T,x) = g (x, m(T,\cdot)) + \int \frac{\partial g} {\partial m} (\xi, m(T,\cdot))(x) m(T,\xi) d\xi.
\end{equation*}
Here $\frac{\partial}{\partial m}$ denotes a Fr\'echet derivative. This PDE system has been analyzed and solved numerically e.g. in~\cite{MR3392611} and we use a similar numerical method below to compute our benchmark solutions when no analytical formula is available.

\vskip 12pt
\noindent
\textbf{Test case 1:} 
As a first testbed, let us consider a linear-quadratic mean field control problem, that is, the dynamics are linear and the cost is quadratic in the state, the control and the expectation of the state. We consider a problem in which the state and the control are both in dimension $d=10$. To wit, recalling the notations~\eqref{eq:MFC1-cost}--\eqref{eq:MFC1-dyn}, we take 
$b(x, \mu, v) = Ax + \bar A \bar \mu + B v,$ $f(x, \mu, v) = Q x^2 + \bar Q (\bar \mu - Sx)^2 + R v^2,$ and $g(x, \mu) = Q_T x^2 + \bar Q_T(\bar \mu - S_Tx)^2 $,
for $x, v\in \RR^{10}, \mu \in \cP_2(\RR^{10}), \bar \mu = \int \xi \mu(d\xi) \in \RR^{10}$, where $A,\bar A, B,$ $Q, \bar Q, R, S,$ $Q_T, \bar Q_T, S_T$ are $10 \times 10$ matrices. For simplicity, we take them to be multiples of the identity matrix, so that the solution is the same on all dimensions. A benchmark solution can be computed using a forward-backward system of one-dimensional ODEs. Using a finite-difference method on the $10$-dimensional problem (i.e. without using the knowledge of the matrices structure) would be computationally prohibitive. 

In contrast, we solve the problem using Algorithm~1, without relying a priori on the structure of the matrices.  Numerical results are presented in Figure~\ref{fig:ex-lq-benchmark}. We recall that in this method, the cost function is also the loss function minimized by the SGD. The value of this loss w.r.t. the number of iterations is displayed in Figure~\ref{fig:ex-lq-benchmark-cost}, for a number of time steps $N_T \in \{10, 20, 100\}$ and a number of samples in one population $N \in \{32, 128, 1024\}$.  Using the solution computed with ODEs, we compare the learned control with the actual optimal control. As shown in Figure~\ref{fig:ex-lq-benchmark-errcontrol}, the $L^2$ distance between them decreases as the number of iterations of SGD increases. As expected, increasing the number of time steps and the number of samples in a population improves the result.    These results were obtained with a neural network having $2$ hidden layers, each with $100$ units. The figures have been obtained by evaluating the performance of the learned control in the following way: for each pair $(N,N_T)$, we trained the neural network and evaluated its performance with a population of $N_{testing} = 1024$ samples. We repeated the experiments on $3$ different runs of SGD and averaged over them to obtain the final curves.

\begin{figure}
\centering
\begin{subfigure}{.49\textwidth}
  \centering
  \includegraphics[width=\linewidth]{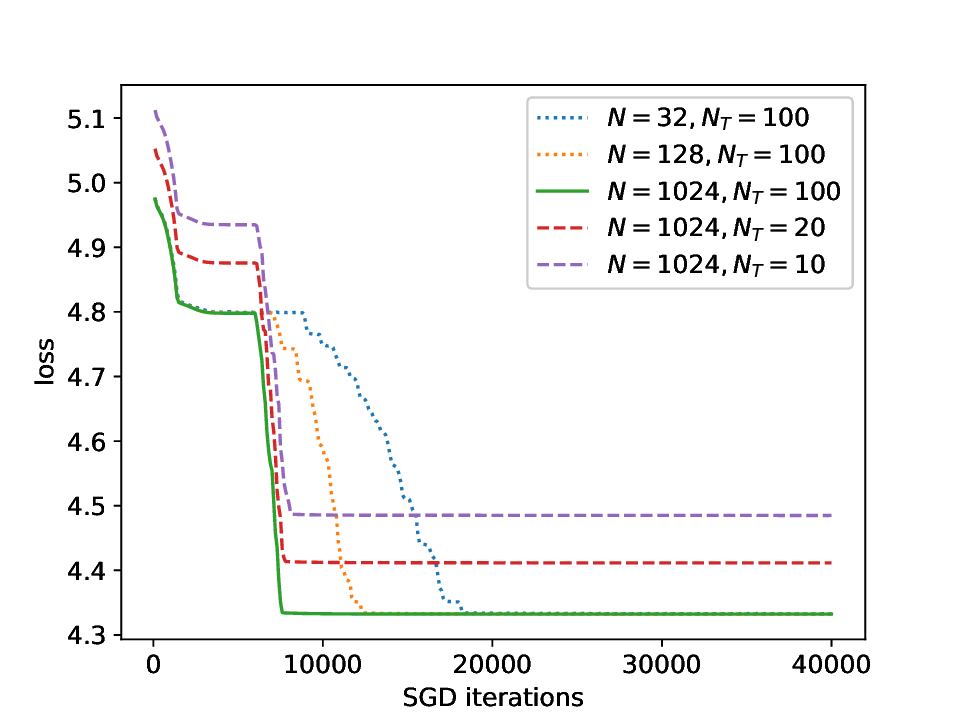}
  \caption{Cost for learned control}
  \label{fig:ex-lq-benchmark-cost}
\end{subfigure}%
\begin{subfigure}{.49\textwidth}
  \centering
  \includegraphics[width=\linewidth]{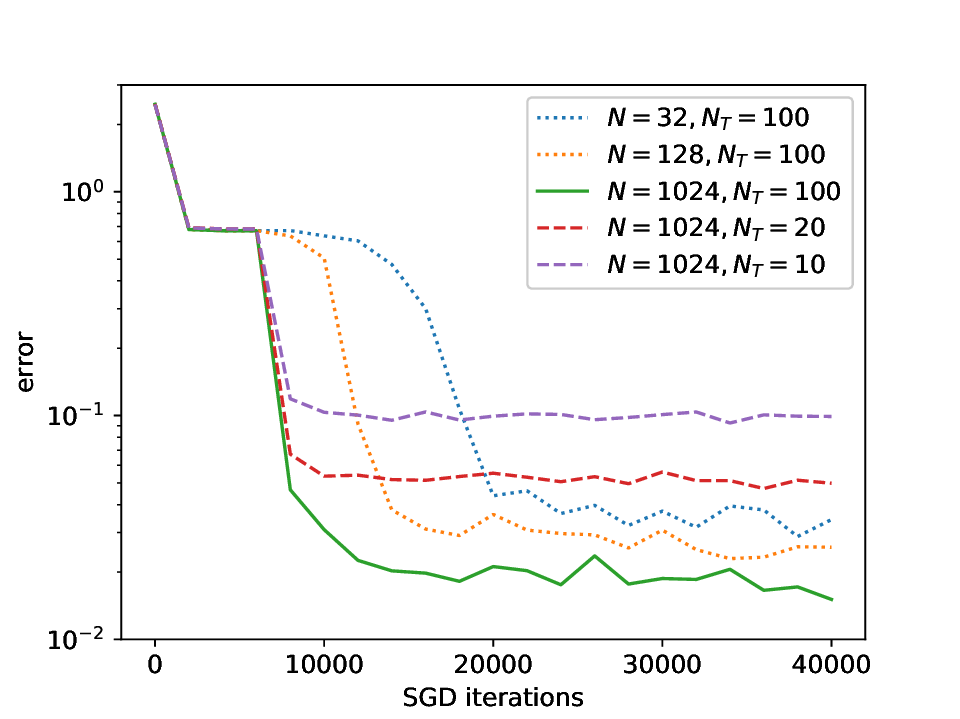}
  \caption{$L^2$ error on the control}
  \label{fig:ex-lq-benchmark-errcontrol}
\end{subfigure}
\caption{Test case 1. Solution computed by Algorithm 1.}
\label{fig:ex-lq-benchmark}
\end{figure}

\vskip 12pt
\noindent
\textbf{Test case 2:} 
Our second example is inspired by the min-LQG problem of~\cite{salhab2015a}. As in the first test case, the dynamics are linear and the running cost function is quadratic. However, we modify the terminal cost to encourage each agent to end up close to one of two targets at the final time. More precisely, considering again a one-dimensional case, with $\xi_1, \xi_2 \in \RR$, we take 
\begin{align*} 
	&g(x,m) = g(x) = \min \left\{  |x - \xi_1 |, | x - \xi_2 | \right\} .
\end{align*} 
Notice that $g$ is differentiable everywhere except at the point $\frac{1}{2}(\xi_1 + \xi_2)$, and the gradient of $g$ is discontinuous.
Numerical results obtained using Algorithm 2 are presented in Figure~\ref{fig:ex-minlq-benchmark} for an example where the two targets are $\xi_1 = 0.25$, $\xi_2 = 0.75$, and the initial distribution is Gaussian with mean $1$. We recall that with this method, we learn $(y_0(\cdot), z(\cdot))$, and then $(Y_{t_n})_{n > 0}$ can be obtained by Monte-Carlo simulations, which yield the blue points. To assess the accuracy of our method, we compare the result with the solution obtained by a deterministic PDE method based on a finite difference scheme (red line).  From the perspective of the coupled system~\eqref{eq:MFC-PDEsystem} of Hamilton-Jacobi-Bellman and Fokker Planck (HJB-FP system for short) characterizing the optimal solution, $Y_t$ correspond to the gradient $\nabla u(t, X_t)$ of the adjoint function $u$, solution to the HJB equation.
We see respectively on Figure~\ref{fig:ex-minlq-benchmark-cost-Y-t0} and Figure~\ref{fig:ex-minlq-benchmark-cost-Y-tT} that $Y_0$ is better approximated than $Y_T$, which is consistent with the fact that $Y_0$ is directly learned under the form $y_0(X_0)$ where $y_0$ is a neural network, whereas $Y_T$ accumulates the errors made on all intermediate time steps until the terminal time.  Note however that the algorithm manages to learn the fact that the gradient of the value function $u$ has a jump at $x=\frac{1}{2}(\xi_1+\xi_2) = 1$. For the results presented here, we have used neural networks for $y_0$ and $z$ with $3$ hidden layers and $100$ units in each hidden layer. We used $N_T = 200$ time steps and $N=2048$.

\begin{figure}
\centering
\begin{subfigure}{.49\textwidth}
  \centering
  \includegraphics[width=\linewidth]{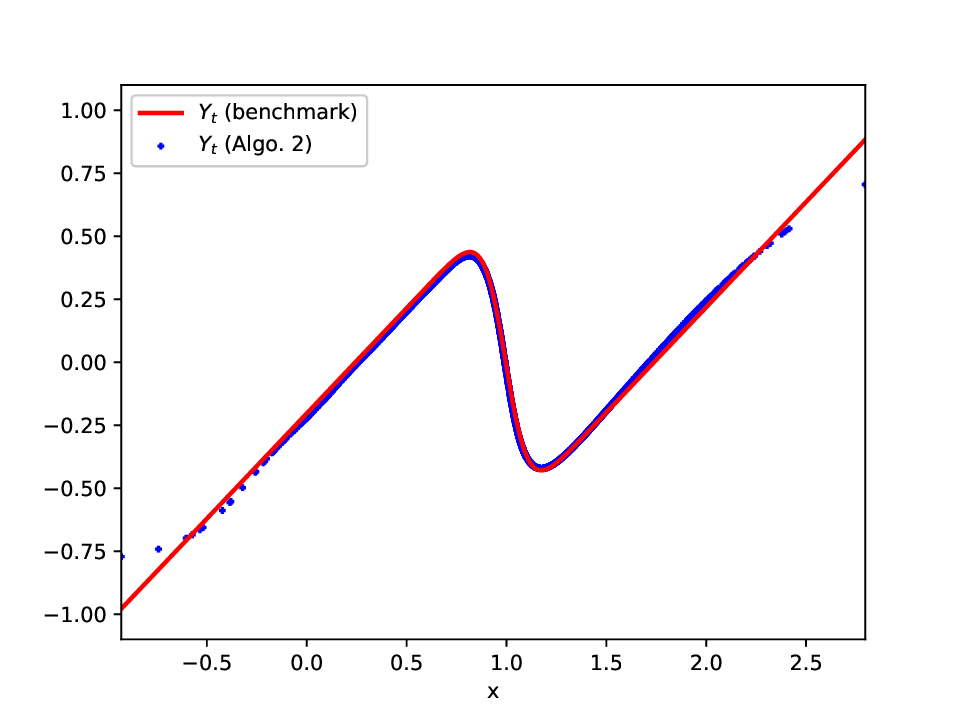}
  \caption{$Y_t$ at time $t=0$}
  \label{fig:ex-minlq-benchmark-cost-Y-t0}
\end{subfigure}%
\begin{subfigure}{.49\textwidth}
  \centering
  \includegraphics[width=\linewidth]{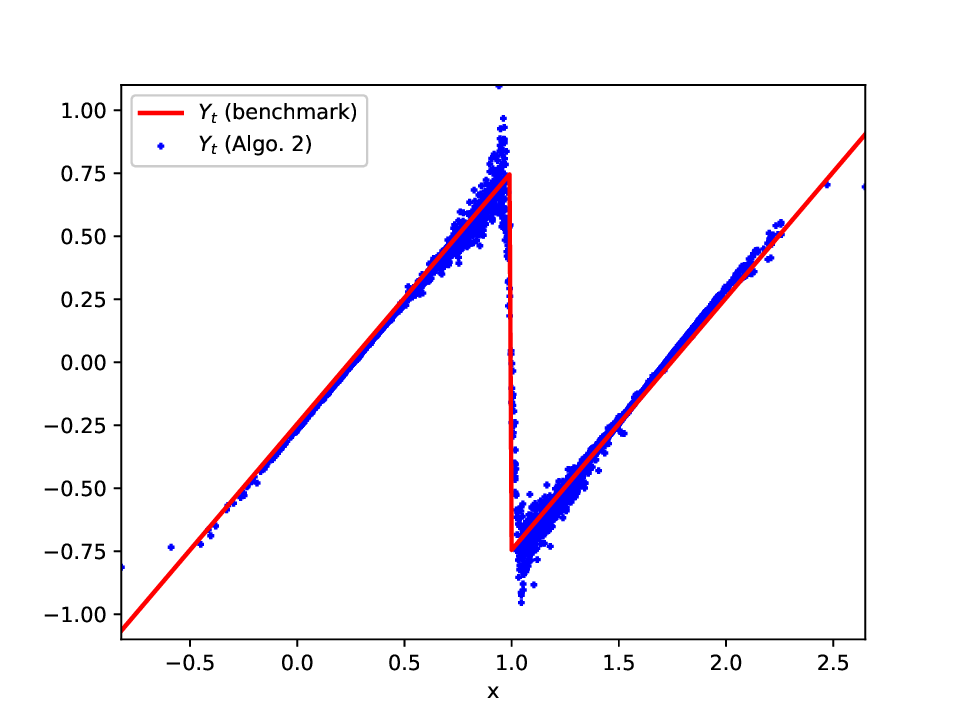}
  \caption{$Y_t$ at time $t=T = 0.2$}
  \label{fig:ex-minlq-benchmark-cost-Y-tT}
\end{subfigure}
\caption{Test case 2. Solution computed by Algorithm 2.}
\label{fig:ex-minlq-benchmark}
\end{figure}

\vskip 12pt
\noindent
\textbf{Test case 3:} 
For the sake of comparison with the existing literature, we now turn our attention to examples considered in~\cite{MR3914553}. We first consider the problem presented in their section 4.2.2, which consists in solving the FBSDE (without MKV dynamics)
\begin{align*}
\begin{cases}
	&d X_t = \rho \cos(Y_t) dt + \sigma dW_t, \qquad X_0 = x_0,
	\\
	& Y_t = \EE_t[\sin(X_T)].
\end{cases}
\end{align*}
Here, $\rho$ is seen as a coupling parameter and, in particular, for $\rho=0$, the forward equation does not depend on the backward component so the equations are decoupled. The system is solved for various values of $\rho$ and for each value of $\rho$, we look at the value of the backward process $Y$ at time $0$. The existence and uniqueness of a solution is discussed in~\cite{MR3914553}, where the authors also point out that the numerical method they propose suffers from a bifurcation phenomenon for large values of $\rho$. Results obtained by our Algorithm 2 are presented in Figure~\ref{fig:ex-CCD72-Y0}. For the sake of comparison, we also display the values obtained using a deterministic method based on a finite difference scheme for the corresponding PDE system~\eqref{eq:MFC-PDEsystem}.  In this example, the value of $Y_0$ corresponds to the value of $\nabla u(0, x_0)$. Here $x_0 = 0$. The left pane (to be compared with~\cite[Figure 2]{MR3914553}) is for the case where the time horizon is $T=1$, while the right pane is for the case $T=2$. For these results, we used a step size $\Delta t = 1/500$, $N=2048$ and neural networks with $5$ hidden layers and $100$ units in each hidden layer. Note that to solve the problem on a larger time horizon, we increase the number of time steps to keep the same step size. However, we keep the same neural network architecture and therefore the same number of parameters.

\begin{figure}
\centering
  \includegraphics[width=0.49\linewidth]{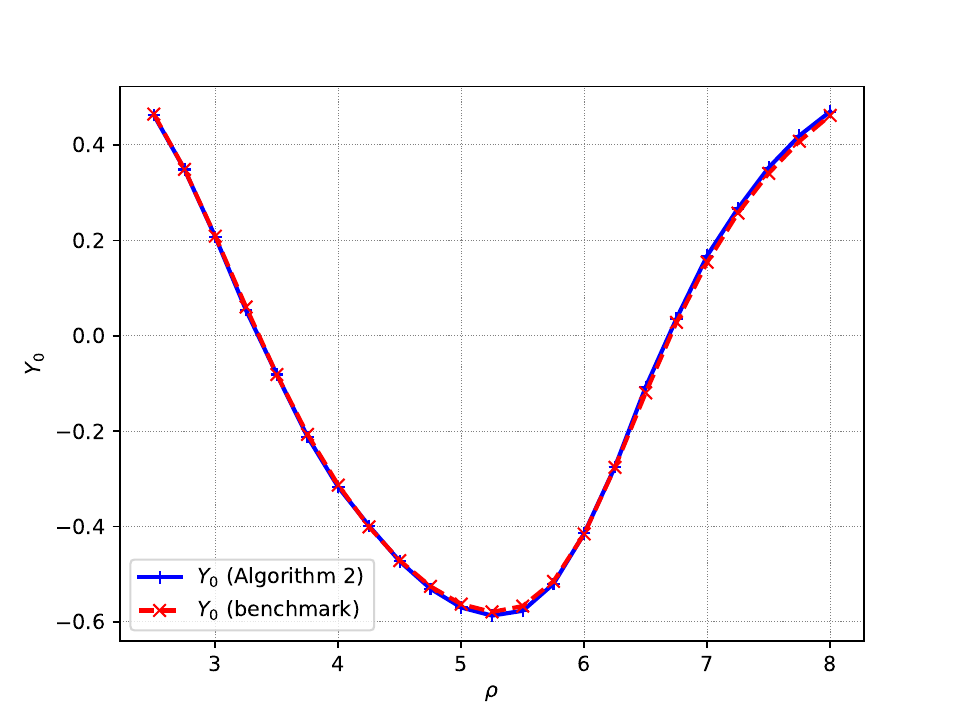}
   \includegraphics[width=0.49\linewidth]{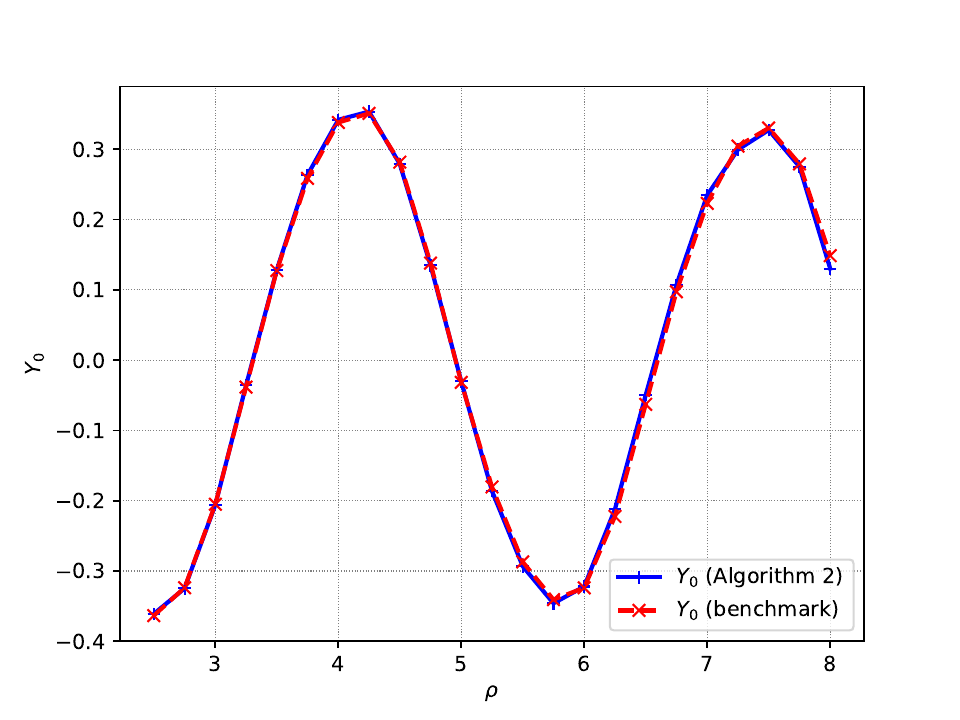}
\caption{Test case 3. Value of $Y_0$ as a function of the parameter $\rho$, computed by Algorithm 2 (in blue) and benchmark values via PDE method (in red), for $T=1$ (left) and $T=2$ (right).}
\label{fig:ex-CCD72-Y0}
\end{figure}

\vskip 12pt
\noindent
\textbf{Test case 4:}
We next turn our attention to a FBSDE system arising from a MFG. Let us recall the example considered in~\cite[Section 4.2.3]{MR3914553}. The problem consists in solving the MFG in which the dynamics of a typical agent is $d X^\alpha_t = \alpha_t dt + d W_t$ and her cost is
$$
	J^{MFG}(\mu, \alpha) = \EE\left[ g(X^\alpha_T) + \int_0^T \left(\frac{1}{2\rho} \alpha^2_t + X^\alpha_t \mathrm{atan}\left(\int x \mu_t(dx)\right) \right) dt \right]
$$
when she uses control $\alpha$ and the flow of distribution of the population is given by $\mu = (\mu_t)_t$, and where the terminal cost $g$ is such that $g'(x) = \arctan(x)$. As in the previous test case, $\rho >0$ is a parameter. Since this is a MFG, Algorithm 1 can not be employed so we resort to Algorithm 2 in order to solve the associated MKV FBSDE system, namely:
\begin{align*}
\begin{cases}
	&d X_t = - \rho Y_t dt + \sigma dW_t, \qquad X_0 = x_0,
	\\
	& d Y_t = \arctan(\EE[X_t]) dt + Z_t dW_t, \qquad Y_T = g'(X_T) := \arctan(X_T).
\end{cases}
\end{align*}

Here again, we solve the problem for various values of $\rho$ and, for each of them, we compute the value of $Y_0$. The numerical results are displayed in Figure~\ref{fig:ex-CCD73-Y0}. We see that the solution is very close to the benchmark, obtained by a deterministic method for the corresponding PDE system. In particular, a noticeable difference with the results of~\cite[Figure 3]{MR3914553} is that our method does not suffer from a bifurcation phenomenon.  This is probably related to the fact that our method relies on the idea of rephrasing the FBSDE system as a control problem driven by two forward SDEs. For the results presented here, we took $x_0=1$ and $T=1$ as in~\cite{MR3914553}. We used $N_T = 200$ time steps, $N=4096$ and neural networks with $3$ hidden layers and $200$ units in each hidden layer. 
\begin{figure}
\centering
  \includegraphics[width=0.7\linewidth]{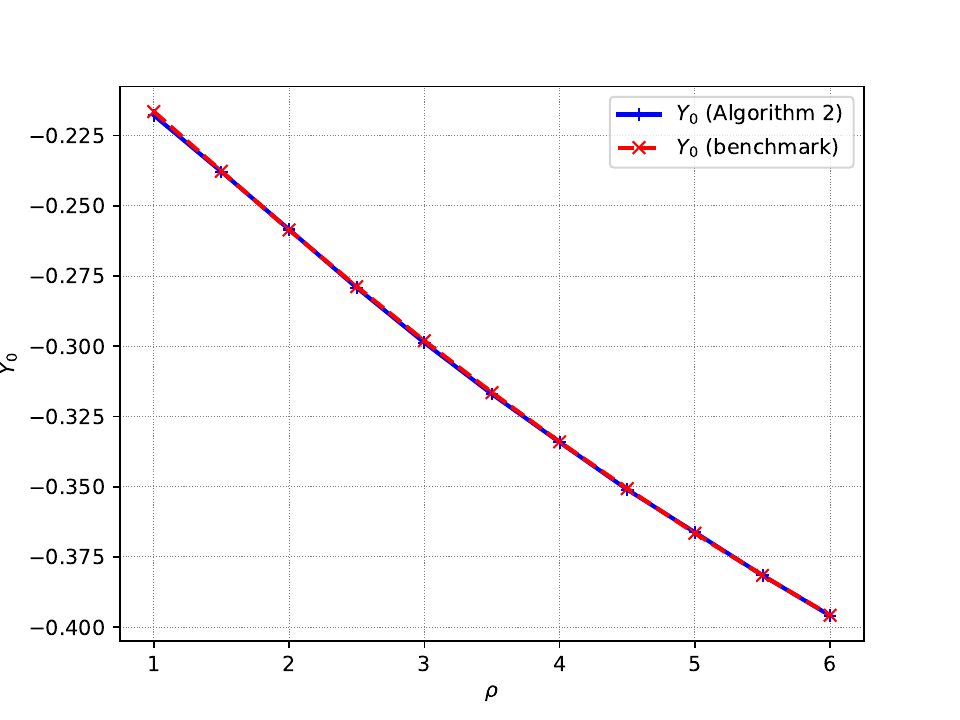}
\caption{Test case 4. Value of $Y_0$ as a function of the parameter $\rho$, computed by Algorithm 2 (in blue) and benchmark values via PDE method (in red).}
\label{fig:ex-CCD73-Y0}
\end{figure}

\vskip 12pt
\noindent
\textbf{Test case 5:} We present an example of mean field control problem with a simple source of common noise. The model is inspired by~\cite{YAJML}, where a more complex model has been studied for crowd motion. For simplicity, we consider the one-dimensional case. The idea is to consider a terminal cost which penalizes the distance to one of two targets, the relevant target being revealed only at time $T/2$. 
Here, we represent the common noise by a random process $\epsilon^0 = (\epsilon^0_t)_{t \in [0,T]}$ which takes value $0$ for $t \in [0, T/2)$ and at time $T/2$ it jumps to value $-c_T$ or $+c_T$, each with probability $1/2$, and then it keeps this value until time $T$. Here $c_T>0$ is a constant and the two targets are $\pm c_T$. For simplicity, we restrict our attention to controls which are feedback functions of $t, X_t$ and $\epsilon^0_t$. The problem is:
	Minimize over $v: [0,T] \times \RR \times \{-c_T,0,c_T\} \to \RR$ the quantity
	\begin{equation*}
		J(v) 
		= 
		\EE\left[\int_0^T |v(t, X_t, \epsilon^0_t)|^2 dt + (X_T - \epsilon^0_T)^2 + (X_T  - \EE[ X_T | \epsilon^0])^2 \right] \, ,
	\end{equation*}
	under the constraint that the process $\bX = (X_t)_{t \ge 0}$ solves the SDE
	\begin{equation*}
		dX_t = v(t, X_t, \epsilon^0_t) dt + \sigma dW_t \, , t \ge 0,
		\quad
		X_0 \sim \mu_0.
	\end{equation*} 
Note that the cost function is assumed to depend on the \emph{conditional} distribution of $X_t$ given the common noise $\epsilon^0$. The solution can be characterized through a system of $6$ PDEs (three coupled systems of HJB and FP equations). We use this to have a benchmark, and compare with the result obtained by Algorithm 1 suitably generalized to handle the common noise. More precisely, at each iteration of SGD, we sample the initial positions of $N$ agents, the $N$ idiosyncratic noises appearing in the dynamics, as well as one path of the common noise. The neural network representing the control function takes as inputs the current time, the current state and the current value of the common noise. Figure~\ref{fig:ex-mfc-commonnoise} displays the evolution of the density in the scenario where $\epsilon^0_T = -c_T$ (in blue) and in the scenario where it is $\epsilon^0_T = c_T$ (in red), with $c_T = 1.5$ and the initial distribution is centered around $0$. The solution obtained by the PDE method is shown with lines, whereas the solution obtained by Algorithm~1 is shown with a histogram obtained by Monte Carlo simulations. We see that, until time $T/2$, the distribution concentrates around $0$, because the agents do not know the final target but they know that the final cost will penalize distance to the mean position. After time $T/2$, the valid target is revealed and the distribution moves toward this target. Note however that due to the running cost, which penalizes displacement, the distribution does not reach exactly the target.  For the results presented here, we used $N_T = 50$ time steps, $N=1024$ and neural a network with $3$ hidden layers and $100$ units in each layer.

\begin{figure}
\centering
\begin{subfigure}{.49\textwidth}
  \centering
  \includegraphics[width=\linewidth]{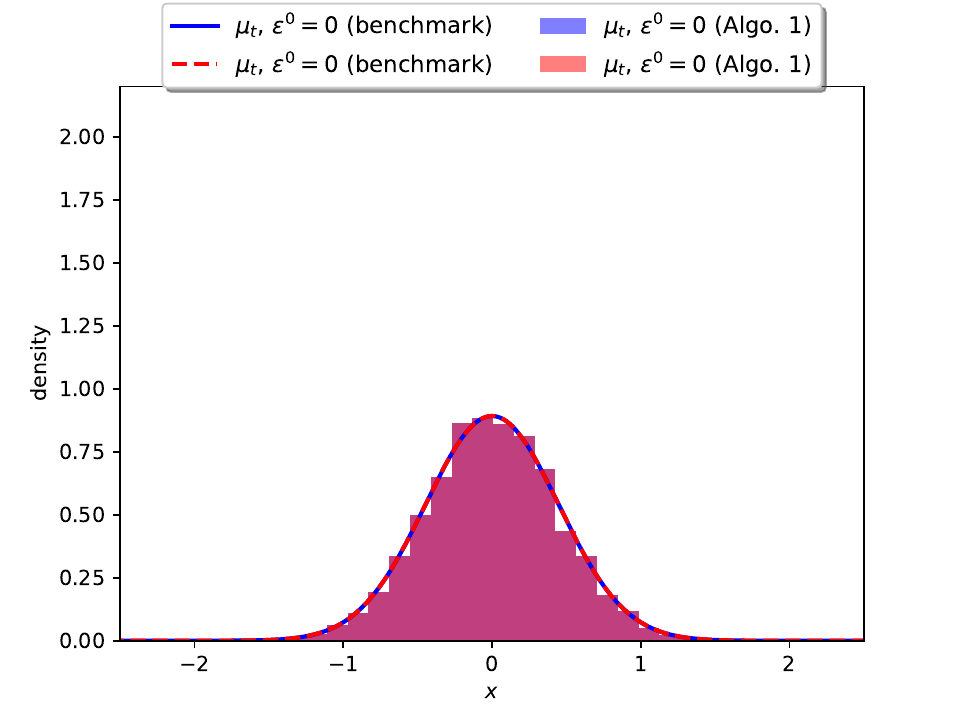}
  \caption*{$t=0$}
\end{subfigure}
\begin{subfigure}{.49\textwidth}
  \centering
  \includegraphics[width=\linewidth]{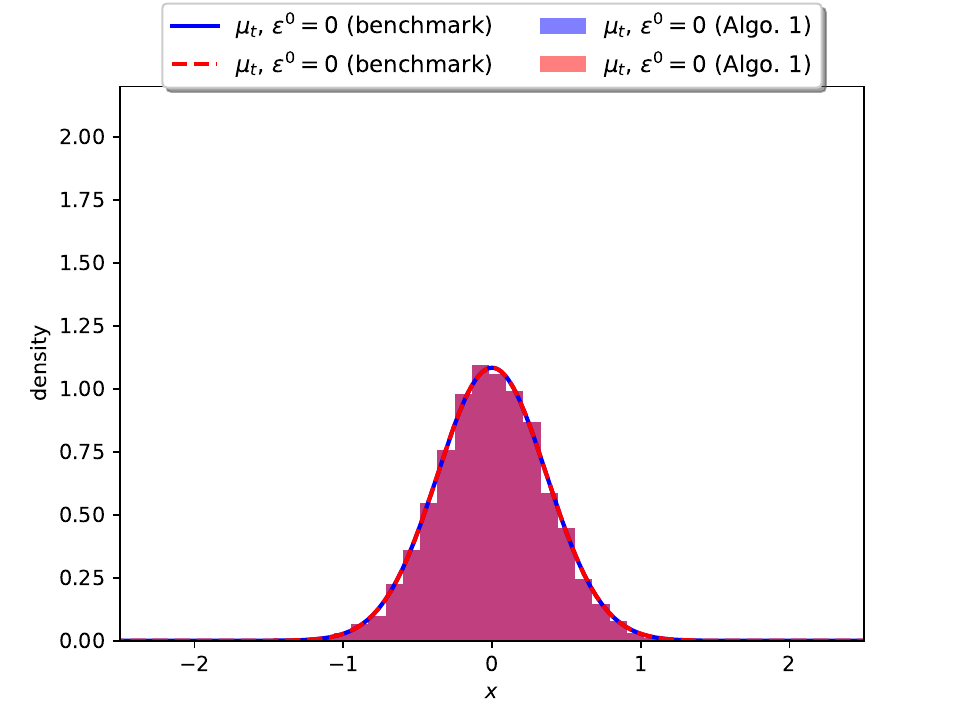}
  \caption*{$t=0.2$}
\end{subfigure}
\\
\begin{subfigure}{.49\textwidth}
  \centering
  \includegraphics[width=\linewidth]{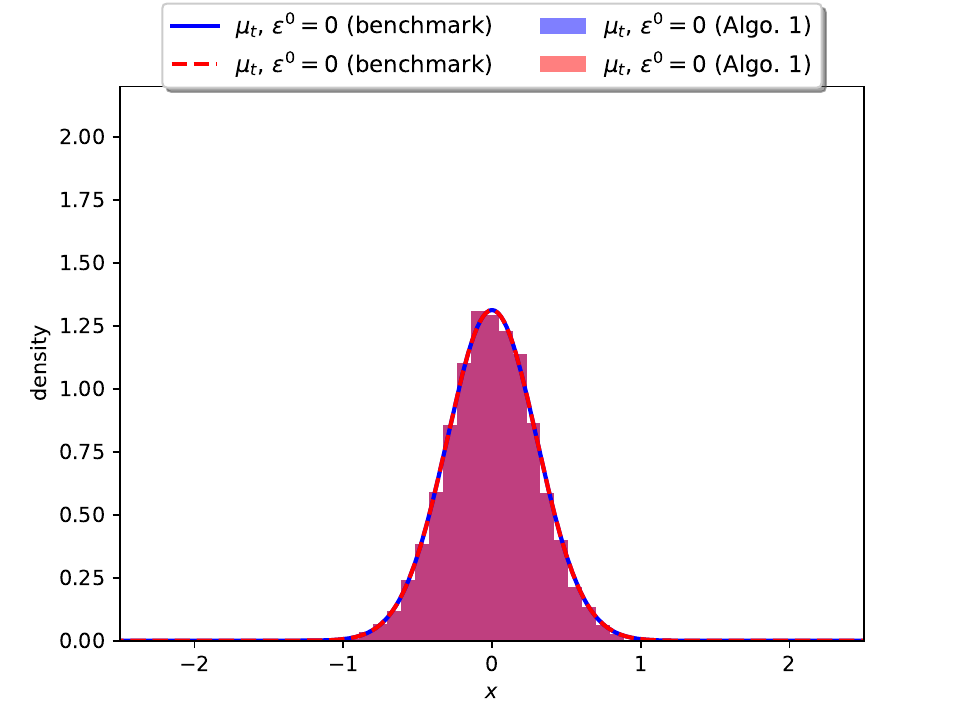}
  \caption*{$t=0.5$}
\end{subfigure}
\begin{subfigure}{.49\textwidth}
  \centering
  \includegraphics[width=\linewidth]{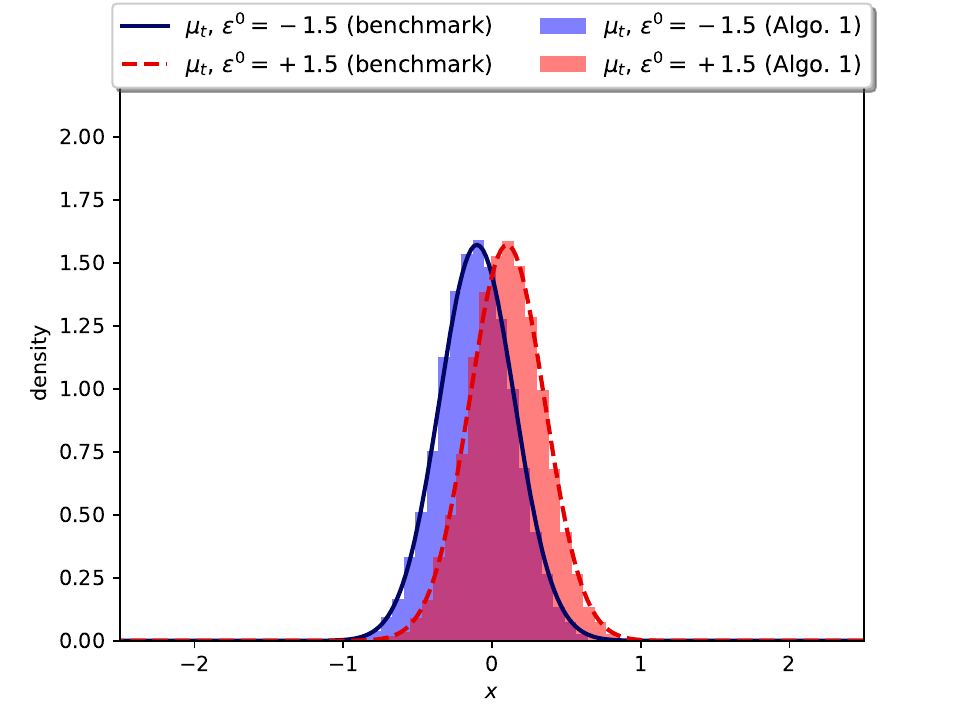}
  \caption*{$t=0.6$}
\end{subfigure}
\\
\begin{subfigure}{.49\textwidth}
  \centering
  \includegraphics[width=\linewidth]{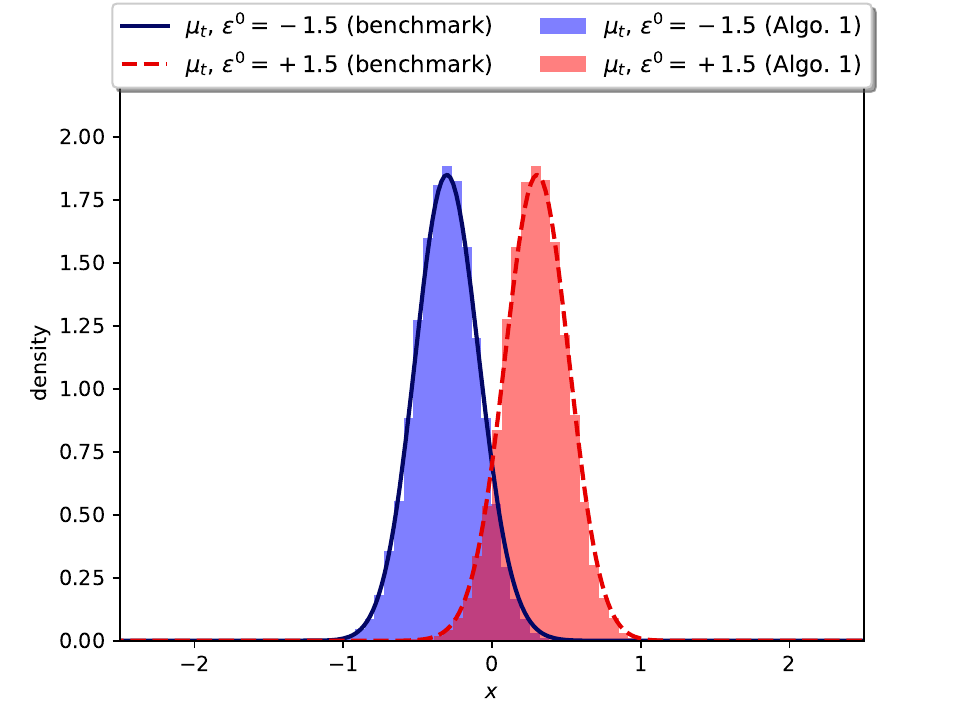}
  \caption*{$t=0.8$}
\end{subfigure}
\begin{subfigure}{.49\textwidth}
  \centering
  \includegraphics[width=\linewidth]{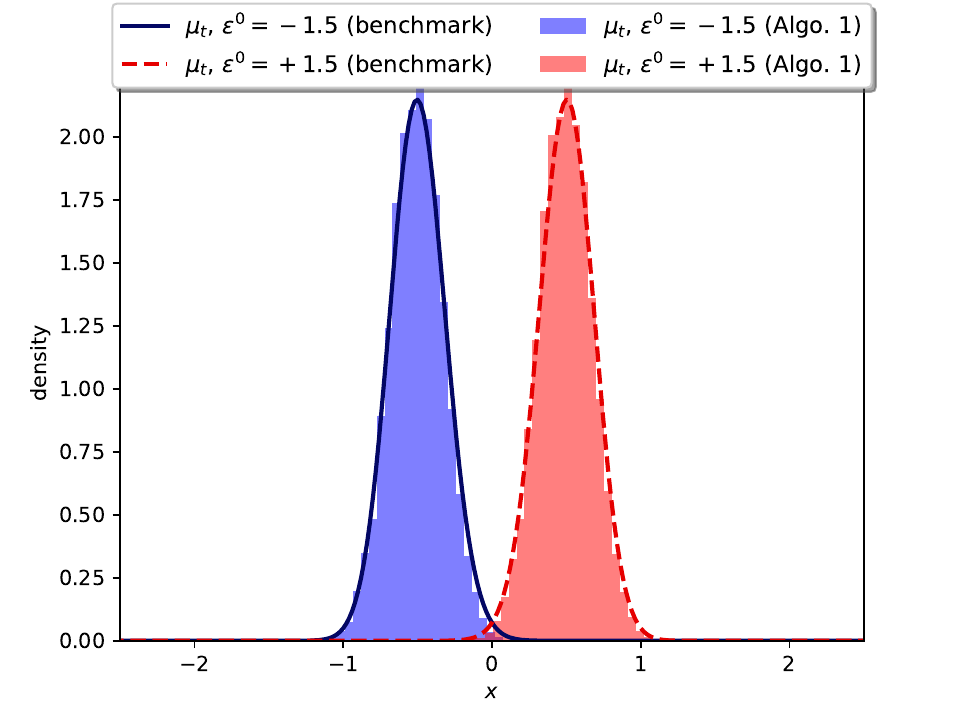}
  \caption*{$t=1.0$}
\end{subfigure}
\caption{Test case 5. Distribution computed by Algorithm 1 (histograms) and by deterministic method for the PDE system (full and dashed lines). Blue (resp. red) corresponds to the scenario where $\epsilon^0$ takes the value $-1.5$ (resp. $+1.5$) at time $T/2 = 0.5$.}
\label{fig:ex-mfc-commonnoise}
\end{figure}

\vskip 12pt
\noindent
\textbf{Test case 6:} Here we consider the example of mean field game for systemic risk introduced in~\cite{MR3325083}. This model involves common noise in the form of a Brownian motion $\bW^0 = (W^0_t)_{t \ge 0}$. Given a flow of conditional mean positions $\bar {\mathbf m} = (\bar m_t)_{t \in [0,T]}$ adapted to the filtration generated by $\bW^0$, the cost function of a representative player is
\begin{align*}
	&J^{MFG}(\bar m, \alpha) 
	\\
	&= \EE\left[\int_0^T \left( \frac{1}{2}\alpha_t^2 - q \alpha_t (\bar m_t - X_t) + \frac{\epsilon}{2} (\bar m_t - X_t)^2 \right) dt + \frac{c}{2} (\bar m_T - X_T)^2\right]
\end{align*}
and the dynamics are
$$
	dX_t = [a (\bar m_t - X_t) + \alpha_t] dt + \sigma \left( \rho \, dW^0_t + \sqrt{1 - \rho^2} dW_t\right).
$$
Here $\rho \in [0, 1]$ is a constant parameterizing the correlation between the noises, and  $q, \epsilon, c, a, \sigma$ are positive constants. We assume that $q \le \epsilon^2$ so that the running cost is jointly convex in the state and the control variables. 
For the interpretation of this model in terms of systemic risk, the reader is referred to~\cite{MR3325083}. 
 The model is of linear-quadratic type and hence has an explicit solution through a Riccati ODE, which we use as a benchmark. 
 
Since this example is a mean field game, we used Algorithm 2 to solve the appropriate FBSDE system. In order to deal with the additional randomness induced by the common noise, we add $\bar m_t$ as an argument of the neural networks playing the roles $y_0(\cdot)$ and $z(\cdot)$ introduced in \S~\ref{eq:method2-FBSDE}. 
For the sake of illustration, Figure~\ref{fig:ex-mfg-sysrisk-traj} displays, for one realization of the common noise, sample trajectories of $X^i$ and $Y^i$ for two different values of $i$, namely $i=1,2$.  The $L^2$ error averaged over $5$ different runs of SGD is shown in Figure~\ref{fig:ex-mfg-sysrisk-error}. For each run, the $L^2$ error on $X$ is computed as: $\left(\frac{1}{N}\sum_{i=1}^N \sum_{t = 0}^{N_T} \|X^{i,{\tt algo}}_t - X^{i,{\tt benchmark}}_t\|^2 \Delta t\right)^{1/2}$, where $X^{i,{\tt algo}}$ and $X^{i,{\tt benchmark}}$ stand for the solution computed respectively by Algorithm~2 and by the ODE method. The $L^2$ error on $Y$ is defined similarly. Clearly, the error decreases as the number $N_T$ of time steps and the number of elements $N$ in the particle population increases. Furthermore, the variance also decreases as $N$ grows. For the numerical tests presented here, we used $\sigma = 0.5, \rho = 0.5, q = 0.5, \epsilon = q^2 + 0.5 = 0.75, a = 1, c = 1.0$ and $T = 0.5$.

\begin{figure}
\centering
\begin{subfigure}{.49\textwidth}
  \centering
  \includegraphics[width=\linewidth]{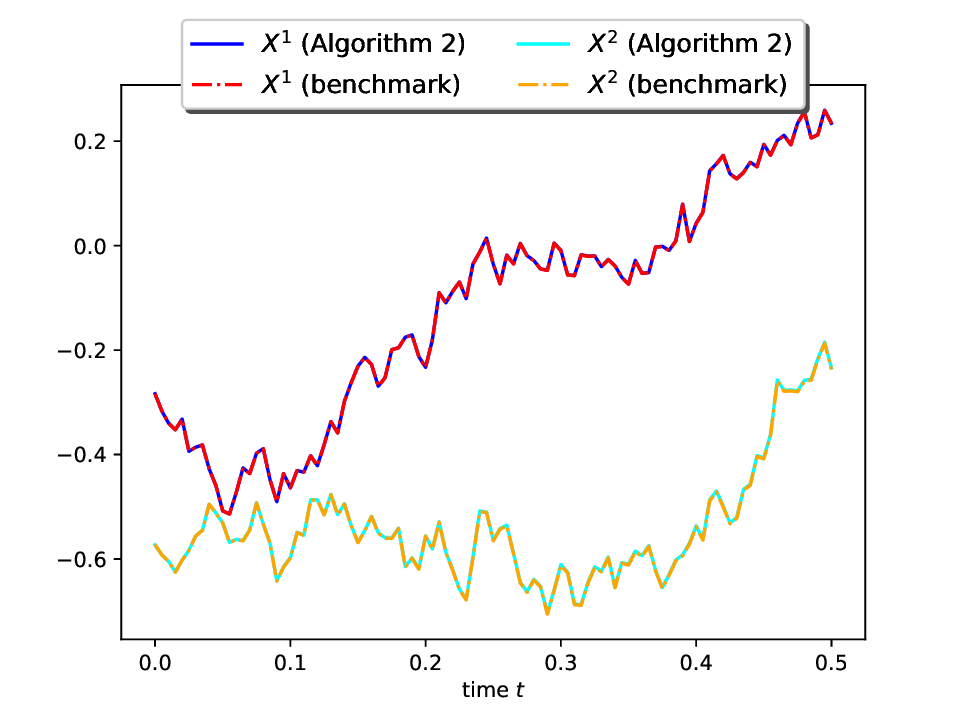}
  \caption*{Trajectories of $X^i, i=1,2$}
\end{subfigure}
\begin{subfigure}{.49\textwidth}
  \centering
  \includegraphics[width=\linewidth]{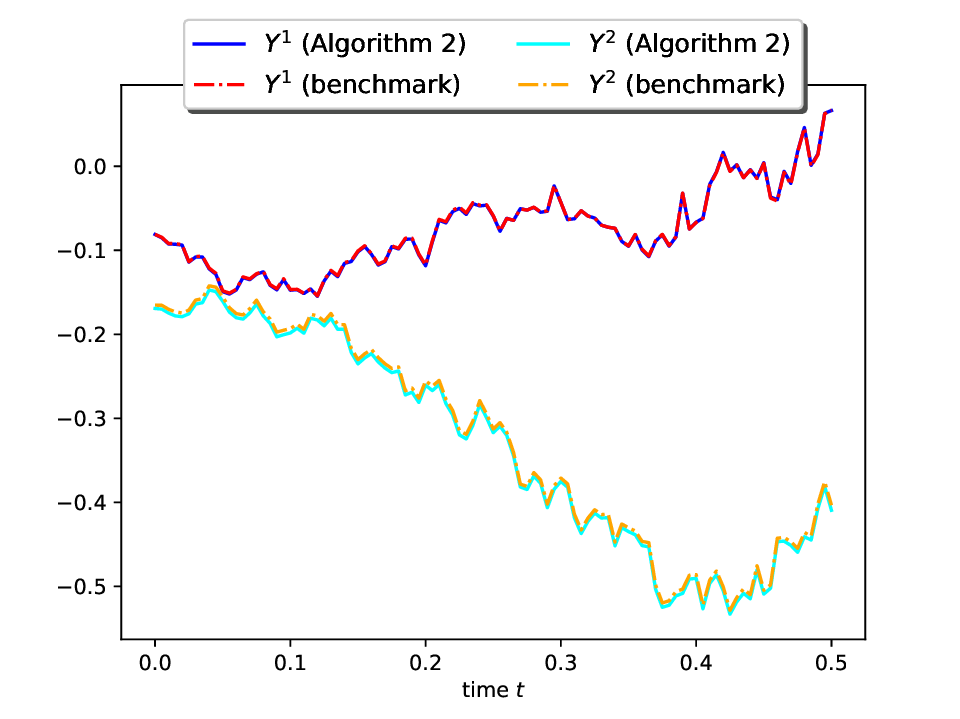}
  \caption*{Trajectories of $Y^i, i=1,2$}
\end{subfigure}
\caption{Test case 6. Sample trajectories: solution computed by Algorithm 2 (full lines, in cyan and blue) and by analytical formula (dashed lines, in orange and red).}
\label{fig:ex-mfg-sysrisk-traj}
\end{figure}

\begin{figure}
\centering
\begin{subfigure}{.49\textwidth}
  \centering
  \includegraphics[width=\linewidth]{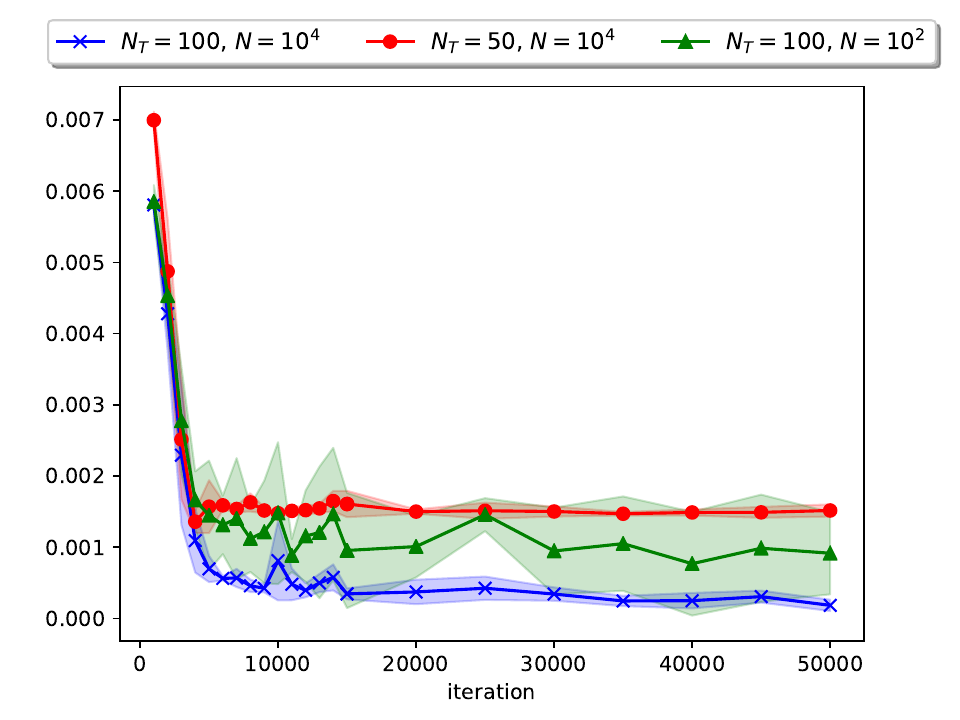}
  \caption*{$L^2$ error on $X$}
\end{subfigure}
\begin{subfigure}{.49\textwidth}
  \centering
  \includegraphics[width=\linewidth]{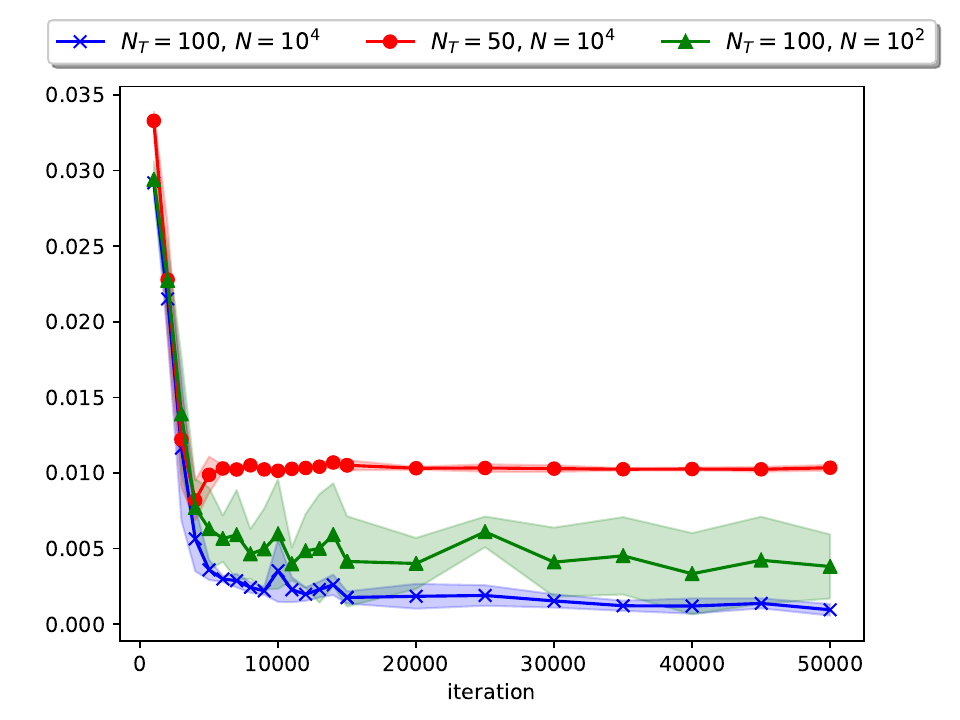}
  \caption*{$L^2$ error on $Y$}
\end{subfigure}
\caption{Test case 6. $L^2$ error for the solution computed by Algorithm 2 compared with the benchmark solution. The number of time steps is denoted by $N_T$ and the number of elements in the population sampled at each iteration of SGD is denoted by $N$.}
\label{fig:ex-mfg-sysrisk-error}
\end{figure}

\section{Conclusion}

The methods proposed here have several advantages. First, prior to this work, only~\cite{balata2019class} proposed stochastic methods for mean field control problems. These methods were not applicable to MFGs, were restricted to a specific (``polynomial'') family of MKV control problems and required either a discretization of space or a suitable choice of basis functions, which is in general hard to identify. Our methods cover a broader class of models and can be used for both MKV control problems and MFG. The fact that our methods do not require space discretization is a major advantage in comparison with deterministic methods based for instance on finite difference schemes. Indeed, as hinted by the LQ example of numerical Test case 1, we expect the two algorithms we propose to work well even in higher dimensions. Furthermore, even though numerical schemes for MKV FBSDEs have been proposed in~\cite{MR3914553,MR3968538}, here we propose to solve such systems by recasting them as (mean field) control problems. This allows us to avoid recursive methods or fixed point methods which typically work only in small time or with small enough coefficients. As demonstrated by Test case 4, our results do not suffer from bifurcation phenomenon. Last point in case, the first method we propose is directly based on the definition of the MKV control problem and does not require deriving any optimality condition (e.g. via PDEs or SDEs). In this sense, it is much more straightforward to implement numerically, which can also be considered as an advantage.

From here, several directions can be considered for future work. On the theoretical side, our proof provides a template to establish a rate of convergence. Improving any of the three steps in this proof would lead to an improvement in the overall convergence result. As far as the second step, i.e. the approximation by neural network, is concerned, to the best of our knowledge, there is not yet any result in the literature which would help improving the bound we obtained. This is because the question of approximating a function and its derivative using a neural network does not seem to have attracted much interest in the past. We believe that this question is recently gaining interest particularly due to the use of neural networks for solving PDEs. We thus expect that in the near future, new results could provide better rates of convergence. On the numerical side, for the results presented here, we have used feedforward fully connected neural networks for the sake of simplicity, but more involved architectures could be employed. This could lead to faster or more accurate computations. Moreover, instead of trying to approximate the optimal control as a function of the time and the state only, one could try to approximate with a neural network the function $V$, which is linked to the optimal control as function of the population distribution. This more challenging problem is the topic of a work in progress. Last, we have applied our second method to examples in which the backward variable plays the role of the derivative of the value function, but the same method can be applied to FBSDE systems stemming from dynamic programming, in which the backward variable represents directly the value function. More generally, the same method can be applied to generic FBSDE systems of MKV type. Hence we believe that the methods presented here will find many more applications.

\appendix

\section{Standing assumptions}
\label{ap:assumptions}
We state a set of technical assumptions which guarantee that the assumptions formulated in Subsection \ref{sub:assumptions} are satisfied.
The conditions we propose include Assumption ``\textnormal{\textbf{Control of MKV Dynamics}}'' from  p. 555 of~\cite{MR3752669} with the additional assumption that the volatility is not controlled: 
\begin{enumerate}[label=\normalfont{\textbf{(A\arabic{*})}}, ref=\textbf{(A\arabic{*})}]
	\it
	\item\label{hyp:ctrlMKV-drift-vol} The drift function $b$ is linear in $x, \mu$ and $\alpha$, and the volatility function $\sigma$ is linear in $x$ and $\mu$. To wit, for all $(t,x,\mu,\alpha) \in [0,T] \times \RR^d \times \cP_2(\RR^d) \times A$, we assume that
	\begin{align*}
		& b(t,x,\mu,\alpha) = b_0(t) + b_1(t)x + \bar{b}_1(t) \bar{\mu} + b_2(t) \alpha,
		\\
		& \sigma(t,x,\mu) = \sigma_0(t) + \sigma_1(t)x + \bar{\sigma}_1(t) \bar{\mu},
	\end{align*}
	for some bounded Lipschitz deterministic functions $b_0, b_1, \bar{b}_1$ and $b_2$ with values in $\RR^d, \RR^{d \times d}, \RR^{d \times d}$ and $\RR^{d \times k}$, and $\sigma_0, \sigma_1$ and $\bar{\sigma}_1$ with values in $\RR^{d \times d}, \RR^{(d \times d) \times d}$ and $\RR^{(d \times d) \times d}$ (the parentheses around $d \times d$ indicate that, for example, $\sigma_1(t)x$ is seen as an element of $\RR^{d \times d}$ whenever $x \in \RR^d$), and where we use the notation $\bar{\mu} = \int x d \mu(x)$ for the mean of a measure $\mu$.
	
	\begin{remark}
	\label{re:Lipschitz_coefs}
	Several technical estimates needed for the control of the errors due to time discretization will be proven in Appendix \ref{app:proof-discreteT} under the assumption that the drift and volatility functions are globally Lipschitz. Note that Assumption \ref{hyp:ctrlMKV-drift-vol} only guarantees Lipschitz continuity in time, locally in the other variables. However, as pointed out in Remark \ref{re:Lip_extension} the proofs of these technical estimates can easily be modified to accommodate Assumption \ref{hyp:ctrlMKV-drift-vol} because of its linear nature.
\end{remark}

	\item\label{hyp:pontryagin-opt-fg-Lip}  $f$ and $g$ satisfy the Assumption ``\textnormal{\textbf{Pontryagin Optimality}}'' from p. 542 of~\cite{MR3752669}, that is:
	\begin{itemize}\label{hyp:pontryagin-opt-fg}
		\item $f$ is differentiable with respect to $(x, \alpha)$, the mappings $(x,\mu,\alpha) \mapsto \partial_x f(t,x,\mu,\alpha)$ and $(x,\mu,\alpha) \mapsto \partial_\alpha f(t,x,\mu,\alpha)$ being continuous for each $t \in [0,T]$. The function $f$ is also differentiable with respect to the variable $\mu$, the mapping $\RR^d \times L^2(\Omega, \cF, \PP; \RR^d) \times A \ni (x, X, \alpha) \mapsto \partial_\mu f(t, x, \cL(X), \alpha)(X) \in L^2(\Omega, \cF, \PP; \RR^{d \times d} \times \RR^{(d \times d) \times d} \times \RR^d)$ being continuous for each $t \in [0,T]$. Similarly, the function $g$ is differentiable with respect to $x$, the mapping $(x,\mu) \mapsto \partial_x g(x,\mu)$ being continuous. The function $g$ is also differentiable with respect to the variable $\mu$, the mapping $\RR^d \times L^2(\Omega, \cF, \PP; \RR^d) \ni (x,X) \mapsto \partial_\mu g(x, \cL(X))(X) \in L^2(\Omega, \cF, \PP; \RR^d)$ being continuous.

		\item\label{hyp:ctrlMKV-fg-Lip} The function $[0,T] \ni t \mapsto f(t, 0, \delta_0, 0)$ is uniformly bounded. There exists a constant $L$ such that, for any $R \geq 0$ and any $(t,x,\mu,\alpha)$ such that $|x|, M_2(\mu), |\alpha| \leq R$, $|\partial_x f(t,x,\mu,\alpha)|$, $|\partial_x g(x,\mu)|$, and $|\partial_\alpha f(t,x,\mu,\alpha)|$ are bounded by $L(1+R)$ and the $L^2(\RR^d, \mu; \RR^d)$-norms of $x' \mapsto \partial_\mu f(t,x,\mu,\alpha)(x')$ and $x' \mapsto \partial_\mu g(x,\mu)(x')$ are bounded by $L(1+R)$.
	\end{itemize}
	In particular, for all $(t,x,\mu,\alpha) \in [0,T] \times \RR^d \times \cP_2(\RR^d) \times A$,
	\begin{align*}
		&|f(t,x',\mu',\alpha') - f(t,x,\mu,\alpha)| + |g(x',\mu') - g(x,\mu)|
		\\
		&\leq L \left[ 1 + |x'| + |x| + |\alpha'| + |\alpha| + M_2(\mu) + M_2(\mu')\right] 
		\\
		&\qquad\times \left[|(x', \alpha') - (x, \alpha)| + W_2(\mu', \mu)\right] .
	\end{align*}

	\item\label{hyp:ctrlMKV-dmu-fg} The derivatives of $f$ and $g$ with respect to $(x,\alpha)$ and $x$ respectively are $L$-Lipschitz continuous with respect to $(x,\alpha,\mu)$ and $(x,\mu)$ respectively, the Lipschitz continuity in the variable $\mu$ being understood in the sense of the $2$-Wasserstein distance. Moreover, for any $t \in [0,T]$, any $x, x' \in \RR^d$, any $\alpha, \alpha' \in \RR^k$, any $\mu, \mu' \in \cP_2(\RR^d)$, and any $\RR^d$-valued random variables $X$ and $X'$ having $\mu$ and $\mu'$ as distributions,
	\begin{align*}
		&\EE\left[| \partial_\mu f(t, x', \mu', \alpha')(X') - \partial_\mu f(t, x, \mu, \alpha)(X) |^2\right]
			\\
			&\leq L\left(|(x',\alpha') - (x,\alpha)|^2 +\EE\left[|X' - X|^2\right] \right),
	\end{align*}
	and
	\begin{align*}
		&\EE\left[| \partial_\mu g(x', \mu')(X') - \partial_\mu g(x, \mu)(X) |^2\right]
			\leq L\left(|x' - x|^2 +\EE\left[|X' - X|^2\right] \right).
	\end{align*}
	\item\label{hyp:ctrlMKV-Lconvex} The function $f$ satisfies the L-convexity property: there is a positive constant $\lambda$ s.t.
	\begin{align*}
		&f(t, x', \mu', \alpha') - f(t, x, \mu, \alpha) - \partial_{(x,\alpha)} f(t, x, \mu, \alpha) \cdot (x'-x, \alpha'-\alpha)
		\\
		& \qquad - \EE\left[\partial_\mu f(t, x, \mu, \alpha)(X) \cdot (X' - X) \right]
		\geq \lambda |\alpha' - \alpha|^2,
	\end{align*}
	for $t \in [0,T]$, $(x,\mu,\alpha) \in \RR^d \times \cP_2(\RR^d) \times A$ and  $(x',\mu',\alpha') \in \RR^d \times \cP_2(\RR^d) \times A$, whenever $X,X' \in L^2(\Omega, \cF, \PP; \RR^d)$ with distributions $\mu$ and $\mu'$ respectively. The function $g$ is also assumed to be L-convex in $(x,\mu)$.
\end{enumerate}
In order to obtain Lipschitz continuity in time of the decoupling field, we will partially strengthen these assumptions and assume, on top of that, the following extra assumption.
\begin{enumerate}[label=\normalfont{\textbf{(B\arabic{*})}}, ref=\textbf{(B\arabic{*})}]
	\it
	\item\label{hyp:pontryagin-opt-fg-extra} For every $\mu \in \cP_2(\RR^d)$ and $\alpha \in A$, the mapping $(t,x) \mapsto \partial_\alpha f(t,x,\mu,\alpha)$ is Lipschitz continuous with a Lipschitz constant independent of $\mu$ and $\alpha$.
	\setcounter{counterExtraAssumptions}{\value{enumi}}
\end{enumerate}

Under these conditions, the minimizer $\hat\alpha$ exists, is unique, and is a regular function of its arguments. See \cite[Theorem 6.19]{MR3752669} and the FBSDE system \eqref{eq:MKV-FBSDE} is well posed, which guarantees the existence of the master field $\cU$. See \cite[Lemma 6.25]{MR3752669}.

\begin{lemma}
\label{lem:hatalpha-Lip}
The function $(t,x,\mu,y) \mapsto \hat\alpha(t,x,\mu,y)$ is Lipschitz continuous, with a Lipschitz constant depending only on the data of the problem.
\end{lemma}

\begin{proof}
From~\cite[Lemma 6.18]{MR3752669}, the map $(t,x,\mu,y)\mapsto \hat\alpha(t,x,\mu,y)$ is measurable, locally bounded and Lipschitz continuous with respect to $(x,\mu,y)$ uniformly in $t \in [0,T]$, the Lipschitz constant depending only upon $\lambda$, the supremum norm of $b_2$ and the Lipschitz constant of $\partial_\alpha f$ in $(x,\mu)$. To complete the proof of the claim, we prove that it is also Lipschitz continuous in $t$.
	We borrow the argument of the proof of~\cite[Lemma 3.3]{MR3752669} (which itself relies on a suitable adaptation of the implicit function theorem to variational inequalities driven by coercive functionals).
	For a given $(t, x, \mu, y)$, $\alpha \mapsto \tilde H(t, x, \mu, y, \alpha)$ is once continuously differentiable and strictly convex so that $\hat \alpha(t, x, \mu, y)$ appears as the unique solution of the variational inequality (with unknown $\alpha$):
	$$
		\forall \beta \in A, \qquad (\beta - \alpha) \cdot \partial_\alpha \tilde H(t, x, \mu, y, \alpha) \geq 0.
	$$
	Consider $\theta = (t, x, \mu, y)$ and $\theta' = (t', x', \mu, y')$ in $[0,T] \times \RR^d \times \cP_2(\RR^d) \times \RR^d$. To alleviate the notations, let us write $a = \hat \alpha(t, x, \mu, y) = \hat \alpha(\theta)$ and $a' = \hat \alpha(\theta')$. The above inequality yields
	\begin{align*}
		(a' - a) \cdot \partial_\alpha \tilde H(\theta, a) \geq 0,
		\qquad
		(a' - a) \cdot \partial_\alpha \tilde H(\theta', a') \leq 0.
	\end{align*}
	Combining these inequalities gives
	$$
		(a' - a) \cdot \partial_\alpha \tilde H(\theta', a')
		\leq 
		(a' - a) \cdot \partial_\alpha \tilde H(\theta, a)
	$$
	hence
	\begin{equation}
	\label{eq:alpha-lip-tmp1}
		(a' - a) \cdot \left[\partial_\alpha \tilde H(\theta, a') - \partial_\alpha \tilde H(\theta', a') \right]
		\geq 
		(a' - a) \cdot \left[\partial_\alpha \tilde H(\theta, a') - \partial_\alpha \tilde H(\theta, a) \right].
	\end{equation}
	Moreover, by Assumption~\ref{hyp:ctrlMKV-Lconvex},
	\begin{align*}
		&f(t, x, \mu, a') - f(t, x, \mu, a) - \partial_{\alpha} f(t, x, \mu, a) \cdot (a'-a)
		\geq \lambda |a' - a|^2.
	\end{align*}
	Exchanging the role of $a$ and $a'$ in the above inequality and summing the resulting inequalities, we deduce that
	\begin{equation}
	\label{eq:alpha-lip-tmp2}
		2\lambda|a' - a|^2
		\leq 
		(a' - a) \cdot \big(\partial_\alpha f(t,x,\mu,a') - \partial_\alpha f(t,x,\mu,a)\big).
	\end{equation}
	Using~\eqref{eq:alpha-lip-tmp1} and~\eqref{eq:alpha-lip-tmp2}, we obtain:
	\begin{align*}
		&2\lambda|a' - a|^2
		\\
		&\leq
		|a' - a| \cdot \big(\partial_\alpha f(t,x,\mu,a') - \partial_\alpha f(t,x,\mu,a)\big)
		\\
		&= 
		|a' - a| \cdot \big(\partial_\alpha \tilde H(\theta,a') - \partial_\alpha \tilde H(\theta,a)\big)
		\\
		&\leq
		|a' - a| \cdot \left[\partial_\alpha \tilde H(\theta, a') - \partial_\alpha \tilde H(\theta', a') \right]
		\\
		&= |a' - a| \cdot \big(\partial_\alpha f(t,x,\mu,a') - \partial_\alpha f(t',x',\mu,a')\big)
		+ |a' - a| \cdot \big(b_2(t)y - b_2(t')y'\big).
	\end{align*}
	The last expression can be bounded as follows, for a constant $C$ depending only on the data of the problem,
	\begin{itemize}
		\item by Assumption~\ref{hyp:ctrlMKV-drift-vol}, $b_2$ is bounded, hence $|b_2(t)y - b_2(t')y'| \leq C |y-y'|$.
		\item by Assumption~\ref{hyp:pontryagin-opt-fg-extra}, 
		$$
			\big(\partial_\alpha f(t,x,\mu,a') - \partial_\alpha f(t',x',\mu,a')\big)
			\leq
			C \left(|t-t'| + |x-x'|\right).
		$$
	\end{itemize}
	Thus
	\begin{align*}
		2\lambda|a' - a|^2
		&\leq
		C |a' - a| \left( |t-t'| + |x-x'| + |y-y'| \right),
	\end{align*}
	which yields the conclusion.
\end{proof}

\vskip 2pt\noindent
We further assume the following regularity of the decoupling field.
\begin{enumerate}[label=\normalfont{\textbf{(B\arabic{*})}}, ref=\textbf{(B\arabic{*})}]
	\setcounter{enumi}{\value{counterExtraAssumptions}}
	\it
	\item\label{hyp:pontryagin-opt-fg-extra-V} The mapping $(t,x) \mapsto V(t,x)$ is Lipschitz continuous.
	\setcounter{counterExtraAssumptions}{\value{enumi}}
\end{enumerate}
The regularity in $x$ is rather standard, but the Lipschitz continuity in time is less so. However, we note that assumptions~\ref{hyp:ctrlMKV-drift-vol}--\ref{hyp:ctrlMKV-Lconvex}, as well as the result of Lemma~\ref{lem:hatalpha-Lip} and the Lipschitz continuity of $V$ given by assumption~\ref{hyp:pontryagin-opt-fg-extra-V} are satisfied for instance under the assumptions \textit{\textbf{(H6)}(i)--(iii)} of Chassagneux, Crisan and Delarue~\cite{ChassagneuxCrisanDelarue_Master} or the assumptions of \textit{\textbf{Theorem 2.4.2.}} of Cardaliaguet, Delarue, Lasry and Lions~\cite{CDLL} .

\vskip 2pt\noindent
Our next assumption concerns the initial distribution.
\begin{enumerate}[label=\normalfont{\textbf{(B\arabic{*})}}, ref=\textbf{(B\arabic{*})}]
	\setcounter{enumi}{\value{counterExtraAssumptions}}
	\it
	\item\label{hyp:pontryagin-opt-fg-extra-mu0} The initial distribution $\mu_0$ is in $\cP_4(\RR^d)$, i.e., there exists a (finite) constant $C_{\mu_0}$ such that
	$$
		M_4(\mu_0) = \int_{\RR^d} |x|^4 d \mu_0(x) \le C_{\mu_0}.
	$$
\end{enumerate}
This mild assumption is useful to obtain well-posedness of the MKV SDE and the interacting particle system, as well as stability estimates.

\vskip 2pt
As explained in Remark \ref{re:C2+C3}, our analysis of the error due to the discretization of time requires extra assumptions.
\begin{enumerate}[label=\normalfont{\textbf{(C\arabic{*})}}, ref=\textbf{(C\arabic{*})}]
	\it
	\item\label{hyp:euler-extra-assumption-b-sigma-f} Denoting $\Theta = (t,x,\mu,\alpha)$ where $t \in [0,T],$ $x \in \RR^d$, $\mu \in \cP_2(\RR^d)$ and $\alpha \in A$, there holds:  $\Theta \mapsto |\partial_t f\bigl(\Theta\bigr)|$ has at most quadratic growth, $\Theta \mapsto |\partial_\alpha f\bigl(\Theta\bigr)|$, $\Theta \mapsto |\partial_x f\bigl(\Theta\bigr)|$, $\Theta \mapsto |\partial_x f\bigl(\Theta\bigr)|$ and $(\Theta,x) \mapsto |\partial_{\mu} f\bigl(\Theta\bigr)(x)|$ have at most linear growth, and 
	$\Theta \mapsto |\partial^2_{xx}f\bigl(\Theta\bigr)|$,	$\Theta \mapsto |\partial^2_{x\alpha}f\bigl(\Theta\bigr)|$, $\Theta \mapsto |\partial^2_{\alpha\alpha} f\bigl(\Theta\bigr)|$,  $(\Theta,x) \mapsto |\partial_{x}\partial_\mu f\bigl(\Theta\bigr)(x)|$, $(\Theta ,x) \mapsto |\partial_{\alpha}\partial_\mu f\bigl(\Theta\bigr)(x)|$, $(\Theta, x) \mapsto |\partial_v\partial_{\mu}f\bigl(\Theta\bigr)(v)|_{v=x}|$, $(\Theta, x) \mapsto |\partial^2_{\mu}f\bigl(\Theta\bigr)(x,x)|$ are bounded by a constant.
	\item\label{hyp:euler-extra-assumption-gradV} $V$ is twice differentiable with respect to $x$ and the mapping $(t,x) \mapsto \partial^2_{x,x} V(t,x)$ is Lipschitz continuous.
	\item\label{hyp:euler-extra-assumption-gradAlpha} $\hat\alpha$ is twice differentiable with respect to $x,y$ and $\partial^2_{x,x} \hat\alpha$, $\partial^2_{y,y} \hat\alpha$ and $\partial^2_{x,y} \hat\alpha$ are Lipschitz continuous functions of $(t,x,\mu,y)$. 
\end{enumerate}
The last two assumptions above ensure that the second order derivative with respect to $x$ of the optimal control $\hat\bv$ is Lipschitz continuous. This fact is used in the analysis of one of the error terms induced by the Euler scheme introduced for the purpose of time discretization.

\vskip 2pt
Again, unless otherwise specified, the constants appearing in the proofs depend only on the data of the problem ($T$, $\mu_0$ and the constants appearing in the above assumptions), and $C$ denotes a generic constant whose value might change from one line to the next.

\section{Proof of Proposition~4}
\label{sec:proof-sec-approx-Nagents-feedback}

The result follows from a slight modification of the proof of~\cite[Theorem 6.17]{MR3753660}. Note that our standing assumptions ensure that the assumption ``\textnormal{\textbf{Control of MKV Dynamics}}'' from~\cite[p. 555]{MR3752669} holds, that $\sigma$ does not depend on the control, and that $\mu_0 \in \cP_4(\RR^d)$. We can thus reuse directly the last inequality of the proof of~\cite[Theorem 6.17]{MR3753660}, and we obtain:
	\begin{equation}
	\label{eq:last-ineq-thm617}
		J^N(\hat \balpha^{N})
		\leq
		J + \epsilon_1(N)
	\end{equation}
	where $J = \inf_{\balpha \in \AA} J(\balpha)$ is the objective function minimized over all admissible open loop controls, and $\hat{\balpha}^N$ is the distributed control given in feedback form by~\eqref{eq:hat-alpha-feedbackform}, namely
	\begin{equation}
	\label{eq:def-hatbalpha-feedback}
		\hat \alpha^{N,i}_t = \hat v (t, X^{N,i}_t)  = \hat \alpha\left( t, X^{N,i}_t, \mu_t, V(t, X^{N,i}_t) \right) \, .
	\end{equation}
	Here and throughout, we make a slight abuse of notation writing $J^N(\balpha)$ for $J^N(\bv)$ when the admissible control $\balpha\in\AA$ is given by a feedback function $v$ such that $\bv\in\VV$.
	Notice that thanks to the symmetry between the agents, $J^N(\hat \bv)$ defined by~\eqref{eq:def-JN-closedloop} can be viewed as the cost of a typical player. In other words, for each $j \in \{1,\dots,N\}$,
	$$
		J^N(\balpha) = J^{N,j}(\balpha)
	$$
	where 
	$$
		J^{N,j}(\balpha) = \EE\left[\int_0^T f(X^j_t, \mu^N_t, \alpha^j_t)dt + g(X^j_T,\mu^N_T)\right]
	$$
	 under the constraint: 
	 \begin{equation*}
		dX^i_t = b(t, X^i_t, \mu^N_t, \alpha^i_t) dt + \sigma(t, X^i_t, \mu^N_t) dW^i_t \, , \quad t \ge 0, \, i \in \{1,\dots,N\} \, ,
	\end{equation*}	
where the $(X^i_0)_{i \in \{1, \dots, N\}}$ are i.i.d. with distribution $\mu_0$, also independent of the Wiener processes $\bW^i, i=1, \dots, N$. Since~\eqref{eq:def-hatbalpha-feedback} says that $\hat \alpha^{N,i}_t$ is given by an admissible feedback function, 
	$$
		J^N(\hat \balpha^{N}) \geq \inf_{\bv \in \VV} J^N(\bv).
	$$
	As a consequence, inequality~\eqref{eq:last-ineq-thm617} yields
	$$
		\inf_{\bv \in \VV} J^N(\bv)
		\leq
		\inf_{\balpha \in \AA} J(\balpha) + \epsilon_1(N),
	$$
	which concludes the proof.

\section{Proofs of the Results of Section~\ref{SEC:APPROX-NN}}
\label{app:proof-approx-NN}

We turn our attention to the proof of Proposition~\ref{lem:var-J-feedback-compact}. 

We will repeatedly make use of the following fact: if $\Phi: \RR \times \RR^d \times \cP_2(\RR^d) \times A \to \RR$ is a (globally) Lipschitz function in the sense that there is a constant $L$ such that
	\begin{align*}
		&|\Phi(t', x',\mu',\alpha') - \Phi(t,x,\mu,\alpha)|
		\leq L \left[|(t', x', \alpha') - (t, x, \alpha)| + W_2(\mu', \mu)\right],
	\end{align*}
then, there is a constant $C$ depending only on $L$ such that for any $ (x^1,\dots,x^N)$ and $(y^1,\dots,y^N)$ in $ (\RR^d)^N$, and $(\alpha^1,\dots,\alpha^N)$ and $(\beta^1,\dots,\beta^N)$ in $A^N$,
\begin{equation}
\label{eq:useful-ineq-averageL2dist-W2}
\begin{split}
	&\frac{1}{N} \sum_{i=1}^N |\Phi(t, y^i, \nu^N,\beta^i) - \Phi(t, x^i, \mu^N,\alpha^i)|^2
	\\
	&\le
	C\left(|t' - t|^2 +  \frac{1}{N}\sum_{i=1}^N |y^i - x^i|^2 + \frac{1}{N}\sum_{i=1}^N |\beta^i - \alpha^i|^2  \right)
\end{split}
\end{equation}
where $\mu^N = \frac{1}{N} \sum_{i=1}^N \delta_{x^i}$  and $\nu^N = \frac{1}{N} \sum_{i=1}^N \delta_{y^i}$ are the empirical distributions corresponding to $\underline x, \underline y$. This remark uses the fact that $W_2(\mu^N, \nu^N) \le \frac{1}{N}\sum_{i=1}^N |y^i - x^i|^2$.

\begin{proof}[Proof of Proposition~\ref{lem:var-J-feedback-compact}]
	We start by bounding the probability that a particle of the interacting system exits a bounded domain.
If $\bv \in \VV$ is a Lipschitz continuous feedback control we denote by  $(\bX^{i,\bv})_{i=1,\dots,N}$ the solution of the state system~\eqref{eq:evolXi-closedloop} controlled by $\bv$, and since we assume that $\int|x|^4\mu_0(dx)<\infty$, a standard stability estimate for the solutions of \eqref{eq:evolXi-closedloop} yields that there exists a constant $C$ depending only on the data of the problem and on the Lipschitz constant of the controls $\bv, \bw$, such that, for $\bvarphi \in \{\bv, \bw\}$,
\begin{equation}
\label{fo:stability}
	\EE \Bigl[ \sup_{ t \in[0, T]} |X^{i,\bvarphi}_t|^2 \Bigr] \leq C, \qquad \EE \Bigl[ \sup_{ t \in[0, T]} |X^{i,\bvarphi}_t|^4 \Bigr] \leq C.
\end{equation}
See for example \cite{MR1108185}. Consequently, Markov's inequality implies that for all $R>0$ we have:
	\begin{equation}
	\label{eq:bound-proba-eventER}
		\PP\left[ \bar{\cE}^{i,\bv}_R \right] \leq \frac{C}{R} \, .
	\end{equation}
where for each $i\in\{1,\dots,N\}$,  $\bar{\cE}^{i,\bv}_R$ denotes the complement of ${\cE}^{i,\bv}_R$ defined by:
	\begin{equation}
	\label{eq:def-event-cEivR}
		\cE^{i,\bv}_R = \Bigl\{ \sup_{ t \in[0, T]} |X^{i,\bv}_t| \leq R \Bigr\}.
	\end{equation}
Now let $R>0$ and $\Gamma>0$ be constants, and let $\bw$ and $\bv$ be Lipschitz continuous feedback controls such that
	\begin{equation}
	\label{eq:varJN-bnd-bv-bw-proof}
		\| \bv_{|\bar{B}_d(0,R)} - \bw_{|\bar{B}_d(0,R)} \|_{\cC^0(\bar{B}_d(0,R))} \leq \Gamma \, .
	\end{equation}
We assume that for each $i$, $X^{i,\bw}_0 = X^{i,\bv}_0 = \xi_0^i$ a.s., with the random variables $\xi_0^i, i=1,\dots,N,$ being  i.i.d. with common distribution $\mu_0$. 
We will use the notation $\mu^{N,\bw}_t = \frac{1}{N} \sum_{i=1}^N \delta_{X^{i,\bw}_t}$ and $\mu^{N,\bv}_t = \frac{1}{N} \sum_{i=1}^N \delta_{X^{i,\bv}_t}$ for the empirical distributions. 
	
	Below, unless otherwise specified, $C$ denotes a generic constant whose value may change from one line to the next, but which in any case depends only on the data of the problem and possibly on the Lipschitz constants of $\bv$ and $\bw$ as well as $v(0,0), w(0,0)$; in particular, $C$ is always independent of $\Gamma$ and $R$.
	
\noindent
\textbf{Step 1:} we first control $|\bX^{i,\bw} - \bX^{i,\bv}|$ in terms of $\Gamma$ and $\PP[\bar{\cE}^{i,\bv}_R]$ which, by~\eqref{eq:bound-proba-eventER}, can be made as small as desired by increasing $R$ as needed.

Notice that, by~\eqref{eq:varJN-bnd-bv-bw-proof} and the Lipschitz continuity of $\bw$ and $\bv$ we have:
\begin{equation*}
\begin{aligned}
	&|w(t,X^{i,\bw}_t) - v(t,X^{i,\bv}_t)|
	\\
	&\leq
	\indic_{\cE^{i,\bv}_R} \left[ |w(t,X^{i,\bw}_t) - w(t,X^{i,\bv}_t)| + |w(t,X^{i,\bv}_t) - v(t,X^{i,\bv}_t)|\right] 
	\\
	&\qquad 
	+ \indic_{\bar \cE^{i,\bv}_R} \left[ |w(t,X^{i,\bw}_t)| + |v(t,X^{i,\bv}_t)| \right]
	\\
	&\leq
	C \left\{ \indic_{\cE^{i,\bv}_R} \left[ |X^{i,\bw}_t - X^{i,\bv}_t| + \Gamma \right] 
		+ \indic_{\bar \cE^{i,\bv}_R} \left[ |X^{i,\bw}_t| + |X^{i,\bv}_t| + 1 \right] \right\}
	\\
	&\leq
	C \left\{ \left[ |X^{i,\bw}_t - X^{i,\bv}_t| + \Gamma \right] 
		+ \indic_{\bar \cE^{i,\bv}_R} \left[ |X^{i,\bw}_t| + |X^{i,\bv}_t| + 1 \right] \right\}.
\end{aligned}
\end{equation*}
As a result:
\begin{equation}
\label{eq:varJ-bdd-wXw-vWv-2}
\begin{aligned}
	&\EE\left[\frac{1}{N}\sum_{i=1}^N |w(t,X^{i,\bw}_t) - v(t,X^{i,\bv}_t)|^2 \right]
	\\
	&\leq
	C \left\{ \EE\left[\frac{1}{N}\sum_{i=1}^N  |X^{i,\bw}_t - X^{i,\bv}_t|^2 \right] + \Gamma^2 
	\right.
	\\
	& \quad\left.
		+ \sqrt{\frac{1}{N}\sum_{i=1}^N \PP(\bar{\cE}^{i,\bv}_R)} \left( \EE\sqrt{\frac{1}{N}\sum_{i=1}^N |X^{i,\bw}_t|^4 } + \EE\sqrt{\frac{1}{N}\sum_{i=1}^N |X^{i,\bv}_t|^4 } + 1 \right) \right\}
	\\
	&\leq
	C \left\{ \EE\left[\frac{1}{N}\sum_{i=1}^N  |X^{i,\bw}_t - X^{i,\bv}_t|^2 \right] + \Gamma^2  + \frac{1}{R} \right\}
\end{aligned}
\end{equation}
where we used \eqref{fo:stability}.
To alleviate the notations, we introduce $\xi^{i,N,\bvarphi}_t = (t, X^{i,\bvarphi}_t, \mu^{N,\bvarphi}_t, \varphi(t,X^{i,\bvarphi}_t))$ for $\bvarphi \in \{\bv, \bw\}$. 
Using the form of the dynamics, we have for all $r \in [0,T]$,
\begin{equation}
\label{eq:bnd-Xw-Xv-1}
\begin{aligned}
	&\frac{1}{N} \sum_{i=1}^N \EE \sup_{0 \leq s \leq r} |X^{i,\bw}_s - X^{i,\bv}_s|^2
	\\
	&\leq
	C \left(
	\frac{1}{N} \sum_{i=1}^N \EE \sup_{0 \leq s \leq r} \left|\int_0^s 
	\left[b(\xi^{i,N,\bw}_t) - b(\xi^{i,N,\bv}_t)  \right] 
	dt \right|^2 
	\right.
	\\
	& 
	\left. 
	\qquad+ 
	 \frac{1}{N} \sum_{i=1}^N \EE \sup_{0 \leq s \leq r} \left|\int_0^s \left[\sigma(t, X^{i,\bw}_t, \mu^{N,\bw}_t) - \sigma(t, X^{i,\bv}_t, \mu^{N,\bv}_t) \right] dW^i_t\right|^2 \right)\, .
\end{aligned}
\end{equation}
For the first term in the right-hand side, we have
\begin{align*}
	&\frac{1}{N} \sum_{i=1}^N \EE \sup_{0 \leq s \leq r} \left|\int_0^s 
	\left[b(\xi^{i,N,\bw}_t) - b(\xi^{i,N,\bv}_t)  \right]  dt \right|^2
	\\
	&
	\leq
	\frac{1}{N} \sum_{i=1}^N \EE \sup_{0 \leq s \leq r} s \int_0^s 
	\left[b(\xi^{i,N,\bw}_t) - b(\xi^{i,N,\bv}_t)  \right] ^2 dt
	\\
	&
	\leq
	r \int_0^r  \frac{1}{N} \sum_{i=1}^N \EE 
	\left[b(\xi^{i,N,\bw}_t) - b(\xi^{i,N,\bv}_t)  \right]^2 dt
\end{align*}
By the (global) Lipschitz continuity of $b$ and~\eqref{eq:useful-ineq-averageL2dist-W2}, so that using~\eqref{eq:varJ-bdd-wXw-vWv-2}, we can bound the quantity above by
\begin{align*}
	&C r \int_0^r  \left(\frac{1}{N}\sum_{i=1}^N \EE|X^{i,\bw}_t - X^{i,\bv}_t|^2 + \frac{1}{N}\sum_{i=1}^N \EE |w(t,X^{i,\bw}_t) - v(t,X^{i,\bv}_t)|^2\right) dt
	\\
	&
	\leq C r \int_0^r  \left(\frac{1}{N}\sum_{i=1}^N \EE \sup_{0 \leq s \leq t}|X^{i,\bw}_s - X^{i,\bv}_s|^2   \right) dt
	+ Cr^2\left(\Gamma^2 +  \frac{1}{R} \right)
	\\
	&
	\leq C \int_0^r  \left(\frac{1}{N}\sum_{i=1}^N \EE \sup_{0 \leq s \leq t}|X^{i,\bw}_s - X^{i,\bv}_s|^2   \right) dt
	+ C\left(\Gamma^2 + \frac{1}{R} \right).
\end{align*}
The term involving $\sigma$ in the right-hand side of~\eqref{eq:bnd-Xw-Xv-1} is estimated in the following way using Doob's maximal inequality: 
\begin{align*}
&\frac{1}{N}\sum_{i=1}^N \EE \sup_{0 \leq s \leq r} \left|\int_0^s \left[\sigma(t, X^{i,\bw}_t, \mu^{N,\bw}_t) - \sigma(t, X^{i,\bv}_t, \mu^{N,\bv}_t) \right] dW^i_t\right|^2
	\\
	&\hskip 55pt
\leq
	C
	\frac{1}{N}\sum_{i=1}^N \EE \int_0^r \left|\sigma(t, X^{i,\bw}_t, \mu^{N,\bw}_t) - \sigma(t, X^{i,\bv}_t, \mu^{N,\bv}_t) \right|^2 dt
	\\
	&\hskip 55pt
\leq
	C
	\int_0^r 
	\frac{1}{N}\sum_{i=1}^N \left( \EE \sup_{0 \leq s \leq t} |X^{i,\bw}_s - X^{i,\bv}_s|^2 \right) dt.
\end{align*}
Going back to~\eqref{eq:bnd-Xw-Xv-1}, we obtain that, for all $r \in [0,T]$,
\begin{align*}
	&\frac{1}{N}\sum_{i=1}^N\EE\bigl[ \sup_{0 \leq s \leq r} |X^{i,\bw}_s - X^{i,\bv}_s|^2\bigr]
	\\
	&\leq
	C \int_0^r \frac{1}{N}\sum_{i=1}^N \EE  \bigl[\sup_{0 \leq s \leq t} |X^{i,\bw}_s - X^{i,\bv}_s |^2 \bigr] dt 
		+ C\left( \Gamma^2 +  \frac{1}{R}\right).
\end{align*}
We conclude, by Gr{\"o}nwall's inequality, that there exists a constant $C$ depending on the Lipschitz constants of $\bv,\bw$ and $v(0,0), w(0,0)$, but not on $\Gamma$ nor $R$ such that
\begin{equation}
\label{eq:bdd-Xw-Xv-GammaE}
	\frac{1}{N}\sum_{i=1}^N\EE \Bigl[ \sup_{0 \leq t \leq T} |X^{i,\bw}_t - X^{i,\bv}_t|^2 \Bigr]
	\leq
	C\Bigl( \Gamma^2 +  \frac{1}{R}\bigr) \, .
\end{equation}

\textbf{Step 2:} we now show the desired bound on $\left | J^N(\bv) - J^N(\bw) \right |$.
 We have
 \begin{align*}
 &J^N(\bv) - J^N(\bw)
	\\
	&=
	\frac{1}{N}\sum_{i=1}^N \EE\left[\int_0^T 
	\left[f(\xi^{i,N,\bv}_t) - f(\xi^{i,N,\bw}_t)  \right] dt\right] 
	\\
	&\qquad 
	+ \frac{1}{N}\sum_{i=1}^N \EE\left[g(X^{i,\bv}_T,\mu^{N,\bv}_T) - g(X^{i,\bw}_T,\mu^{N,\bw}_T)\right]
 \end{align*}
 We first study the terminal cost. Using the local Lipschitz continuity of $g$ as articulated in the second bullet point of Assumption~\ref{hyp:ctrlMKV-fg-Lip} we get
\begin{align*}
 	&\frac{1}{N} \sum_{i=1}^N \EE |g(X^{i,\bv}_T,\mu^{N,\bv}_T) - g(X^{i,\bw}_T,\mu^{N,\bw}_T)|
	\\
	&\hskip 15pt
	\le \frac{L}{N} \sum_{i=1}^N \EE\Bigl[ \bigl(1+|X^{i,\bv}_T|+|X^{i,\bw}_T|+M_2(\mu^{N,\bv}_T)+M_2(\mu^{N,\bw}_T)\bigr)  \;  
	\\
	&\qquad\qquad \times
	\bigl(|X^{i,\bv}_T- X^{i,\bw}_T|+W_2(\mu^{N,\bv}_T,\mu^{N,\bw}_T)\bigr)\Bigr]
	\\
	&\hskip 15pt
	\le L\Bigl( \frac{1}{N} \sum_{i=1}^N \EE\Bigl[ \bigl|1+|X^{i,\bv}_T|+|X^{i,\bw}_T|+M_2(\mu^{N,\bv}_T)+M_2(\mu^{N,\bw}_T)\bigr|^2\Bigr]\Bigr)^{1/2}    
	\\
	&\qquad\qquad \times 
	\Bigl( \frac{1}{N} \sum_{i=1}^N \EE\Bigl[ \bigl(|X^{i,\bv}_T- X^{i,\bw}_T|+W_2(\mu^{N,\bv}_T,\mu^{N,\bw}_T)\bigr)^2\Bigr]\Bigr)^{1/2}
	\\
	&\hskip 15pt
	\le C \Bigl(1+ \frac{1}{N} \sum_{i=1}^N \EE[ |X^{i,\bv}_T|^2]+\frac{1}{N} \sum_{i=1}^N \EE[|X^{i,\bw}_T|^2]\Bigr)^{1/2}   
	\\
	&\qquad\qquad \times \Bigl( \frac{1}{N} \sum_{i=1}^N \EE[|X^{i,\bv}_T- X^{i,\bw}_T|^2]\Bigr)^{1/2}
	\\
	&\hskip 15pt
	\leq C \left( \Gamma^2 +  \frac{1}{R} \right)^{1/2}.
 \end{align*}
Next, we consider the variation of the running cost. Again, we use the local Lipschitz continuity of $f$ as articulated in the second bullet point of Assumption~\ref{hyp:ctrlMKV-fg-Lip}. We obtain:
\begin{equation}
\label{eq:cmp-w-v-taylor-f}
\begin{aligned}
	&\frac{1}{N}\sum_{i=1}^N \EE \left|
	f(t, X^{i,\bv}_t, \mu^{N,\bv}_t, v(t,X^{i,\bv}_t)) - f(t, X^{i,\bw}_t, \mu^{N,\bw}_t, w(t,X^{i,\bw}_t))
	\right|
	\\
	&
	\le \frac{L}{N} \sum_{i=1}^N \EE\Bigl[ 
	\Bigl(1+ \sum_{\varphi \in\{\bv, \bw\}} [|X^{i,\bvarphi}_t|+|v(t,X^{i,\bvarphi}_t)| +M_2(\mu^{N,\bvarphi}_t)] \Bigr)
	\\
	&
	\qquad \times \bigl(|X^{i,\bv}_t- X^{i,\bw}_t|+|v(t,X^{i,\bv}_t) - w(t,X^{i,\bw}_t)|+W_2(\mu^{N,\bv}_t,\mu^{N,\bw}_t)\bigr)\Bigr]\\
	&
	\le C\Bigl(1+ \frac{1}{N} \sum_{i=1}^N \EE[\sup_{0\le t\le T} |X^{i,\bv}_t|^2]+\frac{1}{N} \sum_{i=1}^N \EE[\sup_{0\le t\le T}|X^{i,\bw}_t|^2]\Bigr)^{1/2}    
	\\
	&\qquad \times \Bigl( \frac{1}{N} \sum_{i=1}^N \EE[\sup_{0\le t\le T}|X^{i,\bv}_t- X^{i,\bw}_t|^2]\Bigr)^{1/2}\\
	&
	\leq C \left( \Gamma^2 +  \frac{1}{R} \right)^{1/2},
\end{aligned}
\end{equation}
where proceeded as we did in the case of the terminal cost function $g$, using the Lipschitz and linear growth properties of $\bv$ and $\bw$, estimates \eqref{eq:varJ-bdd-wXw-vWv-2} and \eqref{eq:bdd-Xw-Xv-GammaE} .

\noindent
\textbf{Conclusion.} 
Putting together the estimates of Step 1 and Step 2, we conclude that
$$
	\left | J^N(\bv) - J^N(\bw) \right |
	\leq 
	C \left( \Gamma^2 +  \frac{1}{R} \right)^{1/2} .
$$
which is the desired conclusion.
\end{proof}

\section{Proofs of the Results of Section~\ref{SEC:APPROX-DISCRETETIME}}
\label{app:proof-discreteT}
We use freely the notations of Subsection \ref{SEC:APPROX-DISCRETETIME} and without any loss of generality, we assume that the random shocks $(\Delta \check W_n^i)_{i,n}$ of the discrete dynamics appearing in the statement of Problem~\ref{pb:discrete-MKV} are the increments of the Brownian motions appearing Problem~\ref{pb:MKV-Nagents}, to wit, for each $i,n$, we assume that
$$
	\Delta \check W_n^i = W^i_{t_{n+1}} - W^i_{t_{n}}.
$$
Recall that we assume that the feedback functions appearing in both Problems are the same Lipschitz function $\varphi$.

For the sake of convenience and later reference, we state and prove the following estimate.

\begin{lemma}
\label{lem:discreteT-estim-Xitn}
Assuming that the functions $b$ and $\sigma$ are Lipschitz in all their variables, for all $t_0$ and $t$  such that $0\le t_0\le t\le T$, we have:
\begin{equation}
	\label{eq:discreteT-estim-Xitn}
		\frac{1}{N} \sum_{i=1}^N \EE |X^{i}_t - X^{i}_{t_0}|^2 
		\leq
		C\Bigl(1+\frac{1}{N} \sum_{i=1}^N \EE [|X^{i}_{t_0}|^2]\Bigr) |t-t_0|,
\end{equation}
where the constant $C$ depends only on the data of the problem, possibly including $T$, the Lipschitz constant of the control $\varphi$ as well as the value $\varphi(0,0)$. In particular, $C$ is independent of $N$, $t_0$ and $t$.
\end{lemma}

\begin{remark}
\label{re:Lip_extension}
The proofs of this lemma and the proof of Lemma~\ref{lem:discreteT-strongerr} are given under the assumption that the drift and volatility functions are globally Lipschitz. Strictly speaking, this assumption is not implied by Assumption \ref{hyp:ctrlMKV-drift-vol} suggested at the beginning of Appendix \ref{ap:assumptions}. Still as explained in Remark \ref{re:Lipschitz_coefs} a minor modification of the proofs given below shows that the same estimates hold, just because of the linearity structure of Assumption \ref{hyp:ctrlMKV-drift-vol}.
\end{remark}

\begin{proof}
It\^o's formula gives
\begin{equation*}
\begin{split}
|X^i_t-X^i_{t_0}|^2&=\int_{t_0}^t2(X^i_s-X^i_{t_0})\cdot b(s, X^{i}_s, \mu^{N}_s, \varphi(s,X^{i}_s)) ds 
\\
&\qquad + 
\int_{t_0}^t\text{trace}[\sigma(s, X^{i}_s, \mu^{N}_s)\sigma(s, X^{i}_s, \mu^{N}_s)^*]ds
+M^i_t\\
&\le \int_{t_0}^t|X^i_s-X^i_{t_0}|^2ds +  \int_{t_0}^t|b(s, X^{i}_s, \mu^{N}_s, \varphi(s,X^{i}_s))|^2 ds 
\\
&\qquad + 
\int_{t_0}^t\text{trace}[\sigma(s, X^{i}_s, \mu^{N}_s)\sigma(s, X^{i}_s, \mu^{N}_s)^*]ds
+M^i_t\\
&\le \int_{t_0}^t|X^i_s-X^i_{t_0}|^2ds + C \int_{t_0}^t \bigl(1+|X^{i}_s|^2+W_2( \mu^{N}_s,\delta_0)^2\bigr)ds
+M^i_t\\
&\le \int_{t_0}^t|X^i_s-X^i_{t_0}|^2ds + C \int_{t_0}^t \Bigl(1+|X^{i}_s|^2+\frac1N\sum_{i=1}^N|X^i_t|^2\Bigr)ds
+M^i_t
\end{split}
\end{equation*}
where $(M^i_t)_{t\ge t_0}$ is the square integrable martingale
$$
M^i_t=\int_{t_0}^t(X^i_s-X^i_{t_0})^*\sigma(s, X^{i}_s, \mu^{N}_s)dW^i_s
$$
and where we used the fact that
$$
|b(s,x,\mu,\varphi(s,x))|^2 + |\sigma(s,x,\mu)|^2\le C(1+|x|^2+W_2(\mu,\delta_0)^2)
$$
implied by the fact that $b$, $\sigma$ and $\varphi$ are assumed to be globally Lipschitz. Note that while its value can change from one line to the next, the constant $C$ is always only dependent upon the data of the problem, possibly including $T$, and the Lipschitz feedback function $\varphi$.
Taking expectations, summing over $i\in\{1,\ldots,N\}$ and dividing by $N$ we obtain:
\begin{equation}
\label{fo:almost}
\begin{split}
	\theta_t &\le \int_{t_0}^t \theta_s ds +C \int_{t_0}^t \Bigl(1+\frac1N\sum_{i=1}^N\EE[|X^i_t|^2]\Bigr)ds
	\\
	&\le (1+C)\int_{t_0}^t \theta_s ds +C\Bigl(1+\frac{1}{N} \sum_{i=1}^N \EE [|X^{i}_{t_0}|^2]\Bigr)( t-t_0)
\end{split}
\end{equation}
if  we denote by $\theta_t$ the left hand side of \eqref{eq:discreteT-estim-Xitn}, and we conclude using Gronwall inequality. 
\end{proof}

\begin{proof}[Proof of Lemma~\ref{lem:discreteT-strongerr}]
For each $i\in\{1,\ldots,N\}$, and $t\in[t_n,t_{n+1}]$ we set $e^{i,n}_t=X^i_t-\check X^i_{t_n}$. It\^o's formula gives
\begin{equation}
\label{fo:Ito2}
\begin{split}
|e^{i,n}_t|^2&=|e^{i,n}_{t_n}|^2+\int_{t_n}^t2e^{i,n}_s\cdot b(s, X^{i}_s, \mu^{N}_s, \varphi(s,X^{i}_s)) ds 
\\
&\qquad + 
\int_{t_n}^t\text{trace}[\sigma(s, X^{i}_s, \mu^{N}_s)\sigma(s, X^{i}_s, \mu^{N}_s)^*]ds
+M^i_t\\
&\le |e^{i,n}_{t_n}|^2+\int_{t_n}^t|e^{i,n}_s|^2ds +  \int_{t_n}^t|b(s, X^{i}_s, \mu^{N}_s, \varphi(s,X^{i}_s))|^2 ds 
\\
&\qquad + 
\int_{t_n}^t |\sigma(s, X^{i}_s, \mu^{N}_s)|^2 ds
+M^i_t\\
&\le |e^{i,n}_{t_n}|^2+\int_{t_n}^t|e^{i,n}_s|^2ds + C \int_{t_n}^t \bigl(1+|X^{i}_s|^2+W_2( \mu^{N}_s,\delta_0)^2\bigr)ds
+M^i_t\\
&\le |e^{i,n}_{t_n}|^2+\int_{t_n}^t|e^{i,n}_s|^2ds + C \int_{t_n}^t \Bigl(1+|X^{i}_s|^2+\frac1N\sum_{i=1}^N|X^i_t|^2\Bigr)ds
+M^i_t
\end{split}
\end{equation}
where $(M^i_t)_{t\ge t_n}$ is the square integrable martingale
$$
M^i_t=\int_{t_n}^te^{i,n*}_s\sigma(s, X^{i}_s, \mu^{N}_s)dW^i_s
$$
and where we used the fact that
$$
|b(s,x,\mu,\varphi(s,x))|^2 + |\sigma(s,x,\mu)|^2\le C(1+|x|^2+W_2(\mu,\delta_0)^2)
$$
implied by the fact that $b$, $\sigma$ and $\varphi$ are assumed to be globally Lipschitz. Note that while its value can change from one line to the next, the constant $C$ is always only dependent upon the data of the problem, possibly including $T$, and the Lipschitz feedback function $\varphi$. If we set:
\begin{equation}
\label{fo:eta_t}
\eta_t=\frac1N\sum_{i=1}^N\EE[|X^i_t-\check X^i_{t_n}|^2],
\end{equation}
Taking expectations, summing over $i\in\{1,\ldots,N\}$ and dividing by $N$ on both ends of \eqref{fo:Ito2} we obtain: 
for $ t_n\le t\le t_{n+1}$,
\begin{equation}
\label{fo:almost}
\eta_t\le \eta_{t_n}+\int_{t_n}^t \eta_s ds +C \int_{t_n}^t \Bigl(1+\frac1N\sum_{i=1}^N\EE[|X^i_t|^2]\Bigr)ds
\le \eta_{t_n}+\int_{t_n}^t \eta_s ds +C \Delta t,
\end{equation}
where we used Lemma \ref{lem:discreteT-estim-Xitn}. Next, Gr\"{o}nwall inequality  implies that:
\begin{equation}
\label{fo:first_theta}
\eta_t\le (\eta_{t_n}+C)\Delta t \;e^{\Delta t},
\end{equation}
and in order to complete the proof, we only need to prove a uniform bound on $\eta_{t_n}$. Obviously:
	\begin{equation}
	\label{fo:first}
|e^i_{t_{n+1}}|^2\le C\Bigl(|e^i_{t_n}|^2+
\Bigl | \int_{t_n}^{t_{n+1}} [\check b^i_{t_n} - b^i_{t} ] dt \Bigr |^2 + 
 \Bigl| \int_{t_n}^{t_{n+1}} [\check \sigma^i_{t_n} - \sigma^i_{t} ] dW^i_t \Bigr|^2\Bigr)
	\end{equation}
	where we used the notation 
	\begin{align*}
	&\check b^i_{t_n} = b({t_n}, {\check X}^{i}_{t_n}, {\check \mu}^{N}_{t_n}, \varphi(t_n,{\check X}^{i}_{t_n})),
	\qquad
	b^i_{t} = b(t, X^{i}_t, \mu^{N}_t, \varphi(t,X^{i}_t))\\
	&\check \sigma^i_{t_n} = \sigma({t_n}, {\check X}^{i}_{t_n}, {\check \mu}^{N}_{t_n}),
	\qquad
	\sigma^i_{t} = \sigma(t, X^{i}_t, \mu^{N}_t).
	\end{align*}
Introducing the notation $\delta_n = \frac{1}{N} \sum_{i=1}^N \EE  |e^i_{t_n}|^2$, taking expectations, summing over $i$ and dividing by $N$ on both sides of \eqref{fo:first} we get:
	\begin{equation}
	\label{fo:second}
		\delta_{n+1}
		\leq  C\Bigl(\delta_{n}+
	\frac1N\sum_{i=1}^N\EE\Bigl[\Bigl | \int_{t_n}^{t_{n+1}} [\check b^i_{t_n} - b^i_{t} ] dt \Bigr |^2\Bigr] + 
	\frac1N\sum_{i=1}^N\EE\Bigl[ \Bigl| \int_{t_n}^{t_{n+1}} [\check \sigma^i_{t_n} - \sigma^i_{t} ] dW^i_t \Bigr|^2\Bigr)
	\end{equation}
We estimate the right hand side of \eqref{fo:second} in the following way.
\begin{equation}
\label{fo:1}
\begin{split}
&\frac1N\sum_{i=1}^N\EE\Bigl[\Bigl | \int_{t_n}^{t_{n+1}} [\check b^i_{t_n} - b^i_{t} ] dt \Bigr |^2\Bigr] \\
&
\le \Delta t \int_{t_n}^{t_{n+1}} \frac{1}{N}\sum_{i=1}^N \EE \left|b(t_n, {\check X}^{i}_{t_n}, {\check \mu}^{N}_{t_n}, \varphi(t_n,{\check X}^{i}_{t_n})) - b(t, X^{i}_t, \mu^{N}_t, \varphi(t,X^{i}_t))\right|^2 dt \\
	\\
&
	\leq \,
	C \Delta t \int_{t_n}^{t_{n+1}} \left(|t_n-t|^2 + \frac{1}{N}\sum_{i=1}^N \EE|{\check X}^{i}_{t_n} - X^{i}_t|^2 
	+ W_2({\check \mu}^{N}_{t_n}, \mu^{N}_t)^2
	\right.
	\\
	&\qquad\qquad\qquad\qquad \left.+\frac{1}{N}\sum_{i=1}^N \EE |\varphi(t_n,{\check X}^{i}_{t_n}) - \varphi(t,X^{i}_t)|^2\right) dt \\
	\\
&
	\leq \,
	C \Delta t \int_{t_n}^{t_{n+1}} \left(|\Delta t|^2 + \frac{1}{N}\sum_{i=1}^N \EE|{\check X}^{i}_{t_n} - X^{i}_t|^2 \right) dt \\
	\\
&
	\leq \,
	C \Delta t \int_{t_n}^{t_{n+1}} \left(|\Delta t|^2 + \frac{1}{N}\sum_{i=1}^N \EE|{\check X}^{i}_{t_n} - X^{i}_{t_n}|^2 + \frac{1}{N}\sum_{i=1}^N \EE|X^{i}_{t_n} - X^{i}_t|^2 \right) dt \\
	\\
&
	= \,
	C \Delta t\left(|\Delta t|^3 + \Delta t \, \delta_n + \frac{1}{N}\sum_{i=1}^N \EE \int_{t_n}^{t_{n+1}}  |X^{i}_{t_n} - X^{i}_t|^2 dt \right),
\end{split}
\end{equation}
where we used the (global) Lipschitz assumption on $b$ and on $\varphi$, and a simple upper bound on the Wasserstein distance between two empirical measures. The constant $C$, whose value can change from one line to the next, depends on the Lipschitz constant of $\varphi$. Similarly we estimate the third term in the right hand side of \eqref{fo:second} by
\begin{equation}
\label{fo:2}
\begin{split}
&\frac1N\sum_{i=1}^N\EE\Bigl[ \Bigl| \int_{t_n}^{t_{n+1}} [\check \sigma^i_{t_n} - \sigma^i_{t} ] dW^i_t \Bigr|^2\Bigr]
\\
=&\frac1N\sum_{i=1}^N\EE\Bigl[ \int_{t_n}^{t_{n+1}} |\check \sigma^i_{t_n} - \sigma^i_{t} |^2 dt\Bigr]
\\
&
\le \int_{t_n}^{t_{n+1}} \frac{1}{N}\sum_{i=1}^N \EE \bigl[|\sigma(t_n, {\check X}^{i}_{t_n}, {\check \mu}^{N}_{t_n}) - \sigma(t, X^{i}_t, \mu^{N}_t)|^2 \bigr]dt\\
&\leq \,
	C  \int_{t_n}^{t_{n+1}} \Bigl(|t_n-t|^2 + \frac{1}{N}\sum_{i=1}^N \EE|{\check X}^{i}_{t_n} - X^{i}_t|^2 
	+ W_2({\check \mu}^{N}_{t_n}, \mu^{N}_t)^2\Bigr)dt \\
&\leq \,
	C  \int_{t_n}^{t_{n+1}} \Bigl(|\Delta t|^2 + \frac{1}{N}\sum_{i=1}^N \EE|{\check X}^{i}_{t_n} - X^{i}_t|^2 \Bigr) dt \\
&
	= \,
	C \Bigl(|\Delta t|^3 + \Delta t \, \delta_n + \frac{1}{N}\sum_{i=1}^N \EE \int_{t_n}^{t_{n+1}}  |X^{i}_{t_n} - X^{i}_t|^2 dt \Bigr).
\end{split}
\end{equation}
Also, because $\bX_t=(X^i_t)_{i=1,\ldots,N}$ is the solution of an $\RR^{Nd}$ valued stochastic differential equation with Lipschitz coefficients, standard existence results imply the existence of a positive constant $K$ depending only on the data of the problem and the control $\varphi$ through its Lipschitz constant and $\varphi(0,0)$ satisfying
\begin{equation}
\label{fo:3}
	\frac{1}{N} \sum_{i=1}^N \EE \int_{t_n}^{t_{n+1}} |X^{i}_{t_n} - X^{i}_t|^2 dt
	\leq
	K |\Delta t|.
\end{equation}
Plugging \eqref{fo:1}, \eqref{fo:2} and \eqref{fo:3} into \eqref{fo:second} we get:
Hence
\begin{align*}
	\delta_{n+1}
		&\leq \delta_n C(1+  \Delta t)
		+
		\tilde K  \Delta t,
\end{align*}
for some constants $C$ and $\tilde K$ depending only on the data of the problem (including $T$) and the control $\varphi$ through its Lipschitz constant and $\varphi(0,0)$. Using the facts $\delta_0=0$ and $n\Delta t\le N_T\Delta t=T$, we conclude using a discrete version of Gr{\"o}nwall inequality.
\end{proof}
We now proceed to the proof of Proposition~\ref{prop:approx-discreteT}. 

\begin{proof}[Proof of Proposition~\ref{prop:approx-discreteT}]
We have
\begin{align*}
	&|J^N(\bvarphi) -  \check J^N(\bvarphi)|
	\leq 
	\delta_f + \delta_g,
\end{align*}
where
\begin{align*}
	&\delta_f = 
	\\
	&\Bigl|\sum_{n=0}^{N_T-1} \int_{t_n}^{t_n+1} 
	\frac{1}{N}\sum_{i=1}^N 
	\EE\left[   f(t, X^i_t, \mu^N_t, \varphi(t, X^i_t)) - f(t_n, \check X^i_{t_n}, \check \mu_{t_n}, \varphi(t_n, \check X^i_{t_n})) \right]dt\Bigr|
\end{align*}
and
$$
	\delta_g
	= \Bigl| \EE[G(\mu^N_T)-G(\check\mu^N_{T})]\Big|
$$
where the function $G$ is defined by $G(\mu)=\int_{\RR^d}g(x,\mu)\mu(dx)$.
Let us first analyze the difference $\delta_g$ which, for obvious notational reasons, is much easier to handle. 
For $\lambda\in[0,1]$ we set $X^i_\lambda=X^i_T+\lambda(\check X^i_T-X^i_T)$ for $i\in\{1,\ldots,N\}$, and $X_\lambda=(X^i_\lambda)_{i=1,\ldots,N}$ and $\mu_\lambda=\frac1N\sum_{i=1}^N\delta_{X^i_\lambda}$. Introducing
the empirical projection $G^N$ of $G$ defined on $\RR^{dN}$ by:
$$
G^N(x^1,\ldots,x^N)=G\Bigl(\frac1N\sum_{i=1}^N\delta_{x^i}\Bigr)
$$
(see \cite[Definition 5.34 p.399]{MR3752669}) we have:

\begin{equation*}
\begin{split}
G(\check\mu^N_T)-G(\mu^N_{T})&=G(\mu_1)-G(\mu_0)\\
&=G^N(X_1)-G^N(X_0)\\
&=\int_0^1\frac{d}{d\lambda}G^N(X_\lambda)d_\lambda\\
&=\int_0^1\sum_{i=1}^N\partial_{x^i}G^N(X_\lambda)\cdot(\check X^i_T-X^i_T)d\lambda\\
&=\frac1N\sum_{i=1}^N\Bigl(\int_0^1\partial_{\mu}G(\mu_\lambda)(X^i_\lambda)d\lambda\Bigr)\cdot(\check X^i_T-X^i_T)
\end{split}
\end{equation*}
where we used the expression of the partial derivatives of $G^N$ in terms of the L-derivative $\partial_\mu G$ as given in 
\cite[Proposition 5.35 p.399]{MR3752669}. So
$$
\delta_g\le \Bigl(\frac1N\sum_{i=1}^N \EE\Bigl[\Bigl|\int_0^1\partial_{\mu}G(\mu_\lambda)(X^i_\lambda)d\lambda\Bigr|^2\Bigr]\Bigr)^{1/2}
\Bigl(\frac1N\sum_{i=1}^N\EE[\check X^i_T-X^i_T|^2]\Bigr)^{1/2}.
$$
Lemma \ref{lem:discreteT-strongerr} implies that the second factor in the above right hand side is bounded from above by $C\sqrt{\Delta t}$, so we only need to prove that the first factor is bounded by a constant depending only upon the data of the problem and the Lipschitz constant of $\varphi$. In order to do so, we compute the L-derivative of $G$ as given in 
\cite[Example 3 p.386]{MR3752669}, and we use the fact that the partial derivatives $\partial_xg(x,\mu)$ and
$\partial_\mu g(x,\mu)(y)$ are of linear growth. We get:

\begin{equation*}
\begin{split}
&\frac1N\sum_{i=1}^N \EE\Bigl[\Bigl|\int_0^1\partial_{\mu}G(\mu_\lambda)(X^i_\lambda)d\lambda\Bigr|^2\Bigr]
\\
&=\frac1N\sum_{i=1}^N \EE\Bigl[\Bigl|\int_0^1\Bigl(\partial_{x}g(X^i_\lambda,\mu_\lambda)
+\frac1N\sum_{j=1}^N\partial_\mu g(X^j_\lambda,\mu_\lambda)(X^i_\lambda)\Bigr)d\lambda \Bigr|^2\Bigr]\\
&\le\frac CN\sum_{i=1}^N \EE\Bigl[\int_0^1\Bigl(|\partial_{x}g(X^i_\lambda,\mu_\lambda)|^2
+\frac1N\sum_{j=1}^N|\partial_\mu g(X^j_\lambda,\mu_\lambda)(X^i_\lambda)|^2\Bigr)d\lambda \Bigr]\\
&\le\frac CN\sum_{i=1}^N \EE\Bigl[\int_0^1|\partial_{x}g(X^i_\lambda,\mu_\lambda)|^2 d\lambda 
+\frac C{N^2} \EE\Bigl[\int_0^1\sum_{i,j=1}^N|\partial_\mu g(X^j_\lambda,\mu_\lambda)(X^i_\lambda)|^2 d\lambda \Bigr]\\
&\le\frac CN\sum_{i=1}^N \EE\Bigl[\int_0^1\Bigl(|X^i_\lambda|^2+W_2(\mu_\lambda,\delta_0)^2\Bigr)d\lambda  \Bigr]\\
&\le\frac CN\sum_{i=1}^N \int_0^1\EE[|X^i_\lambda|^2]d\lambda\\
&\le\frac CN\sum_{i=1}^N\EE[|X^i_T|^2] +\frac C{3N}\sum_{i=1}^N \EE[|\check X^i_T-X^i_T|^2]
\end{split}
\end{equation*}
which is bounded from above as desired.

\vskip 2pt
The control of $\delta_f$ is a little bit more involved. First we split $\delta_f$ as the sum $\delta_f=\delta^1_f+\delta^2_f$ with:
\begin{align*}
&\delta^1_f =
	\\
	&\Bigl|\sum_{n=0}^{N_T-1} \int_{t_n}^{t_{n+1}} 
	\frac{1}{N}\sum_{i=1}^N 
	\EE\left[   f(t, X^i_t, \mu^N_t, \varphi(t, X^i_t)) - f(t_n, X^i_{t_n}, \mu^N_{t_n}, \varphi(t_n,  X^i_{t_n})) \right]dt\Bigr|
\end{align*}
and
\begin{align*}
&\delta^2_f =
	\\
	&\Delta t\; \Bigl|\sum_{n=0}^{N_T-1} 	\frac{1}{N}\sum_{i=1}^N 
	\EE\left[   f(t_n, X^i_{t_n}, \mu^N_{t_n}, \varphi(t_n, X^i_{t_n})) - f(t_n, \check X^i_{t_n}, \check \mu_{t_n}, \varphi(t_n, \check X^i_{t_n})) \right]\Bigr|.
\end{align*}
The estimation of $\delta^1_f$ cannot be done following the strategy used in the previous appendix or in the estimation of $\delta^2_f$ below. Indeed, the quantity to control involves the difference of quantities evaluated at two different times, $t_n$ and $t$ to be specific. It is thus natural to use It\^o's formula to express this difference, hence our reliance on a stronger differentiability assumption for the running cost function $f$ and the feedback function $\varphi$. In order to rely on the classical It\^o's formula (as opposed to the chain rule for functions of marginal laws of diffusion processes as in \cite[Proposition 5.102 p485]{MR3752669} for example) we introduce the empirical projection $f^N$ defined by:
$$
f^N\bigl(t,x,(x^1,\cdots,x^N)\bigr)=f\bigl(t,x,\frac1N\sum_{i=1}^N\delta_{x^i},\varphi(t,x)\bigr).
$$
It\^o's formula gives:
\begin{equation*}
\begin{split}
&f(t, X^i_t, \mu^N_t, \varphi(t, X^i_t)) - f(t_n, X^i_{t_n}, \mu^N_{t_n}, \varphi(t_n,  X^i_{t_n}))\\
&\qquad
=f^N\bigl(t, X^i_t, (X^1_t,\cdots,X^N_t)\bigr) - f^N\bigl(t_n, X^i_{t_n}, (X^1_{t_n},\cdots,X^N_{t_n})\bigr)\\
&\qquad
=\int_{t_n}^t \partial_t f^N\bigl(s,X^i_s,(X^1_s,\cdots,X^N_s)\bigr) ds\\
&\qquad\qquad 
+\sum_{j=0}^N\int_{t_n}^t \partial_{x^j} f^N\bigl(s, X^i_s, (X^1_s,\cdots,X^N_s)\bigr) b\bigl(s,X^j_s,\mu^N_s,\varphi(s,X^j_s)\bigr)ds\\
&\qquad\qquad
+\frac12\sum_{h,k=0}^N\int_{t_n}^t \partial^2_{x^hx^k}f^N\bigl(s,X^i_s,(X^1_s,\cdots,X^N_s)\bigr)\;d[X^h,X^k]_s +M^i_t
\end{split}
\end{equation*}
where we used the convention $x^0=x$ and $X^0_s=X^i_s$, and where $(M^i_t)_{t_n\le t\le T}$ is a mean-zero square-integrable martingale. Formulas become quite lengthy when we try to go back to the function $f$ and its partial derivatives. So in order to shorten them, we introduce the following notations:
$$
\Theta^h_s=\bigl(s, X^h_s, \mu^N_s, \varphi(s, X^h_s)\bigr)
\qquad\text{and}\qquad
\Lambda^h_s=\bigl(s, X^h_s, (X^1_s,\cdots,X^N_s)\bigr).
$$
as well as the notation 
$
a(t,x,\mu)=\sigma(t,x,\mu)\sigma(t,x,\mu)^*
$
for the diffusion matrix. This gives:

\begin{equation*}
\begin{split}
&f(t, X^i_t, \mu^N_t, \varphi(t, X^i_t)) - f(t_n, X^i_{t_n}, \mu^N_{t_n}, \varphi(t_n,  X^i_{t_n}))\\
=&\int_{t_n}^t \Bigl[\partial_t f\bigl(\Theta^i_s\bigr) +\partial_\alpha f\bigl(\Theta^i_s\bigr)\partial_t\varphi(s,X^i_s)\Bigr]\, ds\\
&
+\int_{t_n}^t \Bigl[\partial_x f\bigl(\Theta^i_s\bigr)+\partial_\alpha f\bigl(\Theta^i_s\bigr)\partial_x\varphi(s,X^i_s)\Bigr]b\bigl(\Theta^i_s\bigr)ds
\\
&+\sum_{j=1}^N\int_{t_n}^t \partial_{x^j} f^N\bigl(\Lambda^i_s\bigr) b\bigl(\Theta^j_s\bigr)ds\\
&
+\frac12\int_{t_n}^t a(s,X^i_s,\mu^N_s)\bigl[\partial^2_{xx}f\bigl(\Theta^i_s\bigr)+2\partial^2_{x\alpha}f\bigl(\Theta^i_s\bigr)\partial_x\varphi(s,X^i_s)\\
&
+\partial_{x}\varphi(s,X^i_s)^*\partial^2_{\alpha\alpha} f\bigl(\Theta^i_s\bigr)\partial_{x}\varphi(s,X^i_s)+\partial_\alpha f\bigl(\Theta^i_s\bigr)\partial^2_{xx}\varphi(s,X^i_s)\bigr]ds\\
&
+\int_{t_n}^ta(s,X^i_s,\mu^N_s)\partial^2_{xx^i}f^N\bigl(\Lambda^i_s\bigr) \,ds
\\
&
+\frac12\int_{t_n}^t \sum_{j=1}^N a(s,X^j_s,\mu^N_s)\partial^2_{x^jx^j}f^N\bigl(\Lambda^i_s\bigr)\;ds +M^i_t
\end{split}
\end{equation*}
We now use \cite[Proposition 5.35 p399]{MR3752669} and \cite[Proposition 5.91 p471]{MR3752669} to express the partial derivatives of the empirical projection $f^N$ in terms of the partial derivatives of the original function $f$. We get:

\begin{equation}
\label{eq:integrals-delta-f-1}
\begin{split}
&f(t, X^i_t, \mu^N_t, \varphi(t, X^i_t)) - f(t_n, X^i_{t_n}, \mu^N_{t_n}, \varphi(t_n,  X^i_{t_n}))\\
=&\int_{t_n}^t \Bigl[\partial_t f\bigl(\Theta^i_s\bigr) +\partial_\alpha f\bigl(\Theta^i_s\bigr)\partial_t\varphi(s,X^i_s)\Bigr]\, ds\\
&
+\int_{t_n}^t \Bigl[\partial_x f\bigl(\Theta^i_s\bigr)+\partial_\alpha f\bigl(\Theta^i_s\bigr)\partial_x\varphi(s,X^i_s)\Bigr]b\bigl(\Theta^i_s\bigr)ds
\\
&+\frac1N\sum_{j=1}^N\int_{t_n}^t \partial_{\mu} f\bigl(\Theta^i_s\bigr)(X^j_s) b\bigl(\Theta^j_s\bigr)ds\\
&
+\frac12\int_{t_n}^t a(s,X^i_s,\mu^N_s)\bigl[\partial^2_{xx}f\bigl(\Theta^i_s\bigr)+2\partial^2_{x\alpha}f\bigl(\Theta^i_s\bigr)\partial_x\varphi(s,X^i_s)\\
&
+\partial_{x}\varphi(s,X^i_s)^*\partial^2_{\alpha\alpha} f\bigl(\Theta^i_s\bigr)\partial_{x}\varphi(s,X^i_s)+\partial_\alpha f\bigl(\Theta^i_s\bigr)\partial^2_{xx}\varphi(s,X^i_s)\bigr]ds\\
&
+\frac1N\int_{t_n}^ta(s,X^i_s,\mu^N_s)\bigl[\partial_{x}\partial_\mu f\bigl(\Theta^i_s\bigr)(X^i_s)
+\partial_{\alpha}\partial_\mu f\bigl(\Theta^i_s\bigr)(X^i_s)\partial_x\varphi(s,X^i_s)\bigr]\,ds
 \\
&
+\frac1{2N}\sum_{j=1}^N\int_{t_n}^t  a(s,X^j_s,\mu^N_s)\partial_v\partial_{\mu}f\bigl(\Theta^i_s\bigr)(v)|_{v=X^j_s}\;ds
\\
&+\frac1{2N^2}\sum_{j=1}^N\int_{t_n}^t  a(s,X^j_s,\mu^N_s)\partial^2_{\mu}f\bigl(\Theta^i_s\bigr)(X^j_s,X^j_s)\;ds 
+M^i_t.
\end{split}
\end{equation}

\vskip 6pt
We now bound the above term by $C \Delta t$. Recall that for the control $\varphi$, estimates~\eqref{eq:hyp-vaphi-Euler} hold. Moreover, from Assumption~\ref{hyp:euler-extra-assumption-b-sigma-f}, using the notation $|\Theta^j_s| = |s| + |X^j_s| + M_2(\mu^N_s) + |\varphi(s, X^j_s)|$, we have
for $b$ and $a$:
$$
	|b\bigl(\Theta^j_s\bigr)| \le C|\Theta^j_s|,
	 \qquad
	 |a(s,X^i_s,\mu^N_s)| \le C \left( |s|^2 + |X^i_s|^2 + M_2(\mu^N_s)^2 \right),
$$
and for the terms involving $f$, $|\partial_t f\bigl(\Theta^i_s\bigr)| \le C (1+|\Theta^i_s|^2)$, $|\partial_t f\bigl(\Theta^i_s\bigr)|$
and $|\partial_\alpha f\bigl(\Theta^i_s\bigr)|$ are bounded from above by  $C (1+|\Theta^i_s|)$, and $|\partial_{\mu} f\bigl(\Theta^i_s\bigr)(X^j_s)|$
by $C (1+|(\Theta^i_s, X^j_s)|)$
while $|\partial^2_{xx}f\bigl(\Theta^i_s\bigr)|$,  $|\partial^2_{x\alpha}f\bigl(\Theta^i_s\bigr)|$, $|\partial^2_{\alpha\alpha} f\bigl(\Theta^i_s\bigr)|$, $|\partial_{x}\partial_\mu f\bigl(\Theta^i_s\bigr)(X^i_s)|$, $|\partial_{\alpha}\partial_\mu f\bigl(\Theta^i_s\bigr)(X^i_s)|$, $|\partial_v\partial_{\mu}f\bigl(\Theta^i_s\bigr)(v)_{\big|v=X^j_s}|$ and $|\partial^2_{\mu}f\bigl(\Theta^i_s\bigr)(X^j_s,X^j_s)|$ are bounded by a constant $C$.
 By the above bounds, we have, for the first integral in~\eqref{eq:integrals-delta-f-1}:
\begin{align*}
	&\left| \EE \int_{t_n}^t \Bigl[\partial_t f\bigl(\Theta^i_s\bigr) +\partial_\alpha f\bigl(\Theta^i_s\bigr)\partial_t\varphi(s,X^i_s)\Bigr]\, ds \right|
	\\
	&\le C \int_{t_n}^t  \EE\left[ |\Theta^i_s|^2 + |(s, \Theta^i_s)| \right] ds
	\\
	&\le C \int_{t_n}^t  \EE\left[ |s|^2 + |X^j_s|^2 + M_2(\mu^N_s)^2 + |\varphi(s, X^j_s)|^2 \right] ds
	\\
	&\le C \Delta t,
\end{align*}
where the $C$ depends only on the data of the problem and on $\varphi$ through its Lipschitz constant and $\varphi(0,0)$. 
The other integral terms in the right hand side of~\eqref{eq:integrals-delta-f-1} can be bounded similarly, except that to bound the term
\begin{align*}
	&\left| \EE\left[  \int_{t_n}^t a(s,X^i_s,\mu^N_s) \partial_\alpha f\bigl(\Theta^i_s\bigr)\partial^2_{xx}\varphi(s,X^i_s) ds \right] \right|
	\\
	&
	\le C \int_{t_n}^t  \EE\left[ \left(|s|^2 + |X^j_s|^2 + M_2(\mu^N_s)^2\right) \left( |s| + |X^j_s| + M_2(\mu^N_s) + |\varphi(s, X^j_s)| \right)  \right]  ds,
\end{align*}
we use stability estimate on the fourth moment.
  Going back to the definition of $\delta^1_f$, we obtain that
$$
	\delta^1_f \le C \Delta t \le C \sqrt{\Delta t}.
$$

\vskip 6pt
As for $\delta^2_f$, it is estimated exactly as we did above in the case of the terminal cost. Indeed:
$$
\delta^2_f=
	\Delta t\; \Bigl|\sum_{n=0}^{N_T-1} 	[F_n(\mu^N_{t_n})-F_n(\check \mu_{t_n})]\Bigr|
$$
where for each $n\in\{0,\ldots,N_T-1\}$,  the function $F_n$ is defined as
$$
F_n(\mu)=\int_{\RR^d}f(t_n,x,\mu,\varphi(t_n,x))\;d \mu(x).
$$
For each $n$, the feedback function $\varphi$ being Lipschitz and (hence) with linear growth, the function $F_n$ has the same properties as the function $G$ above. As a result, one can bound the difference $F_n(\mu^N_{t_n})-F_n(\check \mu_{t_n})$ in the same way, and since the bound is independent of $n$, $\delta^2_f$ admits the same upper bound as $\delta_g$.
This concludes the proof. 
\end{proof}

\bibliographystyle{plain}
 \small
\bibliography{deeplearningMFG-bib}

\end{document}